\documentclass[10pt]{amsart} 

\usepackage{amsmath}
\usepackage[english]{babel} 
\usepackage{enumitem}
\usepackage[utf8]{inputenc} 
\usepackage[a4paper,left=2.5cm,right=2.5cm,bottom=3cm]{geometry} 
\usepackage{indentfirst} 
\usepackage{amsmath,amssymb,amsfonts,amsthm} 
\numberwithin{equation}{section}
\usepackage{bm} 
\usepackage{graphicx} 
\usepackage[export]{adjustbox} 
\usepackage{pdflscape} 
\usepackage{fancyhdr} 
\usepackage{natbib} 
\setcitestyle{numbers,square}
\usepackage{flafter} 
\usepackage[framemethod=tikz]{mdframed} 
\usepackage{color} 
	\definecolor{yellow-green}{rgb}{0.6, 0.8, 0.2}
	\definecolor{viridian}{rgb}{0.25, 0.51, 0.43}
\usepackage{wrapfig} 
\usepackage{lipsum} 
\usepackage{mathrsfs}
\usepackage[all]{xy}
\usepackage{framed}
\usepackage{tikz}
\usepackage{quiver}
\usetikzlibrary{hobby}
\usepackage{thmtools}
\usepackage{thm-restate}
\usepackage{dsfont}
\usepackage{lmodern}
\usepackage{amsmath,amssymb,amsthm,mathtools}
\usepackage{graphicx}

\graphicspath{ {Images/} } 

\usepackage{multicol} 
\usepackage{float} 

\usepackage{nicematrix}
\usetikzlibrary{patterns}
\usetikzlibrary{matrix,decorations.pathreplacing}
\usetikzlibrary{decorations.pathreplacing,calligraphy}
\usepackage[colorlinks=true,linkcolor=cyan,citecolor=viridian]{hyperref} 
\hypersetup{
    pdfcreator={},
    pdfproducer={LaTeX},
    hypertexnames=false,  
    linktocpage=true,
    colorlinks=true
}
\usepackage{prettyref}
\usepackage[nameinlink]{cleveref}
\newtheorem{theorem}{Theorem}[section]
\newtheorem{corollary}[theorem]{Corollary}
\newtheorem{proposition}[theorem]{Proposition}
\newtheorem{lemma}[theorem]{Lemma}
\newtheorem{definition}[theorem]{Definition}
\newtheorem{remark}[theorem]{Remark}

\newtheorem{conj}{Conjecture}

\newtheorem{assumption}{Assumption}[section]
\newrefformat{lemma}{Lemma \ref{#1}}
\newrefformat{coro}{Corollary \ref{#1}}
\newrefformat{thm}{Theorem \ref{#1}}
\newrefformat{prop}{Proposition \ref{#1}}
\newrefformat{rem}{Remark \ref{#1}}
\newrefformat{example}{Example \ref{#1}}
\newrefformat{defn}{Definition \ref{#1}}
\newrefformat{section}{Section \ref{#1}}
\newrefformat{conj}{Conjecture \ref{#1}}
\newrefformat{apdx}{Appendix \ref{#1}}
\newrefformat{fact}{Fact \ref{#1}}
\newrefformat{ass}{Assumption \ref{#1}}
\newcommand{\dd}{{\mathrm d}}
\newcommand{\EE}{{\mathcal{E}}}
\newcommand{\bEE}{{\mathbb{E}}}
\newcommand{\FF}{{\mathcal{F}}}
\newcommand{\KK}{{\mathcal{K}}}
\newcommand{\LL}{{\mathcal{L}}}

\newcommand{\VV}{{\mathcal{V}}}
\newcommand{\WW}{{\mathcal{W}}}
\newcommand{\SO}{{\mathrm{SO}}}

\newcommand{\sslash}{
	\mathchoice{\mathbin{\mkern-3mu/\mkern-6mu/\mkern-3mu}}
	{\mathbin{\mkern-3mu/\mkern-6mu/\mkern-3mu}}
	{\mathbin{\mkern-2mu/\mkern-5mu/\mkern-1mu}}
	{\mathbin{\mkern-2mu/\mkern-5mu/\mkern-1mu}}
}
\newcommand{\iu}{{\sqrt{-1}}}

\newcommand{\R}{\mathbb R}
\newcommand{\D}{\mathbb D}

\newcommand{\CH}{\mathbb{CH}}

\newcommand{\Id}{\mathrm{id}}
\DeclareMathOperator{\Hom}{Hom}
\DeclareMathOperator{\rank}{rank}
\DeclareMathOperator{\tr}{tr}
\DeclareMathOperator{\diag}{diag}

\DeclareMathOperator{\End}{End}
\DeclareMathOperator{\Dev}{Dev}

\newcommand{\PU}{\mathrm{PU}}

\newcommand{\PP}{\mathbb P}
\newcommand{\K}{\mathcal K}

\newcommand{\one}{\mathds{1}}

\newcommand{\wt}{\widetilde}

\newcommand{\C}{\mathbb C}

\newcommand{\id}{\operatorname{id}}
\newcommand{\norm}[1]{\left\lVert #1\right\rVert}

\newcommand{\Lcal}{\mathcal L}
\newcommand{\Tcal}{\mathcal T}
\newcommand{\Vf}{V_f}
\newcommand{\Vcal}{\mathcal V}
\newcommand{\Wcal}{\mathcal W}
\newcommand{\Tp}{T_+}
\newcommand{\Nm}{N_-}
\newcommand{\Np}{N_+}
\newcommand{\Hsc}{\mathscr H}
\newcommand{\Hsczero}{\mathscr H_0}
\newcommand{\dbar}{\bar\partial}
\newcommand{\GR}{G_{\mathbb{R}}}
\newcommand{\PGR}{{\mathrm{P}}G_{\mathbb{R}}}
\newcommand{\KR}{K_{\mathbb{R}}}
\newcommand{\PKR}{{\mathrm{P}}K_{\mathbb{R}}}

\newcommand{\Ecal}{\mathcal{E}}
\newcommand{\Ocal}{\mathcal{O}}
\newcommand{\Fcal}{\mathcal{F}}

\DeclareRobustCommand{\upRoman}[1]{%
  \ifmmode
    \mathrm{\uppercase\expandafter{\romannumeral#1}}%
  \else
    \textup{\uppercase\expandafter{\romannumeral#1}}%
  \fi
}
\usepackage{microtype}

\title{Non-maximal quasi-isometric $\PU_{2,n+1}$-representations via alternating surfaces in complex pseudo-hyperbolic spaces}

\author{Qiongling Li}
\address{Chern Institute of Mathematics and LPMC, Nankai University, Tianjin 300071, China}
\email{qiongling.li@nankai.edu.cn}
\author{Junming Zhang}\address{Chern Institute of Mathematics and LPMC, Nankai University, Tianjin 300071, China}\email{junmingzhang@mail.nankai.edu.cn}

\begin{document}




\begin{abstract}
For every admissible non-maximal Toledo invariant $t$, we construct a locus of irreducible representations of a closed genus-$g$ surface group into $\mathrm{PU}_{2,n+1}$ with Toledo invariant $t$ whose orbit maps are quasi-isometric embeddings. These loci have real codimension $10(g-1)$ or $10(g-1)+2(n-1)$ in the character variety depending on the Toledo value. Our construction uses $4$-cyclic Higgs bundles and their correspondence with $\partial$-alternating surfaces in complex pseudo-hyperbolic spaces. The associated equivariant minimal maps into the symmetric space are bi-Lipschitz embeddings. We further construct cocompact domains of discontinuity in the Shilov boundary. Their quotients are smooth fiber bundles over the surface with fiber homeomorphic to $S^{2n-1}\times S^{2n-1}$.
\end{abstract}
\maketitle
\tableofcontents



\section{Introduction}

The theory of maximal representations shows how an extremal
topological invariant can impose strong geometric constraints on a
surface-group representation. For Hermitian target groups, maximal
representations are discrete and faithful and preserve a maximal
tube-type subdomain \cite{burger2010surface}. In rank two, their
geometry is closely related to special surfaces in pseudo-Riemannian
spaces and to equivariant minimal surfaces in symmetric spaces
\cite{collier2019geometry}. This raises a natural question: which
of these geometric features persist at non-maximal Toledo values?

In this paper, we address this question for \(\PU_{2,n+1}\).
We construct smooth families of non-maximal representations with
quasi-isometric orbit maps, equivariant minimal embeddings into
the associated symmetric space, and cocompact domains of
discontinuity in its Shilov boundary. Our construction uses two
families of \(4\)-cyclic Higgs bundles and their geometric
realization as \(\partial\)-alternating surfaces in complex
pseudo-hyperbolic spaces. The degree data of the Higgs bundles
determine both the dimensions of the representation loci and the
isomorphism classes of the resulting boundary quotient bundles.

\subsection{Non-maximal representations} Let \(S\) be a closed oriented surface of genus \(g\geqslant2\),
let \(n\geqslant1\), and set \(G=\PU_{2,n+1}\). Denote by
\[
\mathfrak X
=
\mathfrak X\bigl(\pi_1(S),G\bigr)
:=
\operatorname{Hom}^{\mathrm{red}}
\bigl(\pi_1(S),G\bigr)/G
\]
the \(G\)-character variety, and by \(\mathcal X_{2,n+1}\) the
associated symmetric space. With the normalization used in this
paper, the Toledo invariant satisfies the Milnor--Wood inequality
\[
\bigl|\operatorname{Tol}(\rho)\bigr|
\leqslant4(g-1).
\]
A representation is called \emph{maximal} when equality holds
and \emph{non-maximal} otherwise; see
\cite{bradlow2003surface,burger2010surface}.

For \(n>1\), the symmetric space \(\mathcal X_{2,n+1}\) is not
of tube type. The tube-type rigidity theorem therefore implies
that every maximal representation preserves a totally geodesic
copy of \(\mathcal X_{2,2}\). More precisely, after conjugation,
\[
\rho\bigl(\pi_1(S)\bigr)
\subset
\mathrm P\bigl(
\mathrm U_{2,2}\times\mathrm U_{n-1}
\bigr)
\subsetneq
\PU_{2,n+1}.
\]
Thus maximality forces the image to preserve a proper
nondegenerate subspace. Our first result constructs irreducible
families outside this rigid maximal locus.

Let
\[
\mathscr T_*
:=
\frac{2\gcd(n+1,2)}{n+3}\mathbb Z
\cap \bigl(-4(g-1),\,4(g-1)\bigr)
\]
be the set of non-maximal Toledo values, and, for each \(t\in\mathscr T_*\), write
\[
\mathfrak X_t
:=
\bigl\{[\rho]\in\mathfrak X
\mid \operatorname{Tol}(\rho)=t\bigr\}.
\]

\begin{theorem}\label{thm:intro-slices}
For every \(t\in\mathscr T_*\), there exist a real smooth manifold \(M_t\) and an
immersion
\[
\iota_t\colon M_t\longrightarrow\mathfrak X_t
\]
whose image \(\mathscr S_t:=\iota_t(M_t)\) is not relatively compact in
\(\mathfrak X_t\) and consists of conjugacy classes of irreducible
representations whose orbit maps are quasi-isometric embeddings.
The manifold \(M_t\) has real dimension
\[
\dim_{\mathbb R} M_t
=
\dim_{\mathbb R}\mathfrak X
-
\begin{cases}
10(g-1), & (n+3)t\in4\mathbb Z,\\
10(g-1)+2(n-1), & \text{otherwise}.
\end{cases}
\]
In particular, all representations in this family are discrete and faithful.
\end{theorem}

The codimension in the first case is independent of \(n\), whereas the
second case can occur only when \(n\) is even. 

Fix a complex structure \(X\) on \(S\). The nonabelian Hodge
correspondence identifies conjugacy classes of reductive
representations into \(G\) with isomorphism classes of
polystable \(G\)-Higgs bundles over \(X\)
\cite{donaldson1987twisted,corlette1988flat,
hitchin1987self,simpson1988constructing}.

Cyclic Higgs bundles have also been used to construct
non-maximal representations with strong geometric properties
in other settings. In particular, the second author proved
that suitable cyclic \(\mathrm{SO}^0_{2,3}\)-Higgs bundles give
rise to non-maximal Anosov representations \cite{zhang2025non},
extending the results and techniques of Filip
\cite{filip2021uniformization}. Our construction draws on
related techniques to study \(\PU_{2,n+1}\)-representations
and their associated minimal surfaces and boundary quotients. There is also a different approach to establish the Anosov property for cyclic Higgs bundles in the case of  $\mathrm{SL}_3\R$ \cite{BD25} and $G_2'$ \cite{DE26}. 

\subsection{Cyclic Higgs bundles and
\texorpdfstring{\(\partial\)}{partial}-alternating surfaces}

For the surface correspondence, we first work on an arbitrary
Riemann surface \(X\), not necessarily compact, and consider
the more general group \(\PU_{m,n+1}\). We study \(4\)-cyclic
harmonic bundles \((\mathcal E,\Phi,H)\) of the form
\begin{equation}\label{eqn:Intro}
(\mathcal E,\Phi)
=
\begin{tikzcd}
{\mathcal L}
&
{\mathcal L K_X^{-1}}
&
{\mathcal V}
&
{\mathcal W}
\arrow["{\mathds{1}}"',from=1-1,to=1-2]
\arrow["\alpha"',from=1-2,to=1-3]
\arrow["\beta"',from=1-3,to=1-4]
\arrow["\gamma"',curve={height=19pt},from=1-4,to=1-1]
\end{tikzcd},
\end{equation}
where each arrow denotes a \(K_X\)-twisted component of the
Higgs field. Set
\[
\mathcal E_0=\mathcal L,\qquad
\mathcal E_1=\mathcal L K_X^{-1},\qquad
\mathcal E_2=\mathcal V,\qquad
\mathcal E_3=\mathcal W.
\]
The ranks and harmonic metric satisfy
\[
\bigl(
\rank(\mathcal E_0),\,
\rank(\mathcal E_1),\,
\rank(\mathcal E_2),\,
\rank(\mathcal E_3)
\bigr)
=
(1,1,n,m-1),
\qquad
H=\bigoplus_{i=0}^{3}H_i,
\]
where \(H_i\) is a Hermitian metric on \(\mathcal E_i\).

To describe the geometry of these bundles, we introduce
\emph{\(\partial\)-alternating surfaces} in complex
pseudo-hyperbolic spaces. Roughly speaking, a
\(\partial\)-alternating surface in \(\CH^{m,n}\) is a
minimal immersion whose associated holomorphic tangent bundle
\(\Vf\) admits an orthogonal splitting
\[
\Vf=\Tp\oplus\Nm\oplus\Np.
\]
Here \(\Tp\) is the positive line generated by
\(\partial^{1,0}f\), \(\Nm\) is negative of rank \(n\), and
\(\Np\) is positive of rank \(m-1\). The defining conditions
constrain the second fundamental form and the remaining
off-diagonal connection terms according to the cyclic
pattern above.

This notion extends the geometry of maximal spacelike
surfaces in real pseudo-hyperbolic spaces: maximal spacelike
surfaces in \(\mathbb H^{2,q}\) studied in \cite{collier2019geometry, LTW} are precisely the
\(\partial\)-alternating surfaces in \(\CH^{2,q}\) whose
images lie in a totally geodesic copy of
\(\mathbb H^{2,q}\); see
Proposition~\ref{prop:real-part}. It is also related to the
general correspondence between cyclic surfaces and cyclic
harmonic bundles developed in \cite{li2026infinitesimal}.
In our setting, this correspondence takes the following form.

\begin{theorem}\label{thm:baseSpecialSignatureintro}
On any Riemann surface \(X\), there is a natural bijection
between:
\begin{enumerate}
\item
\(4\)-cyclic \(\PU_{m,n+1}\)-harmonic bundles as in
\eqref{eqn:Intro}, modulo graded unitary gauge equivalence;

\item
equivariant \(\partial\)-alternating immersions
\(f\colon\widetilde X\to\CH^{m,n}\), modulo
splitting-preserving equivalence.
\end{enumerate}
Under this correspondence, the Gauss map of \(f\) coincides
with the equivariant minimal map
\(h\colon\widetilde X\to\mathcal X_{m,n+1}\)
associated with the harmonic bundle.
\end{theorem}

For the construction of the representation loci, we now
return to complex structures \(X\) on the closed surface
\(S\). We say that a $4$-cyclic Higgs bundle  is of type
\(\upRoman{1}\) (respectively, type \(\upRoman{2}\)) if
\[
\bigl(
\rank(\mathcal E_0),\,
\rank(\mathcal E_1),\,
\rank(\mathcal E_2),\,
\rank(\mathcal E_3)
\bigr)
=
(1,1,n,1)
\quad
\bigl(\text{respectively, }(1,1,1,n)\bigr).
\]
For \(\tau\in\{\upRoman{1},\upRoman{2}\}\) and an
admissible vector \((d_i)=(d_0,d_1,d_2,d_3)\), let
\[
\mathcal M^{\star}_{(d_i),\tau}\longrightarrow\mathcal T(S)
\]
be the joint moduli space whose fiber over \(X\) consists
of polystable non-maximal cyclic Higgs bundles of type \(\tau\) with
\(\deg(\mathcal E_i)=d_i\) for \(i=0,1,2,3\).
Let \(\mathcal M^{\mathrm s,\star}_{(d_j),\tau}\) denote its
stable locus. There is an isomorphism between 
$\mathcal{M}^{\star}_{(d_j),\tau}$ and $\mathcal{M}^{\star}_{(d_j'),\tau}$ when
$(d_j)-(d_j')\in\mathbb{Z}\cdot(1,1,n,m-1)$ given by tensoring a line bundle. Modulo this equivalence, we obtain $\mathcal{M}^{\star}_{[d_j],\tau}$ and $\mathcal M^{\mathrm s,\star}_{[d_j],\tau}$ as the joint moduli space of polystable non-maximal $4$-cyclic $\PU_{m,n+1}$-Higgs bundle of type $\tau$ with charateristic $[d_j]\in\mathbb{Z}^4/\mathbb{Z}\cdot(1,1,n,m-1)$ and its stable locus respectively.

The fiberwise nonabelian Hodge correspondence
defines a map
\[
\operatorname{NAH}\colon
\mathcal M^{\star}_{[d_j],\tau}\longrightarrow\mathfrak X,
\]
and we write
\[
\mathcal R_{[d_j],\tau}
:=
\operatorname{NAH}
\bigl(\mathcal M^{\star}_{[d_j],\tau}\bigr),
\qquad
\mathcal R^{\mathrm{irr}}_{[d_j],\tau}
:=
\operatorname{NAH}
\bigl(\mathcal M^{\mathrm s,*}_{[d_j],\tau}\bigr).
\]

\subsection{Minimal embeddings and domains of discontinuity}

The four-step grading implies
\(\operatorname{tr}(\Phi^2)=0\), so the associated
equivariant harmonic maps into \(\mathcal X_{2,n+1}\)
are conformal. The \(\partial\)-alternating structure
provides the additional geometric control needed to
establish global embedding and distance estimates.

\begin{theorem}\label{thm:intro-geometry}
Let \((X,[\mathcal E,\Phi])\in\mathcal M^{\star}_{[d_j],\tau}\),
let \([\rho]=\operatorname{NAH}(X,\mathcal E,\Phi)\), and
let \(h_\rho\colon\widetilde X\to\mathcal X_{2,n+1}\)
be the associated equivariant minimal map. Then:
\begin{enumerate}
\item
The map \(h_\rho\) is a bi-Lipschitz embedding with
respect to the hyperbolic metric in the conformal class
of \(X\), and its induced metric has strictly negative
Gaussian curvature.

\item
For one, and hence every, \(o\in\mathcal X_{2,n+1}\),
the orbit map
\[
\pi_1(S)\longrightarrow\mathcal X_{2,n+1},
\qquad
\gamma\longmapsto\rho(\gamma)o,
\]
is a quasi-isometric embedding. In particular, \(\rho\)
is discrete and faithful.

\item
The group \(\rho(\pi_1(S))\) acts freely and properly
discontinuously on \(\mathcal X_{2,n+1}\). The quotient
\[\rho(\pi_1(S))\backslash\mathcal X_{2,n+1}\] is the
total space of a smooth fiber bundle over \(S\) with
fiber homeomorphic to \(\mathbb R^{4n+2}\).
Moreover, \(h_\rho\) descends to an embedded minimal
surface in this locally symmetric quotient.
\end{enumerate}
\end{theorem}

\begin{remark}
The construction includes a real subfamily. Under the
standard inclusion
\(\SO_{2,n+1}\hookrightarrow\mathrm U_{2,n+1}\),
followed by projectivization, maximal
\(\SO_{2,n+1}\)-representations have vanishing unitary
Toledo invariant and occur in suitable loci
\(\mathcal R_{[d_j],\upRoman{1}}\). Thus maximality in
the real orthogonal subgroup is compatible with
non-maximality in the ambient unitary group.
\end{remark}

The same geometric structure produces domains of
discontinuity in the Shilov boundary \(\mathcal S_0\)
of \(\mathcal X_{2,n+1}\), identified with the space
of maximal totally isotropic complex \(2\)-planes in
\(\mathbb C^{n+3}\).

\begin{theorem}\label{thm:intro-domains}
For every \([\rho]\in\mathcal R_{[d_j],\tau}\), there
exists a \(\rho(\pi_1(S))\)-invariant open subset
\(\Omega\subset\mathcal S_0\) on which
\(\rho(\pi_1(S))\) acts freely, properly discontinuously,
and cocompactly.

The quotient \(\rho(\pi_1(S))\backslash\Omega\) is
the total space of a smooth fiber bundle over \(S\)
with fiber homeomorphic to
\(S^{2n-1}\times S^{2n-1}\). For fixed \(\tau\),
the isomorphism class of this fiber bundle is determined
by the degree data \([d_i]\).
\end{theorem}

Fiber-bundle structures on quotients of domains of
discontinuity have been studied for certain deformations
through Anosov representations
\cite{alessandrini2025fiber}. Geometric constructions
using quasi-geodesic immersions provide another approach
to such fibrations \cite{davalo2025nearly}. Our construction
identifies the fiber explicitly and relates the isomorphism
class of the quotient bundle to the degree data of the
Higgs bundle.

Theorems~\ref{thm:intro-geometry}
and~\ref{thm:intro-domains} require no Anosov hypothesis.
When \(n=1\), the representations under consideration
are known to be Anosov by the results discussed in
\cite{burger2010surface,zhang2025non}; see
\prettyref{rem:q=1}. For \(n>1\), we do not establish
their Anosov property.

\subsection{Degree strata and dimensions}

We now describe the degree strata underlying the
construction. The degrees satisfy
\(d_1=d_0-(2g-2)\). Let
\[
\mu=\frac{d_0+d_1+d_2+d_3}{n+3}
\]
be the slope of \(\mathcal E\), and define
\[
\begin{aligned}
u_{\upRoman{1}}
&=d_1-d_3,
&
v_{\upRoman{1}}
&=d_0-d_2+(n-1)\mu,
\\
u_{\upRoman{2}}
&=d_2-d_0,
&
v_{\upRoman{2}}
&=d_3-d_1-(n-1)\mu.
\end{aligned}
\]
These quantities are unchanged by tensoring all
summands with a common line bundle. Note that $|\operatorname{Tol}(\rho)|=|v_{\tau}-u_{\tau}-4(g-1)|.$

Let \(\mathscr D_\tau\) be the set of admissible degree
vectors satisfying
\[
-4(g-1)
\leqslant u_\tau
\leqslant v_\tau
\leqslant4(g-1),
\]
and let \(\mathscr D_\tau^{\mathrm s}\) be the subset
satisfying
\[
\begin{cases}
-4(g-1)
\leqslant u_\tau
\leq v_\tau
\leqslant4(g-1),
& n=1,\\[1mm]
-4(g-1)
\leqslant u_\tau
< v_\tau
<4(g-1),
& n>1.
\end{cases}
\]
Let $D_\tau^{\star}=D_\tau\cap \{0<v_\tau-u_\tau<8(g-1)\}, D_\tau^{s, \star}=D_\tau^s\cap \{0<v_\tau-u_\tau<8(g-1)\}.$

\begin{theorem}\label{thm:intro-moduli}
Fix \(\tau\in\{\upRoman{1},\upRoman{2}\}\) and an
admissible degree vector \((d_i)\). Then:
\begin{enumerate}
\item
The moduli space \(\mathcal M^{\star}_{(d_j),\tau}\) is
nonempty if and only if \((d_j)\in\mathscr D_\tau^\star\).

\item
The stable locus \(\mathcal M^{\mathrm s,\star}_{(d_j),\tau}\)
is nonempty if and only if
\((d_j)\in\mathscr D_\tau^{\mathrm s,\star}\).

\item When nonempty, $\mathcal M^{s,\star}_{[d_j],\tau}$
is a smooth real manifold of dimension
\[
N_\tau=(n^2+2n+7)(2g-2)+2(1-n)u_\tau.
\]
The restriction of the nonabelian Hodge map
\[
\mathrm{NAH}\colon
\mathcal M^{s,\star}_{[d_j],\tau}\longrightarrow\mathcal X
\]
is a smooth immersion with image
$R^{\mathrm{irr}}_{[d_j],\tau}$.
\end{enumerate}
\end{theorem}

The dimension statement allows the complex structure
on \(S\) to vary and describes the images of the joint
moduli spaces in the character variety, not merely
the Higgs-bundle moduli spaces over a fixed Riemann
surface. Thus the degree data control both the
dimensions of the representation loci and the topology
of the boundary quotients, while the surface
correspondence governs their metric geometry.

We conclude with two conjectures. The first concerns
the \(\partial\)-alternating immersions themselves,
rather than their Gauss maps; the second strengthens
the quasi-isometric embedding property of the
associated representations.

\begin{conj}
\begin{enumerate}
\item
Every complete \(\partial\)-alternating immersion into
\(\CH^{2,n}\) or \(\CH^{n+1,1}\) is an embedding.

\item
Every representation in \(\mathcal R_{[d_j],\tau}\)
is Anosov.
\end{enumerate}
\end{conj}

\subsection*{Organization}
In Section \ref{sec:cyclicHitchinequation}, we review $4$-cyclic harmonic bundles and analyze the associated Hitchin equations. In Section \ref{sec:QuasiIsometry}, we show the quasi-isometric property of associated representatiosn and the bi-Lipschitz property of the associated minimal map. In Section \ref{sec:GeometricStructure}, we construct domain of discontinuity for $4$-cyclic harmonic bundles. In Section \ref{sec:AlternatingSurface}, we study $\partial$-alternating surfaces in $\CH^{p,q}$ and show their correspondence with $4$-cyclic harmonic bundles. In Section \ref{sec:Moduli}, we study the moduli space of $4$-cyclic Higgs bundles and the associated sublocus in the character variety. 

\subsection*{Declaration of AI use}
The small-scale estimate in \prettyref{lemma:exponentialgrowth} is suggested by Gemini 3.1 Pro. We also use ChatGPT 5.5 Pro for filling  the proof of \prettyref{prop:real-part} and check the diffeomorphic type of the fibre in \prettyref{prop:FiberBundle}. ChatGPT 5.6 Sol is used to improve the language and format of Abstract and Introduction.
 
The main text and the appendices of the manuscript were composed by the authors. All mathematical statements and cited references in the manuscript were independently checked by the authors. The authors take full responsibility for the content of the manuscript.

\subsection*{Acknowledgement}
J. Zhang would like to thank Brian Collier and Filippo Mazzoli for lots of helpful discussions on maximal spacelike immersions into the pesudo-hyperbolic space $\mathbb{H}^{2,n}$.

Both authors are partially supported by the National Key R\&D
Program of China No. 2022YFA1006600, the Fundamental Research Funds for the
Central Universities, and Nankai Zhide foundation. Q. Li is sponsored by the
Alexander von Humboldt Foundation. J. Zhang is supported by NSF of China grant No. 125B2007.

\section{\texorpdfstring{$4$}{4}-cyclic harmonic bundles}\label{sec:cyclicHitchinequation}

\subsection{\texorpdfstring{$4$}{4}-cyclic harmonic bundles}
\label{sec:hbundle}
\begin{definition}
 A \emph{$4$-cyclic $\mathrm{U}_{m,n+1}$-Higgs bundle} $(\Ecal,\Phi)$ on $X$ is of the form 
 \begin{equation}\label{eq:fourcycle}(\EE,\Phi)=\begin{tikzcd}
	{\mathcal{L}} & {\mathcal{L}\mathcal{K}_X^{-1}} & {\mathcal{V}} & {\mathcal{W}}
	\arrow["{\mathds{1}}"', from=1-1, to=1-2]
	\arrow["\alpha"',  from=1-2, to=1-3]
	\arrow["\beta"',  from=1-3, to=1-4]
	\arrow["\gamma"', curve={height=19pt}, from=1-4, to=1-1]
\end{tikzcd}
\end{equation}
where
\[
  \mathds{1}\in
  H^{0}\bigl(\Hom(\Lcal,\Lcal\KK_X^{-1})\otimes\KK_X\bigr)
  =
  H^{0}(\Ocal),
  \]
\[
  \alpha\in
  H^{0}\bigl(\Hom(\Lcal\KK_X^{-1},\Vcal)\otimes\KK_X\bigr),
  \qquad
  \beta\in
  H^{0}\bigl(\Hom(\Vcal,\Wcal)\otimes \KK_X\bigr),\qquad
  \gamma\in
  H^{0}\bigl(\Hom(\Wcal,\Lcal)\otimes \KK_X\bigr).
\]
The ranks of the summands are
\[
  \rank(\Lcal)=\rank(\Lcal\KK_X^{-1})=1,\qquad
  \rank(\Vcal)=n,\qquad
  \rank(\Wcal)=m-1.
\]
\end{definition}

\begin{definition}\label{defn:harmonic}
Fix a K\"ahler form $\omega$ on $X$ with $\int_X\omega=1$. A \emph{$4$-cyclic $\mathrm{U}_{m,n+1}$-harmonic bundle} on $X$ is a triple $(\Ecal,\Phi, h)$ where $(\Ecal, \Phi)$ is a $4$-cyclic
$\mathrm{U}_{m,n+1}$-Higgs bundle and $h$ is an Hermitian metric satisfying 
\[
  h=h_\Lcal\oplus h_{\Lcal K^{-1}}\oplus h_\Vcal\oplus h_\Wcal
\]
and the Hitchin equation \[F(\nabla^h)+[\Phi,\Phi^{*_h}]=-2\pi\iu\mu\cdot\omega\otimes\id_\EE.\]
\end{definition}

The center of $\mathrm{U}_{m,n+1}$ is the diagonal copy of $\mathrm{U}_1$. Hence
tensoring by a holomorphic line bundle $\Fcal$ acts on the Higgs bundle
\[
  (\Ecal,\Phi)
  \longmapsto
  (\Ecal\otimes\Fcal,\Phi).
\]
Under this operation, all four cyclic summands are tensored by the
same line bundle $\Fcal$. Therefore the ratios
\[
  \Lcal \KK_X^{-1}\otimes\Lcal^{-1}\cong \KK_X^{-1},
  \qquad
  \Vcal\otimes\Lcal^{-1},
  \qquad
  \Wcal\otimes\Lcal^{-1}
\]
are invariant. The spaces in which the arrows
$\mathds 1,\alpha,\beta,\gamma$ live depend only on these ratios.
We define \emph{$4$-cyclic $\PU_{m,n+1}$-Higgs bundle} on $X$ as an equivalence class of $4$-cyclic $\mathrm{U}_{m,n+1}$-Higgs bundles modulo tensoring by holomorphic line bundles.

If we restrict to choose a holomorphic line bundle $\Fcal$ equipped with a projectively flat Hermitian metric $h_0$, 
tensoring by a $(\Fcal, h_0)$ acts as:
\[
  (\Ecal,\Phi,h)
  \longmapsto
  (\Ecal\otimes\Fcal,\Phi,h\otimes h_0).
\] We define a $4$-cyclic
$\PU_{m,n+1}$-harmonic bundle as an equivalence class of
$4$-cyclic $\mathrm{U}_{m,n+1}$-harmonic bundles modulo tensoring by projectively flat line bundles.

\subsection{Estimate of Hitchin equation}
Let \(z=x+iy\) be a local holomorphic coordinate on \(X\). We write
\(g=\tilde g\,|dz|^2\), i.e.\ the real Riemannian metric
\(\tilde g(dx^2+dy^2)\). Its complex bilinear extension satisfies
\(g_{\mathbb C}(\partial_z,\partial_{\bar z})=\tilde g/2\). The associated
Hermitian metric on \(T^{1,0}X\) is normalized by
\(h_{T^{1,0}X}(\partial_z,\partial_z)=\tilde g\); equivalently,
\(h_{T^{1,0}X}(\xi,\eta)=2g_{\mathbb C}(\xi,\overline{\eta})\). Hence
\(|dz|_g^2=\tilde g^{-1}\).

If \((E,h)\) is Hermitian and \(s\in\Gamma(\End E)\), set
\(|s|_h^2=\tr(ss^{*_h})\). For
\(\sigma=s\,dz^{\otimes p}\otimes d\bar z^{\otimes q}\), set
$\|\sigma\|_{h,g}^2=|s|_h^2\tilde g^{-(p+q)}.$
When \(h\) is clear, we write \(\|\sigma\|_g\). If
\(\alpha=s\,dz\in\Omega^{1,0}(\End(E))\), we use
\[
  |\alpha|_h^2:=|s|_h^2\cdot |dz|^2
\]
to denote the induced possibly-degenerate conformal metric.

We use
\[
\omega_g=i\tilde g\,dz\wedge d\bar z=2\,\mathrm{vol}_g,\qquad
\Delta_g=\tilde g^{-1}\partial_z\partial_{\bar z},\qquad
K_g=-2\tilde g^{-1}\partial_z\partial_{\bar z}\log\tilde g .
\]
The contraction \(\Lambda_g\) is defined by
\(\eta=(\Lambda_g\eta)\omega_g\); equivalently, if
\(\eta=a\,dz\wedge d\bar z\), then $\Lambda_g\eta=\frac{a}{i\tilde g}.$
For a Hermitian holomorphic line bundle \((L,h)\), if \(h=e^\varphi\), then
$F_h=\bar\partial\partial\varphi.$ In particular, for the metric
$h_{T^{1,0}X}$ on $\KK_X^{-1}$,
\[
F_{\KK_X^{-1}}=\bar\partial\partial\log\tilde g,\qquad
i\Lambda_gF_{\KK_X^{-1}}=K_g/2.
\]
With this convention, if \(\alpha=s\,dz\in\Omega^{1,0}(\End(E))\),
\[i\Lambda_g\tr(\alpha\wedge\alpha^{*_h})=\|\alpha\|_{h,g}^2.\]

We recall the following formula as a slight modification of \cite[Lemma 2.4]{mochizuki2016asymptotic}.
\begin{lemma}[{\protect\cite[Lemma 3.3]{dai2024bounded}}]
\label{lemma:EquationSection}

Let $(\mathcal E, \Phi, H)$ be a harmonic bundle. Locally, $\Phi=fdz$ for a local holomorphic section $f$ of $\End(\mathcal E)$. 
For a holomorphic section $s$ of $\End(\mathcal E)$, we have
\[\partial_z\partial_{\bar z}\log |s|_H^2=\frac{|[s, f^{*_H}]|_H^2-|[s,f]|_H^2}{|s|_H^2}+\frac{|s|_H^2|\partial_z s|_H^2-|\langle\partial_z s, s\rangle|^2}{|s|_H^4}\geqslant\frac{|[s, f^{*_H}]|_H^2-|[s,f]|_H^2}{|s|_H^2}.\]
\end{lemma}
As an immediate corollary, we obtain
\begin{lemma}\label{lemma:EquationSectionCurvature}Let $(\mathcal E, \Phi, H)$ be a harmonic bundle and $g$ be background Hermitian metric.
For a holomorphic section $s$ of $\End(\mathcal E)\otimes \mathcal K_X$, we have 
\[\Delta_g\log \|s\|_{H,g}^2\geqslant \frac{\|[s, \Phi^{*_H}]\|_{H,g}^2-\|[s,\Phi]\|_{H,g}^2}{\|s\|_{H,g}^2}+\frac{1}{2}K_g.\]
\end{lemma}

Let $(\EE,\Phi, H)$ be a $4$-cyclic harmonic bundle. Hitchin's equation
\[F_H + [\Phi,\Phi^{*_H}] =\mu\Id_E= -ic\omega_g\Id_{\EE}\]
becomes
\begin{equation}\label{eqn:HitchinEq}
\begin{aligned}
  i\Lambda_g F_{\mathcal L}
    + i\Lambda_g(\mathds 1^{*_H}\wedge \mathds 1)
    + i\Lambda_g\bigl(\gamma\wedge \gamma^{*_H}\bigr)
    &= c,
 \\
  i\Lambda_g F_{\mathcal L\mathcal K_X^{-1}}
    + i\Lambda_g\bigl(\alpha^{*_H}\wedge\alpha\bigr)+i\Lambda_g(\mathds 1\wedge\mathds 1^{*_H})
    &= c,
 \\
  i\Lambda_g F_{\mathcal V}
    +i\Lambda_g(\beta^{*_H}\wedge \beta)
    + i\Lambda_g\bigl(\alpha\wedge\alpha^{*_H}\bigr)
    &= c,\\
      i\Lambda_g F_{\mathcal W}
    + i\Lambda_g(\gamma^{*_H}\wedge\mathds \gamma)
    + i\Lambda_g\bigl(\beta\wedge\beta^{*_H}\bigr)
    &= c.
\end{aligned}\end{equation}
Moreover, let $g$ be the Hermitian metric on $\KK_X^{-1}$ defined by ${H|_{\LL}}^{-1}\cdot H|_{\LL\KK_X^{-1}}$.

By the definition of $g$ and Equation (\ref{eqn:HitchinEq}), we have the Gaussian curvature \begin{eqnarray*}
  K_g=&2i \Lambda_g F(g)=-2i\Lambda_g F_{\mathcal{\LL}}+2i\Lambda_g F_{\LL\KK_X^{-1}}\\
  =&2(\|\gamma\|^2+\|\alpha\|^2)-4.\end{eqnarray*}

\begin{lemma}\label{lemma:higgsfieldnorm}
 Suppose $g$ is complete. We have $\|\beta\|^2\leqslant \min\{\rank(\Vcal),\rank(\Wcal)\}.$
If the equality holds at one point, it holds at every point. 

When $\min\{\operatorname{rank}(\mathcal V),
\operatorname{rank}(\mathcal W)\}=1$, equality implies that
the harmonic bundle splits as the direct sum of a cyclic
$\mathrm{U}_{2,2}$-harmonic bundle, on which the induced $\beta$-arrow
is an isomorphism, and a $\mathrm{U}_{q-1}$-harmonic bundle.
If $X$ is closed, the original $\mathrm{U}_{2,q+1}$-harmonic bundle
is maximal.
\end{lemma}
\begin{proof}

Let $s_\beta$ be the section of $\End(\mathcal E)\otimes \mathcal K_X$ formed by $\beta$. From Lemma \ref{lemma:EquationSectionCurvature}, we have
\begin{eqnarray*}
\Delta_g\log\|s_\beta\|^2&\geqslant&\frac{\|[s_\beta,\Phi^{*_H}]\|^2-\|[s_\beta,\Phi]\|^2}{\|s_\beta\|^2}+\frac{1}{2}K_g\\&=&\frac{2\|\beta^*\beta\|^2-\|\beta\alpha\|^2-\|\gamma\beta\|^2}{\|\beta\|^2}+\frac{1}{2}K_g\\
&\geqslant &\frac{2\|\beta\|^2}{\min\{\rank(\Vcal),\rank(\Wcal)\}}-\|\alpha\|^2-\|\gamma\|^2+\frac{1}{2}K_g\\
&= &\frac{2\|\beta\|^2}{\min\{\rank(\Vcal),\rank(\Wcal)\}}-2.
\end{eqnarray*}

From the Cheng-Yau maximum principle, we have $\|\beta\|^2\leq \min\{\rank(\Vcal),\rank(\Wcal)\}.$ Moreover, by the strong maximum principle, either $\|\beta\|^2<\min\{\rank(\Vcal),\rank(\Wcal)\}$ or $\|\beta\|^2\equiv\min\{\rank(\Vcal),\rank(\Wcal)\}.$ 

For the latter case, using Lemma \ref{lemma:EquationSection}, we have $\beta$ has no zeros, and 
\[\|\beta\alpha\|^2=\|\beta\|^2\cdot \|\alpha\|^2, \quad \|\gamma\beta\|^2=\|\beta\|^2\cdot \|\gamma\|^2,
\]
\[\|\langle\partial_z\beta,\beta\rangle\|^2=\|\partial_z\beta\|^2\cdot \|\beta\|^2.\]
If $\rank(\Vcal)=1$, take $N$ the image of $\beta$ in $\Wcal$ viewing as an injective map $\beta: \mathcal V\mathcal K_X^{-1}\rightarrow \mathcal W$. Thus $N\cong \mathcal V\mathcal K_X^{-1}$. Take $N^{\perp}$ the $H$-orthogonal bundle of $N$ inside $\mathcal W$.

With respect to the decomposition $\mathcal W=N\oplus N^{\perp}$, $\gamma=(\gamma_1, \gamma_2)$. Thus $\gamma\beta=\gamma_1\beta.$ From $\|\gamma\beta\|^2=\|\beta\|^2\cdot \|\gamma\|^2$, we have $\|\gamma_1\|^2=\|\gamma\|^2$ and thus $\gamma_2=0$. 

With respect to the decomposition $\mathcal W=N\oplus N^{\perp}$, the holomorphic structure on $\mathcal W$ is written as
\[\bar\partial_W=\begin{pmatrix}\bar\partial_N&\eta\\0&\bar\partial_{N^{\perp}}\end{pmatrix},\]
where $\bar\partial_N, \bar\partial_{N^{\perp}}$ are the induced holomorphic structure on $N, N^{\perp}$, and $\eta\in \Omega^{0,1}(\Hom(N^{\perp}, N))$. The $(1,0)$-part of the Chern connection is 
\[\partial_W=\begin{pmatrix}\partial_N&0\\-\eta^{*_H}&\partial_{N^{\perp}}\end{pmatrix}.\]
Thus $\partial_W\beta=(\partial_N\beta, -\eta^{*_H}\beta)$. From $\|\langle\partial\beta,\beta\rangle\|^2=\|\partial\beta\|^2\cdot \|\beta\|^2$, we have $\eta^{*_H}\beta=0$ and thus $\eta\equiv0$ from $\beta$ being scalar. 
Therefore, $\mathcal W=N\oplus N^{\perp}$ is also a holomorphic splitting. The harmonic bundle splits as 
\begin{equation}
  \Big(\begin{tikzcd}
	{\mathcal{L}} & {\mathcal{L}\mathcal{K}_X^{-1}} & {\mathcal{V}} & {\mathcal{V}\KK_X^{-1}}
	\arrow["{\mathds{1}}"', from=1-1, to=1-2]
	\arrow["\alpha"',  from=1-2, to=1-3]
	\arrow["\beta={\mathds{1}}"',  from=1-3, to=1-4]
	\arrow["\gamma"', curve={height=19pt}, from=1-4, to=1-1]
\end{tikzcd}\Big)\oplus (N^{\perp},0, H|_{N^{\perp}}).
\end{equation}
Here, the first harmonic bundle is a maximal $\mathrm{U}_{2,2}$-harmonic bundle when $X$ is closed with harmonic metric $H|_{\mathcal L}\oplus H|_{\Lcal\mathcal K^{-1}}\oplus H|_{\mathcal V}\oplus \mathcal H|_{N}$ satisfying $H|_{\mathcal L}^{-1}H|_{\Lcal\mathcal K^{-1}}= H|_{\mathcal V}^{-1}H|_{N}.$

For the case $\rank(\WW)=1$, the proof is similar by considering $\beta\alpha$ instead of $\gamma\beta$.
\end{proof}

\begin{lemma}Assume $\min\{\rank(\Vcal),\rank(\Wcal)\}=1$. If $g$ is complete, then $K_g<0$ or $K_g\equiv 0$. In the latter case, $X$ is not hyperbolic and $q_4=\tr(\Phi^4)$ is nowhere vanishing. As a corollary, if $(\EE,\Phi)$ is not in the latter case, 
\[[\Phi,\Phi^{*_H}]\neq 0.\]
\end{lemma}
\begin{proof}
  Let $\delta$ be the section of $\End(\mathcal E)\otimes \mathcal \KK_X$ given by $\alpha$ and $\gamma$. Thus $\|\delta\|^2=\|\alpha\|^2+\|\gamma\|^2.$ From Lemma \ref{lemma:EquationSectionCurvature}, we have
\begin{eqnarray*}
\Delta_g\log\|\delta\|^2&\geq&\frac{\|[\delta,\Phi^{*_H}]\|^2-\|[\delta,\Phi]\|^2}{\|\delta\|^2}+\frac{1}{2}K_g\\&=&\frac{2\|\alpha^*\alpha\|^2+2\|\gamma^*\gamma\|^2-\|\beta\circ\alpha\|^2-\|\gamma\circ\beta\|^2}{\|\alpha\|^2+\|\gamma\|^2}-1+\frac{1}{2}K_g\\
&\geq&(\|\alpha\|^2+\|\gamma\|^2)-\|\beta\|^2-1+\frac{1}{2}K_g\\
&\geqslant&2(\|\alpha\|^2+\|\gamma\|^2)-4.
\end{eqnarray*}
In the third inequality, we make use of the Cauchy-Schwarz inequality $\|AB\|^2=\tr(B^*A^*AB)=\tr(A^*ABB^*)\leq \|A^*A\|\cdot\|BB^*\|$ for two $n\times n$ matrices $A$ and $B$; and the fact that $|AA^*|^2=|A|^4$ for an $1\times n$ matrix $A$. In the fourth inequality, we make use of Lemma \ref{lemma:higgsfieldnorm} that $\|\beta\|\leq 1$.

Applying the Cheng-Yau maximum principle, we have $\|\delta\|^2\leq 2$. Moreover, by the strong maximum principle, either $\|\delta\|^2<2$ or $\|\delta\|^2\equiv 2$. While in the latter case, we have $\|\alpha\|=\|\gamma\|=\|\beta\|\equiv 1$ in which case the harmonic bundle splits as a $\mathrm{U}_{2,2}$-harmonic bundle and a $\mathrm{U}_{q-1}$-harmonic bundle. So $q_4=\tr(\Phi^4)$ is nowhere vanishing and $g$ is complete and flat, which forces the base Riemann surface $X$ is either $\mathbb C, \mathbb C^*$ or $\mathbb T^2$.

The last claim on $[\Phi,\Phi^{*_H}]$ follows from Hitchin equation.
\end{proof}

\begin{lemma}\label{lem:BoundedAbove}
Suppose $X$ is equipped with a complete hyperbolic metric $g_X$. For a $4$-cyclic harmonic bundle $(\EE,\Phi,H)$, if $q_4=\tr(\Phi^4)$ is bounded with respect to $g_X$, then $\|\Phi\|_{h,g_X}$ is bounded.
\end{lemma}
\begin{proof}
From \cite[Proposition 3.12]{LiMochizuki1}, if $\tr(\Phi^i)$ is bounded with respect to $g_X$ for $i=1,\cdots,\operatorname{rank}(\EE)$, then $\|\Phi\|_{h,g_X}^2$ is bounded from above. In particular, for $4$-cyclic Higgs bundles,   $\tr(\Phi^i)$ is bounded with respect to $g_X$ for $i=1,\cdots,\operatorname{rank}(\EE)$ if and only if $q_4=\tr(\Phi^4)$ is bounded with respect to $g_X$. The claim then follows.
\end{proof}

Note that from the $4$-cyclic  harmonic bundle $(\EE,\Phi,H)$ over $X$, we obtain a pair $(\rho, f)$ where $\rho:\pi_1(S)\rightarrow G$ is a representation and $f:\widetilde X\rightarrow\mathcal X_{2,q+1}$ is a $\rho$-equivariant harmonic map. Here, we additionally have $\tr(\Phi^2)=0$ from the $4$-cyclic property, and thus $f$ is conformally minimal.

We conclude some easy consequences of the equivariant minimal surfaces in $\mathcal X_{2,q+1}$ associated to such harmonic bundles. In later Theorem \ref{thm:embedding}, we will improve the conclusion on $f$.  
\begin{proposition}
Suppose $X$ is equipped with a complete hyperbolic metric $g_X$ and $\|q_4\|_{g_X}$ is bounded. Assume that $\|\mathds{1}\|_{H}^2$ is complete.
Let $f$ be the associated equivariant conformal harmonic (minimal) map from $\widetilde X$ to $\mathcal X_{2,q+1}$ associated to $(\EE,\Phi,H)$. Then $f$ is a Lipschitz immersion and the image of $f$ is strictly negatively curved. 
\end{proposition} 
\begin{proof}
Since $f$ is conformal, the induced metric of $f$ is $C|\Phi|_H^2$, (see \cite[Section 5]{Li2019HiggsBundles}), for some positive constant $C$ arising from normalization of the metric on $\mathcal X_{2,q+1}$. Since $\Phi$ is nowhere to disappear, the immersion property is obvious. 

The Lipschitz property follows directly from the lemma \ref{lem:BoundedAbove}.

It is known that the curvature of the image of $f$ is at most (e.g. see \cite[Section 5]{Li2019HiggsBundles}) 
\[-C'\cdot \frac{\|[\Phi,\Phi^{*_H}]\|^2}{\|\Phi\|^4}\] for some positive constant $C'$ arising from normalization of the metric on $\mathcal X_{2,q+1}$.  
\end{proof}

\subsection{Duality between the two types of harmonic bundles}\label{subsec:duality}
We focus on  \[(\EE,\Phi)=\begin{tikzcd}
	{\mathcal{L}} & {\mathcal{L}\mathcal{K}_X^{-1}} & {\mathcal{V}} & {\mathcal{W}}
	\arrow["{\mathds{1}}"', from=1-1, to=1-2]
	\arrow["\alpha"',  from=1-2, to=1-3]
	\arrow["\beta"', from=1-3, to=1-4]
	\arrow["\gamma"', curve={height=19pt}, from=1-4, to=1-1]
\end{tikzcd}\]
where either $\rank(\Vcal)=1$ or $\rank(\Wcal)=1$. If $\rank(\Wcal)=1$, we call it is of \textbf{type I}. If $\rank(\Vcal)=1$, we call it is of \textbf{type II}.

Let $(\Ecal,\Phi)$
be a Higgs bundle with associated representation
$\rho:\pi_1(X)\to\PU_{2,q+1}.$
The dual Higgs bundle is $(\Ecal^\vee,\Phi^\vee),$
where
$\Ecal^\vee=\operatorname{Hom}(\Ecal,\mathcal O_X)$ and $\Phi^\vee$ is the natural dual of $\Phi$.
Equivalently, for $\alpha\in \Ecal^\vee$ and $v\in \Ecal$,
$(\theta^\vee\alpha)(v)=-\alpha(\theta v).$

The corresponding dual representation is the contragredient representation
\[
\rho^\vee:\pi_1(X)\to\PU_{2,q+1},
\]
given by $\rho^\vee(\gamma)
=
\left(\rho(\gamma)^{-1}\right)^{\mathrm T}.$

Let $f_\rho:\widetilde X\to \operatorname{Sym}(\PGR)$
be the $\rho$-equivariant harmonic map associated to $(E,\Phi)$. Define the duality isometry
\[
\mathcal D:\operatorname{Sym}(\PGR)\to \operatorname{Sym}(\PGR)
\]
by $\mathcal D(g\cdot\PKR)
=
(g^{-1})^{\mathrm T}\cdot\PKR.$
Then the harmonic map associated to the dual representation is
$f_{\rho^\vee}
=
\mathcal D\circ f_\rho.$

Indeed, if
$f_\rho(\gamma z)=\rho(\gamma)\cdot f_\rho(z),$
then
$f_{\rho^\vee}(\gamma z)
=
\rho^\vee(\gamma)\cdot f_{\rho^\vee}(z).$
Since \(\mathcal D\) is an isometry of the symmetric space, \(f_{\rho^\vee}\) is harmonic whenever \(f_\rho\) is harmonic.

For cyclic Higgs bundles $(\mathcal E, \Phi)$ of type \upRoman{2} with associated representation $\rho$ and harmonic map $f$, we consider its dual Higgs bundle 
\[    (\Ecal^\vee,\Phi^\vee)=\begin{tikzcd}
	{ (\Lcal \KK_X^{-1})^\vee} & {\mathcal{L}^\vee} & {\mathcal{W}^\vee} & {\mathcal{V}^\vee}
	\arrow["{-\mathds{1}}^\vee"',  from=1-1, to=1-2]
	\arrow["-\gamma^\vee"',  from=1-2, to=1-3]
	\arrow["-\beta^\vee"',  from=1-3, to=1-4]
	\arrow["-\alpha^\vee"', curve={height=19pt}, from=1-4, to=1-1]
\end{tikzcd}\]
with harmonic metric also dual to each other. Thus we obtain a cyclic Higgs bundle of type \upRoman{1}. The associated representation and harmonic map is $(\rho^\vee, f^{\vee})$.  

\section{Quasi-isometric property}
\label{sec:QuasiIsometry}
We prove that the equivariant $\partial$-alternating surfaces in $\mathbb{CH}^{2,q}$ and in $\mathbb{CH}^{q+1,1}$ give bi-Lipschitz embedded minimal surfaces in the symmetric space of $\PGR=\PU_{2,q+1}$. Let $(\EE,\Phi)$ be the corresponding $\GR=\mathrm{U}_{2,q+1}$-Higgs bundle. 

\subsection{Index Estimates}

In this section we give some estimates for a general class of functions. The estimates can be used to establish the quasi-isometric embedding property for Higgs bundles by choosing suitable functions. The main idea of the estimates follows from \cite[Section 2.2]{filip2021uniformization} but with improvements. One can also compare the estimates with \cite[Section 4]{zhang2025non}.

We fix a complete Riemannian manifold $(M,g)$ with the distance function $d\colon M\times M\to\mathbb{R}$ and mainly consider the smooth function $f\colon M\to\mathbb{R}\in C^\infty(M;\mathbb{R})$ satisfying part of the following conditions:
\begin{itemize}
    \item[(S1)] $f$ is a non-negative Morse function, i.e. $f\geqslant0$ and all critical points of $f$ are non-degenerate;

    \item[(S2)] $\|\dd f\|_g\geqslant\varepsilon_0\cdot f$ for a positive constant $\varepsilon_0$;

    \item[(S3)] $\|\dd f\|_g\geqslant\varepsilon_1\cdot f^{1/2}$ for a positive constant $\varepsilon_1$;
\end{itemize}
where $\|-\|_g$ denotes the norm associated with $g$. 

The following index estimates was proven in \cite{filip2021uniformization} and restated in \cite[Theorem 4.3]{zhang2025non} for functions satisfying (S1)(S2)(S3). 

\begin{theorem}\label{thm:index}
Let $f\in C^\infty(M;\mathbb{R})$ be a smooth function satisfying (S1) and (S2).
    \begin{itemize}
        \item[(1)] The only critical points of $f$ are local minima, which occur when $f(x) = 0$.

        \item[(2)] $\inf_{M}f= 0$, and furthermore if $f$ attains its infimum (indeed minimum $0$) at some $x_{\mathrm{min}}\in M$, then for any sequence $(x_n)_{n=1}^\infty$ with $\lim_{n\to\infty}f(x_n)=0$, $(x_n)_{n=1}^\infty$ converges to $x_{\mathrm{min}}$. In particular, $f$ has at most one critical point.

        \item[(3)] In addition, if $f$ satisfies (S3) as well, then its negative gradient $-\nabla f$ is forward complete, i.e. any integral curve of its associated flow extends to $+\infty$. Moreover, $f$ has precisely one critical point, which is the endpoint of all integral curves of the negative gradient flow.
    \end{itemize}
\end{theorem}

We prove the following lower bound of $f$ as an improvement of \cite[Lemma 2.2.12]{filip2021uniformization} and \cite[Lemma 4.4]{zhang2025non}.

\begin{lemma}\label{lemma:exponentialgrowth}
    Suppose that $f$ satisfies (S1)-(S3) and achieves its minimum (which is necessarily exists and unique by \prettyref{thm:index}) at $x_{\mathrm{min}}$.
    Then for any $x\in M$,
\[f(x)\geqslant\begin{cases}
    (\varepsilon_1^2/4)\cdot d(x_{\mathrm{min}},x)^2& \mbox{when }f(x)\leqslant(\varepsilon_1/\varepsilon_0)^2,\\
    (\varepsilon_1/\varepsilon_0)^2\cdot\exp\left(\varepsilon_0\cdot d(x_{\mathrm{min}},x)-2\right)& \mbox{when }f(x)>(\varepsilon_1/\varepsilon_0)^2,
\end{cases}\]
    where $d$ denotes the distance function of $(M,g)$, $\varepsilon_0$ and $\varepsilon_1$ are the constants in (S2)(S3) respectively. 
\end{lemma}

\begin{proof}
    Fix an arbitrary $x\in M$. We make use of the integral curve $\gamma$ of the negative gradient flow, i.e. 
   \[\begin{cases}
    \dfrac{\dd\gamma}{\dd t}(t)=-\nabla f(\gamma(t)),\\
    \gamma(0)=x.
\end{cases}\]
and we know that the maximal existence time of $\gamma$ is $+\infty$ with $\gamma(+\infty)=x_{\mathrm{min}}$ by \prettyref{thm:index} and we set $h=f\circ\gamma$. Then $-\dfrac{\dd h}{\dd t}(t)=\|\nabla f(\gamma(t))\|_{g}^2$. Hence we obtain that
\[\begin{aligned}
    d(x,x_{\mathrm{min}})
    \leqslant&\operatorname{length}(\gamma|_{[0,+\infty]})\\
    =&\int_{0}^{+\infty}\left\|\dfrac{\dd\gamma}{\dd t}\right\|_{g}\dd t\\
    =&\int_{0}^{+\infty}\left\|\nabla f(\gamma(t))\right\|_{g}\dd t\\
    =&-\int_{0}^{+\infty}\dfrac{\dd h/\dd t}{\left\|\nabla f(\gamma(t))\right\|_{g}}\dd t\\
    \leqslant&-\int_{0}^{+\infty}\dfrac{\dd h/\dd t}{\max\{\varepsilon_0 h,\varepsilon_1h^{1/2}\}}\dd t\\
    =&\int_{0}^{h(0)}\dfrac{\dd h}{\max\{\varepsilon_0 h,\varepsilon_1h^{1/2}\}}.
\end{aligned}\]

Now suppose that $h(0)=f(x)\leqslant(\varepsilon_1/\varepsilon_0)^2$, then $\varepsilon_0h\leqslant\varepsilon_1h^{1/2}$ on $[0,+\infty)$. Therefore,
\[d(x,x_{\mathrm{min}})\leqslant\int_{0}^{h(0)}\dfrac{\dd h}{\varepsilon_1h^{1/2}}=\dfrac{2}{\varepsilon_1}\cdot h(0)^{1/2}=\dfrac{2}{\varepsilon_1}\cdot f(x)^{1/2}.\]
On the other hand, when $h(0)=f(x)>(\varepsilon_1/\varepsilon_0)^2$, there exists unique $t_0>0$ such that $h(t_0)=(\varepsilon_1/\varepsilon_0)^2$. Therefore, 
\[\begin{aligned}
    d(x,x_{\mathrm{min}})
    \leqslant&\int_{0}^{h(0)}\dfrac{\dd h}{\max\{\varepsilon_0 h,\varepsilon_1h^{1/2}\}}\\
    =&\int_{0}^{h(t_0)}\dfrac{\dd h}{\varepsilon_1h^{1/2}}+\int_{h(t_0)}^{h(0)}\dfrac{\dd h}{\varepsilon_0 h}\\
    =&\dfrac{2}{\varepsilon_0}+\dfrac{1}{\varepsilon_0}\cdot\log\dfrac{ h(0)}{\left(\varepsilon_1/\varepsilon_0\right)^2}.
\end{aligned}\]
Thus the lemma follows.
\end{proof}

\subsection{Norm domination in the wedge product}

To find a suitable function satisfying (S1)-(S3), we give some linear-algebraic lemmata in this section.


Let $\mathbb{C}^{2,q+1}$ be the linear space $\mathbb{C}^{q+3}$ equipped with an indefinite Hermitian pairing $\langle\bullet,\ast\rangle$ of signature $(2,q+1)$. We define $\|\bullet\|^2:=\langle\bullet,\bullet\rangle$. We fix an orthogonal decomposition $\mathbb{C}^{2,q+1}=P\oplus N$ such that $P$ is spacelike and $N$ is timelike. Then $\dim P=2$ and $\dim N=q+1$. Let $\pi_P$ and $\pi_N$ denote the projection onto $P$ and $N$ respectively. The orthogonal decomposition also induces an indefinite Hermitian pairing and an orthogonal decomposition on the exterior power of $\mathbb{C}^{2,q+1}$. In particular, we consider $\bigwedge^2\mathbb{C}^{2,q+1}$ and the projections $\pi_{+,+}$, $\pi_{+,-}$, $\pi_{-,-}$ onto the subspace $\bigwedge^2P$, $P\otimes N$, $\bigwedge^{2}N$ respectively.

Let $(e_1,e_2)$ be an orthogonal basis of a $2$-dimensional subspace $V$ in $\mathbb{C}^{2,q+1}$. Below, we always assume $\|e_i\|^2\in\{1,0,-1\}$.



\begin{lemma}\label{lemma:normcomparison1}
    When the restriction of the indefinite Hermitian pairing on $V$ is positive semi-definite, i.e. $\|e_i\|^2\in\{0,1\}$, we have
    \[\|\pi_{+,+}(e_1\wedge e_2)\|^2\geqslant \|\pi_{-,-}(e_1\wedge e_2)\|^2.\]
    Moreover, when $V$ is spacelike, i.e. $\|e_i\|^2=1$, we have
    \[\|\pi_{+,+}(e_1\wedge e_2)\|^2\geqslant 1+\|\pi_{-,-}(e_1\wedge e_2)\|^2\geqslant\max\{1,\|\pi_{-,-}(e_1\wedge e_2)\|^2\}.\]
\end{lemma}

\begin{proof}

    We have the orthogonal decomposition $e_i=p_i+\nu_i$ with respect to $\mathbb{C}^{2,q+1}=P\oplus N$. Let $G=(\langle p_i,p_j\rangle)_{1\leqslant i,j\leqslant 2}$, $B=(\langle\nu_i,\nu_j\rangle)_{1\leqslant i,j\leqslant 2}$ denote the Gram matrix of $(p_i)$, $(\nu_i)$ respectively. Then $G=\operatorname{diag}(\|e_1\|^2,\|e_2\|^2)-B$. We obtain that
    \[\begin{aligned}
        \|\pi_{+,+}(e_1\wedge e_2)\|^2
        =&\|p_1\wedge p_2\|^2\\
        =&\det G\\
        =&(\|e_1\|^2-\|\nu_1\|^2)(\|e_2\|^2-\|\nu_2\|^2)-|\langle\nu_1,\nu_2\rangle|^2\\
        =&\|e_1\|^2\|e_2\|^2-\|e_2\|^2\|\nu_1\|^2-\|e_1\|^2\|\nu_2\|^2+\det(-B)\\
        \geqslant&\|\nu_1\wedge\nu_2\|^2\quad(\mbox{since }\|e_i\|^2\geqslant0\mbox{ and }\|\nu_i\|^2\leqslant0)\\
        =&\|\pi_{-,-}(e_1\wedge e_2)\|^2.
    \end{aligned}\]
    When $V$ is spacelike, we obtain that
    \[\|\pi_{+,+}(e_1\wedge e_2)\|^2=1-\|\nu_1\|^2-\|\nu_2\|^2+\det(-B)\geqslant1+\|\nu_1\wedge\nu_2\|^2=1+\|\pi_{-,-}(e_1\wedge e_2)\|^2.\]
    Now the last inequality follows from that both $1$ and $\|\pi_{-,-}(e_1\wedge e_2)\|^2$ are non-negative.
\end{proof}

\begin{lemma}\label{lemma:normcomparison2}
    When the restriction of the indefinite Hermitian pairing on $V$ is positive semi-definite, i.e. $\|e_i\|^2\in\{0,1\}$, then
    \[\|\pi_{+,+}(e_1\wedge e_2)\|^2\geqslant -\dfrac{1}{2}\|\pi_{+,-}(e_1\wedge e_2)\|^2.\]
\end{lemma}

\begin{proof}
    Note that $\|e_1\wedge e_2\|^2=\|e_1\|^2\|e_2\|^2\geqslant0$. We obtain that
    \[0\leqslant\|e_1\wedge e_2\|^2=\|\pi_{+,+}(e_1\wedge e_2)\|^2+\|\pi_{+,-}(e_1\wedge e_2)\|^2+\|\pi_{-,-}(e_1\wedge e_2)\|^2.\]
    Therefore 
    \[-\|\pi_{+,-}(e_1\wedge e_2)\|^2\leqslant2\|\pi_{+,+}(e_1\wedge e_2)\|^2\]
    by \prettyref{lemma:normcomparison1}.
\end{proof}

\subsection{Establish (S1)-(S3) for suitable function}\label{sec:choosefunction}
We consider the cyclic harmonic $\mathrm{U}_{2,q+1}$-Higgs bundle of type I:
\[
    (\EE,\Phi)=\begin{tikzcd}
	{\mathcal{L}} & {\mathcal{L}\mathcal{K}_X^{-1}} & {\mathcal{V}} & {\mathcal{W}}
	\arrow["{\mathds{1}}"', from=1-1, to=1-2]
	\arrow["\alpha"', from=1-2, to=1-3]
	\arrow["\beta"', from=1-3, to=1-4]
	\arrow["\gamma"', curve={height=19pt}, from=1-4, to=1-1]
\end{tikzcd}
\]
over the universal cover $\widetilde{X}$ of $X$. Here $\LL$, $\LL\KK_X^{-1}$ and $\mathcal{W}$ are all line bundles and $\VV$ is of rank $q$. It admits the harmonic metric $h$ and the induced projectively flat connection $\mathrm{D}^h=\nabla^h+\Phi+\Phi^{*_h}$,that is,
\[F(\mathrm{D}^h)=-2\pi\iu\mu\cdot\omega\otimes\operatorname{id}_{\EE}\]
for the lifting $\omega$ of the area $1$ K\"ahler form $\omega_X$. Since $\widetilde{X}$ is simply connected, there is a real $1$-form $\alpha$ on $\widetilde{X}$ such that $\dd\alpha=\mu\cdot\omega$, hence \[\mathrm{D}^{h,\alpha}:=\mathrm{D}^h+2\pi\iu\alpha\otimes\operatorname{id}_{\EE}\] is flat and \[\nabla^{h,\alpha}=\nabla^h+2\pi\iu\alpha\otimes\operatorname{id}_{\EE}\] is the skew-Hermitian part of $\mathrm{D}^{h,\alpha}$ with respect to the metric $h$. 

For our convenience, we denote $\LL,\LL\KK_X^{-1},\VV,\WW$ by $\EE_0,\EE_1,\EE_2,\EE_3$ respectively.  

We fix a basepoint $x_0\in\widetilde{X}$ and an oriented $2$-plane $\mathbf{S}$ in $\EE_{x_0}$. With respect to the flat connection $\mathrm{D}^{h,\alpha}$, $\mathbf{S}$ can be extended to a global oriented $2$-plane section, which we still denote by $\mathbf{S}$, on $\widetilde{X}$ flatly. Since the flat connection preserves the indefinite pairing $\langle\bullet,\ast\rangle$, the signature of $\langle\bullet,\ast\rangle|_{\mathbf{S}}$ is constant. Equivalently, via Pl\"ucker embedding, we can choose an ordered basis $(e_1,e_2)$ of $\mathbf{S}_{x_0}$. They can be extended to global flat sections, which we still denote by $e_1,e_2$, on $\widetilde{X}$. Then $\mathbf{S}$ is the same as the global flat section $e_1\wedge e_2$ in $(\wedge^2\EE,\wedge^2\Phi)$ with respect to the flat connection $\wedge^2\mathrm{D}^{h,\alpha}$. We will denote by $\pi_{j,k}$ the projection onto $\EE_j\wedge\EE_k$ and by $H$ the induced harmonic metric $\wedge^2h$ of $(\wedge^2\EE,\wedge^2\Phi)$. Also we have the connection \[\nabla^{H,\alpha}=\nabla^H+4\pi\iu\alpha\otimes\id_{\wedge^2\EE}\] which is skew-Hermitian with respect to $H$.

Below we will consider the subspace $\mathbf{S}$ such that $\langle\bullet,\ast\rangle|_{\mathbf{S}}$ is positive semi-definite. In particular, we will choose perpendicular basis $e_1,e_2$ of $\mathbf{S}$ such that
\begin{itemize}
    \item $e_1,e_2$ are spacelike and have unit norm when $\langle\bullet,\ast\rangle|_{\mathbf{S}}$ is positive definite;

    \item $e_1$ is spacelike and has unit norm, $e_2$ is isotropic when $\langle\bullet,\ast\rangle|_{\mathbf{S}}$ is neither positive definite nor $0$;

    \item $e_1,e_2$ are lightlike when $\langle\bullet,\ast\rangle|_{\mathbf{S}}$ is $0$.
\end{itemize}

We simply denote $e_1\wedge e_2$ by $\mathbf{e}$ and $\pi_{j,k}(\mathbf{e})$ by $\mathbf{e}_{j,k}$. Now we define the function $f_{\mathbf{S}}$ over $\widetilde{X}$ as $-\|\mathbf{e}_{0,3}\|^2=\|\mathbf{e}_{0,3}\|_H^2$. 

Fix a complete conformal metric $g$ on $X$, and use the same
notation for its lift to $\widetilde X$. Norms of Higgs-field
components are computed using the harmonic metric $h$ and $g$. In the following, we first assume that
\begin{assumption}\label{ass:positive-gap} 
   There exists a constant $\varepsilon>0$ such that
\[
\|1\|_{h,g}-\|\beta\|_{h,g}\geq\varepsilon.
\]
\end{assumption}

\begin{remark}
    Note that when $X$ is closed, \prettyref{ass:positive-gap} holds for non-maximal $(\EE,\Phi)$ by \prettyref{lemma:higgsfieldnorm}.
\end{remark}

Then we prove that $f_{\mathbf{S}}$ satisfies (S1)(S2) under this assumption. Moreover, we will prove that $f_{\mathbf{S}}$ satisfies (S3) when $\mathbf{S}$ is spacelike. 

\begin{lemma}\label{lemma:S23}
    If \prettyref{ass:positive-gap} holds, there is a constant $\varepsilon_0>0$ independent of the choice of $\mathbf{S}$ such that $\|\dd f_{\mathbf{S}}\|_g\geqslant\varepsilon_0\cdot f_{\mathbf{S}}$. Moreover, when $\mathbf{S}$ is spacelike, $\|\dd f_{\mathbf{S}}\|_g\geqslant\sqrt{2}\varepsilon_0\cdot f_{\mathbf{S}}^{1/2}$. 
\end{lemma}

\begin{proof}

    \[\begin{aligned}
        \|\dd f_{\mathbf{S}}\|_g
        =&\|\dd H(\mathbf{e}_{0,3},\mathbf{e}_{0,3})\|_g\\
        =&\|H(\nabla^{H,\alpha}(\mathbf{e}_{0,3}),\mathbf{e}_{0,3})+H(\mathbf{e}_{0,3},\nabla^{H,\alpha}(\mathbf{e}_{0,3}))\|_{g}\\
        =&\left\|H((\operatorname{id}_{\EE_0}\wedge\beta)(\mathbf{e}_{0,2})+(\one\wedge\operatorname{id}_{\EE_3})^{*_H}(\mathbf{e}_{1,3}),\mathbf{e}_{0,3})\right.\\
        &\ \left.+H(\mathbf{e}_{0,3},(\operatorname{id}_{\EE_0}\wedge\beta)(\mathbf{e}_{0,2})+(\one\wedge\operatorname{id}_{\EE_3})^{*_H}(\mathbf{e}_{1,3}))\right\|_{g}\\
        =&\sqrt{2}\cdot\left\|H((\one\wedge\operatorname{id}_{\EE_3})^{*_H}(\mathbf{e}_{1,3}),\mathbf{e}_{0,3})\right.\\
        &\ \left.+H(\mathbf{e}_{0,3},(\operatorname{id}_{\EE_0}\wedge\beta)(\mathbf{e}_{0,2}))\right\|_{g}\quad(\mbox{since }\dd z\perp\dd\bar z\mbox{ and }\|\dd z\|=\|\dd\bar z\|)\\
        \geqslant&\sqrt{2}\cdot\left|\|H((\one\wedge\operatorname{id}_{\EE_3})^{*_H}(\mathbf{e}_{1,3}),\mathbf{e}_{0,3})\|-\|H(\mathbf{e}_{0,3},(\operatorname{id}_{\EE_0}\wedge\beta)(\mathbf{e}_{0,2}))\|_{g}\right|\\
        \end{aligned}\]

        By \prettyref{lemma:higgsfieldnorm}, we obtain that
        \begin{equation}\label{eq:higgsfieldcomparison}
            \|\one\wedge\operatorname{id}_{\EE_3}\|=\|\one\|>\|\beta\|=\|\operatorname{id}_{\EE_0}\wedge\beta\|
        \end{equation}
        and by \prettyref{lemma:normcomparison1} we have
        \begin{equation}\label{eq:normcomparison}
            \|\mathbf{e}_{1,3}\|_H\geqslant\|\mathbf{e}_{0,2}\|_H.
        \end{equation}
        Therefore,
        \[\begin{aligned}
        \|\dd f_{\mathbf{S}}\|_g
        \geqslant&\sqrt{2}\cdot\left|\|H((\one\wedge\operatorname{id}_{\EE_3})^{*_H}(\mathbf{e}_{1,3}),\mathbf{e}_{0,3})\|-\|H(\mathbf{e}_{0,3},(\operatorname{id}_{\EE_0}\wedge\beta)(\mathbf{e}_{0,2}))\|_{g}\right|\\
        \geqslant&\sqrt{2}\cdot\left(\|\one\|-\|\beta\|\right)\cdot\|\mathbf{e}_{1,3}\|_H\cdot\|\mathbf{e}_{0,3}\|_H
    \end{aligned}\]
    Therefore, by \prettyref{lemma:normcomparison2}
    \[\|\dd f_{\mathbf{S}}\|_g\geqslant\left(\|\one\|-\|\beta\|\right)\cdot\|\mathbf{e}_{0,3}\|_H^2=\left(\|\one\|-\|\beta\|\right)\cdot f_{\mathbf{S}}.\]
    And when $\mathbf{S}$ is spacelike, by \prettyref{lemma:normcomparison1} we obtain that
    \[\|\dd f_{\mathbf{S}}\|_g\geqslant\sqrt{2}\cdot\left(\|\one\|-\|\beta\|\right)\cdot\|\mathbf{e}_{0,3}\|_H=\sqrt{2}\cdot\left(\|\one\|-\|\beta\|\right)\cdot f_{\mathbf{S}}^{1/2}.\]

 By \prettyref{ass:positive-gap}, $\|\one\|-\|\beta\|\geqslant \epsilon,$ for some positive constant $\epsilon.$ Thus we have the statements.
\end{proof}

\begin{lemma}\label{lemma:p1-morse}
    $f_{\mathbf{S}}$ is a non-negative Morse function.
\end{lemma}

\begin{proof}
    Since $\|\dd f_{\mathbf{S}}\|_{g}\geqslant\left(\|\one\|-\|\beta\|\right)\cdot f_{\mathbf{S}}$ by \prettyref{lemma:S23} and $\|\one\|-\|\beta\|>0$ by \prettyref{lemma:higgsfieldnorm}, we obtain that the critical points of $f_{\mathbf{S}}$ only occur when $\mathbf{e}_{0,3}(x)=0$. Now we compute the Hessian of $f_{\mathbf{S}}$ at its critical points. Given two real vector fields $T_1,T_2$ around a critical point $x$. Since $f_{\mathbf{S}}(x)=0$, by the compatibility between $\nabla^{H,\alpha}$ and the Hermitian metric $H$, one can readily check that
    \begin{equation}\label{eq:secondderi}
        T_1T_2(f_{\mathbf{S}})(x)=\left(H(\nabla_{T_1}^{H,\alpha}(\mathbf{e}_{0,3}),\nabla_{T_2}^{H,\alpha}(\mathbf{e}_{0,3}))+H(\nabla_{T_2}^{H,\alpha}(\mathbf{e}_{0,3}),\nabla_{T_1}^{H,\alpha}(\mathbf{e}_{0,3}))\right)(x).
    \end{equation}
    Now we take a local unit frame $l$ of $\EE_{0}\wedge\EE_3$ around $x$. Locally we have $(\one\wedge\operatorname{id}_{\EE_3})^{*_H}(\mathbf{e}_{1,3})(\partial/\partial\bar z)=sl$ and $(\operatorname{id}_{\EE_0}\wedge\beta)(\mathbf{e}_{0,2})(\partial/\partial z)=tl$ for some complex-valued smooth functions $s$ and $t$. By \prettyref{eq:higgsfieldcomparison}, \prettyref{eq:normcomparison} and $\|\partial/\partial z\|_{g}=\|\partial/\partial \bar z\|_{g}$ we obtain that $|s|>|t|$.

    With respect to the natural coordinate basis $\partial/\partial x,\partial/\partial y$, a quick calculation shows that the coordinate Hessian of $f_v$ at $x_0$ can be represented as
    \[\begin{pmatrix}
        2|s+t|^2&\iu\left[\overline{(s+t)}(s-t)-\overline{(s-t)}(s+t)\right]\\\iu\left[\overline{(s+t)}(s-t)-\overline{(s-t)}(s+t)\right]&2|s-t|^2
    \end{pmatrix}.\]
    It suffices to prove that the determinant of above matrix is not $0$. Actually, the determinant is
    \[\begin{aligned}
        &4\left[|s+t|^2\cdot|s-t|^2\right]+\left[\overline{(s+t)}(s-t)-\overline{(s-t)}(s+t)\right]^2\\=&\left[\overline{(s+t)}(s-t)+\overline{(s-t)}(s+t)\right]^2\geqslant0
    \end{aligned}\]
    and the ``='' holds iff \[\begin{aligned}
        &\overline{(s+t)}(s-t)+\overline{(s-t)}(s+t)=0\\
        \iff&\overline{(s-t)}(s+t)=r\iu\mbox{ for some real }r\\
        \iff& s=\dfrac{r\iu+1}{r\iu-1}t\mbox{ for some real }r\\
        \implies&|s|=|t|,
    \end{aligned}\]
    contradiction.
\end{proof}

Therefore, by \prettyref{thm:index} and \prettyref{lemma:exponentialgrowth} we obtain the following corollary.

\begin{corollary}\label{coro:indexfS}
    If \prettyref{ass:positive-gap} holds, $f_\mathbf{S}$ has at most one critical point, which has to be its zero. Furthermore, if $\mathbf{S}$ is spacelike, then $f_{\mathbf{S}}$ has exactly one critical point $x_{\mathrm{min}}$ which is its minima and there exists positive constant $\varepsilon_0>0$ which is independent of the choice of $\mathbf{S}$ such that 
    \[f_{\mathbf{S}}(x)\geqslant\begin{cases}
    (\varepsilon_0^2/2)\cdot d(x_{\mathrm{min}},x)^2& \mbox{when }f_{\mathbf{S}}(x)\leqslant2,\\
    2\cdot\exp\left(\varepsilon_0\cdot d(x_{\mathrm{min}},x)-2\right)& \mbox{when }f_{\mathbf{S}}(x)>2.
\end{cases}\]
\end{corollary}

\subsection{Proof of bi-Lipschitz embedding}

\begin{theorem}\label{thm:embedding} Assume that \prettyref{ass:positive-gap}  holds with $g=g_X$ and $\|q_4\|_{g_X}$ is bounded,
where $g_X$ is the complete hyperbolic metric on $X$, then the minimal map $h\colon\widetilde{X}\to\operatorname{Sym}(\PGR)=\PGR/\PKR\cong\GR/\KR$ is a bi-Lipschitz embedding. 
    
    Moreover, when $X$ is compact hyperbolic, then the equivariant minimal map $h\colon\widetilde{X}\to\operatorname{Sym}(\PGR)$ is a bi-Lipschitz embedding. As a corollary, the orbit map of the corresponding representation $\rho\colon\pi_1(X)\to\PU_{2,q+1}$ is a quasi-isometric embedding.
\end{theorem}

\begin{proof}
     We first assume that \prettyref{ass:positive-gap} holds. Recall that we fix a basepoint $x_0\in \widetilde{X}$. Let $\mathbf{S}$ be the spacelike $2$-plane $(\LL\KK_X^{-1}\oplus\WW)_{x_0}$ and $f_{\mathbf{S}}$ be the function defined as in \prettyref{sec:choosefunction}. Then $f_{\mathbf{S}}(x_0)=0$ and thus $x_0$ is the unique minima of $f_{\mathbf{S}}$ by \prettyref{coro:indexfS}. 
    
    We fix an arbitrary $x\in\widetilde{X}$. Use the flat connection $\mathrm{D}^{h,\alpha}$, we can identify the principal $\GR$-bundle over $\widetilde{X}$ with the trivial principal $\GR$-bundle $\widetilde{X}\times \GR\to \widetilde{X}$ which lifts the trivialization of the principal $\PGR$-bundle over $\widetilde{X}$ using the flat connection that is the projectivization of $\mathrm{D}^h$. And $h$ can be regarded as a $\KR$-subbundle in the trivial $\GR$-bundle with $h(x_0)=\KR$. Let $h(x)=g\KR$. We have the KAK decomposition $g=k_+\mu(g)k_{-1}$, where $k_+,k_-\in \KR$ and $\mu$ denotes the Cartan projection, which is independent on the choice of $g$. Then
    \[d_{\operatorname{Sym}(\GR)}(h(x),h(x_0))=\|\mu(g)\|.\] Let
    \[(\mu_1,\mu_2,0,\cdots,0,-\mu_2,-\mu_1):=\mu(g).
    \]
    We can take orthonormal vectors $u_1,u_2,v_1,v_2$, where $u_i$'s are positive and $v_i$'s are negative, such that $w_1^+:=u_1+v_1$, $w_2^+:=u_2+v_2$ ,$w_2^-:=u_2-v_2$, $w_1^-:=u_1-v_1$ are the eigenvectors of $\mu_1,\mu_2,-\mu_2,-\mu_1$. Let $\mathbf{S}$ be the spacelike $2$-plane spanned by the orthonormal basis $\mathbf{e}=(k_-^{-1}u_1,k_-^{-1}u_2)$. Then we obtain that
    \[k_{+}^{-1}\left(\mathbf{e}(x)\right)=\dfrac{\mathrm{e}^{\mu_1+\mu_2}\cdot w_1^+\wedge w_2^++\mathrm{e}^{\mu_1-\mu_2}\cdot w_1^+\wedge w_2^-+\mathrm{e}^{-\mu_1+\mu_2}\cdot w_1^-\wedge w_2^++\mathrm{e}^{-\mu_1-\mu_2}\cdot w_1^-\wedge w_2^-}{4},\]
    and 
    \[\begin{aligned}
        \|\mathbf{e}(x)\|_H^2
        =&\dfrac{\mathrm{e}^{2\mu_1+2\mu_2}+\mathrm{e}^{2\mu_1-2\mu_2}+\mathrm{e}^{-2\mu_1+2\mu_2}+\mathrm{e}^{-2\mu_1-2\mu_2}}{4}\\
        =&\cosh(2\mu_1)\cdot\cosh(2\mu_2)\\
        \leqslant&(\cosh(2\mu_1))^2.
    \end{aligned}\]
    Now we obtain that
    \[\begin{aligned}
        (\cosh(2\mu_1))^2
        \geqslant&\|\mathbf{e}(x)\|_H^2\\
        \geqslant&\|\mathbf{e}_{1,3}(x)\|_H^2+\|\mathbf{e}_{0,3}(x)\|_H^2\\
        \geqslant&1+\|\mathbf{e}_{0,3}(x)\|_H^2\quad\mbox{(by \prettyref{lemma:normcomparison1})}\\
        =&1+f_{\mathbf{S}}(x).
    \end{aligned}\]
    Therefore, $(\sinh(2\mu_1))^2\geqslant f_{\mathbf{S}}(x)$. Recall that by \prettyref{coro:indexfS},

\[f_{\mathbf{S}}(x)\geqslant\begin{cases}
    (\varepsilon_0^2/2)\cdot d(x_{0},x)^2& \mbox{when }f_{\mathbf{S}}(x)\leqslant2,\\
    2\cdot\exp\left(\varepsilon_0\cdot d(x_{0},x)-2\right)& \mbox{when }f_{\mathbf{S}}(x)>2.
    \end{cases}\]
    
When $f_{\mathbf{S}}(x)\leqslant2$, this yields that $\varepsilon_0\cdot d(x_0,x)\leqslant2$.
Hence \[\mu_1\geqslant\dfrac{1}{2}\cdot\operatorname{arcsinh}\left(\dfrac{\varepsilon_0}{\sqrt{2}}\cdot d(x_0,x)\right)\geqslant\dfrac{\varepsilon_0\log(\sqrt{2}+\sqrt{3})}{4}\cdot d(x_0,x)\]
since $\operatorname{arcsinh}(t)\geqslant t\cdot\log(\sqrt{2}+\sqrt{3})/\sqrt{2}$ when $t\in[0,\sqrt{2}]$.

When $f_{\mathbf{S}}(x)>2$ and $d(x_0,x)>2/\varepsilon_0$,
we have
\[\begin{aligned}
    \mu_1\geqslant&\dfrac{\operatorname{arcsinh}(\sqrt{f_{\mathbf{S}}(x)})}{2}\\
    \geqslant&\dfrac{\log2}{2}+\dfrac{\log f_{\mathbf{S}}(x)}{4}\quad\mbox{(since }\operatorname{arcsinh}(t)=\log(t+\sqrt{t^2+1})\geqslant\log(2t)\mbox{when }t\geqslant0\mbox{)}\\
    \geqslant&\dfrac{\varepsilon_0}{4}\cdot d(x_0,x)+\dfrac{3\log 2}{4}-\dfrac{1}{2}\\
    >&\dfrac{\varepsilon_0}{4}\cdot d(x_0,x).
\end{aligned}\]

When $f_{\mathbf{S}}(x)>2$ and $d(x_0,x)\leqslant2/\varepsilon_0$,
we obtain that
\[\begin{aligned}
    \mu_1\geqslant&\dfrac{\operatorname{arcsinh}(\sqrt{f_{\mathbf{S}}(x)})}{2}\\
    >&\dfrac{\log(\sqrt{2}+\sqrt{3})}{2}\quad\mbox{(since }f_{\mathbf{S}}(x)>2\mbox{)}\\
    \geqslant&\dfrac{\varepsilon_0\log(\sqrt{2}+\sqrt{3})}{4}\cdot d(x_0,x).\quad\mbox{(since }d(x_0,x)\leqslant2/\varepsilon_0\mbox{)}
\end{aligned}\]

Combining all cases, we obtain the uniform lower bound
\[\mu_1\geqslant\dfrac{\varepsilon_0}{4}\cdot d(x_0,x).\]
Therefore,
\[d_{\operatorname{Sym}(\GR)}(h(x),h(x_0))=\|\mu(g)\|=\sqrt{2(\mu_1^2+\mu_2^2)}\geqslant\sqrt{2}\mu_1\geqslant\dfrac{\varepsilon_0}{2\sqrt{2}}\cdot d(x_0,x).\]

By Lemma \ref{lem:BoundedAbove} and the assumption that $\|q_4\|_{g_{\mathbb D}}$ is bounded, we have the metric upper bound $h^*g_{\operatorname{Sym}(\GR)}\leqslant C g_{\D}$ for some constant $C$ depending on the bound of $|q_4|_{g_{\mathbb D}}$. This implies the Lipschitz upper bound \[d_{\operatorname{Sym}(\GR)}(h(x),h(x_0))\leqslant C\cdot d(x,x_0).\]

Thus the minimal map is a proper, injective immersion, hence an embedding.

Now assume that $X$ is compact hyperbolic. Set
\[
g_0=|1|_h^2.
\]
Since $g_0$ is complete, Lemma \ref{lemma:higgsfieldnorm} implies that either
$\|\beta\|_{h,g_0}<1$ everywhere or $\|\beta\|_{h,g_0}\equiv1$.

In the first case, compactness implies that \prettyref{ass:positive-gap} holds with respect to $g_X$.
Moreover, $\|q_4\|_{g_X}$ is bounded, so the first part
of the proof applies.

If $\|\beta\|_{h,g_0}\equiv1$, Lemma \ref{lemma:higgsfieldnorm} gives an orthogonal splitting
into a cyclic $U_{2,2}$-harmonic bundle and a compact factor.
It therefore suffices to consider the rank-four factor,
on which the $\beta$-arrow is an isomorphism.  Now \[\wedge^2(\EE,\Phi)=
\begin{tikzcd}
	{\mathcal{E}_0\wedge\mathcal{E}_1} & {\mathcal{E}_0\wedge\mathcal{E}_2} & {\mathcal{E}_1\wedge\mathcal{E}_2} & {\mathcal{E}_1\wedge\mathcal{E}_3} & {\mathcal{E}_2\wedge\mathcal{E}_3} \\
	&& {\mathcal{E}_0\wedge\mathcal{E}_3}
	\arrow["\alpha", from=1-1, to=1-2]
	\arrow["{\mathds{1}}", from=1-2, to=1-3]
	\arrow["{\mathds{1}}"', from=1-2, to=2-3]
	\arrow["{\mathds{1}}", from=1-3, to=1-4]
	\arrow["\gamma"', curve={height=19pt}, from=1-4, to=1-1]
	\arrow["\alpha", from=1-4, to=1-5]
	\arrow["\gamma"', curve={height=19pt}, from=1-5, to=1-2]
	\arrow["{\mathds{1}}"', from=2-3, to=1-4]
\end{tikzcd}.\] 
Since $\EE_0\cong\EE_1\KK_X$ and $\EE_2\cong\EE_3\KK_X$, after tensoring by $(\EE_1\EE_2)^{-1}\cong(\EE_0\EE_3)^{-1}$, we obtain 
\[\begin{tikzcd}
	{\mathcal{E}_0\mathcal{E}_2^{-1}} & {\mathcal{K}_X} & {\mathcal{O}_X} & {\mathcal{K}_X^{-1}} & {\mathcal{E}_2\mathcal{E}_0^{-1}} \\
	&& {\mathcal{O}_X}
	\arrow["\alpha", from=1-1, to=1-2]
	\arrow["{\mathds{1}}", from=1-2, to=1-3]
	\arrow["{\mathds{1}}"', from=1-2, to=2-3]
	\arrow["{\mathds{1}}", from=1-3, to=1-4]
	\arrow["\gamma"', curve={height=19pt}, from=1-4, to=1-1]
	\arrow["\alpha", from=1-4, to=1-5]
	\arrow["\gamma"', curve={height=19pt}, from=1-5, to=1-2]
	\arrow["{\mathds{1}}"', from=2-3, to=1-4]
\end{tikzcd}.\]
By changing the basis of the middle two $\mathcal{O}_X$, this Higgs bundle is equivalent to 
\[\begin{tikzcd}
	{\mathcal{E}_0\mathcal{E}_2^{-1}} & {\mathcal{K}_X} & {\mathcal{O}_X} & {\mathcal{K}_X^{-1}} & {\mathcal{E}_2\mathcal{E}_0^{-1}} \\
	&& {\mathcal{O}_X}
	\arrow["\alpha"', from=1-1, to=1-2]
	\arrow["{\mathds{1}}"', from=1-2, to=1-3]
	\arrow["{\mathds{1}}"', from=1-3, to=1-4]
	\arrow["\gamma"', curve={height=19pt}, from=1-4, to=1-1]
	\arrow["\alpha"', from=1-4, to=1-5]
	\arrow["\gamma"', curve={height=19pt}, from=1-5, to=1-2]
	\arrow["\oplus"{description}, draw=none, from=2-3, to=1-3]
\end{tikzcd},\]
which is a maximal $\SO_{2,3}^0\times\SO_1\mathbb{R}$-Higgs bundle (see \cite[Proposition 2.26]{collier2019geometry}). The $\SO_{2,3}^0$-part gives the same minimal map as $(\EE,\Phi)$ and it can be regarded as a 4-cyclic $\mathrm{U}_{2,3}$-Higgs bundle 
\[\begin{tikzcd}
	{\mathcal{K}_X} & {\mathcal{O}_X} & {\mathcal{K}_X^{-1}} & {\mathcal{E}_2\mathcal{E}_0^{-1}\oplus\mathcal{E}_0\mathcal{E}_2^{-1}}
	\arrow["{\mathds{1}}"', from=1-1, to=1-2]
	\arrow["{\mathds{1}}"', from=1-2, to=1-3]
	\arrow["{(\alpha,\gamma)^{\mathrm{T}}}"', from=1-3, to=1-4]
	\arrow["{(\gamma,\alpha)}"', curve={height=19pt}, from=1-4, to=1-1]
\end{tikzcd}.\]

Note that it must be non-maxiaml as a $\mathrm{U}_{2,3}$-Higgs bundle and hence satisfy \prettyref{ass:positive-gap}.  Both the positive and negative blocks of this
$\mathrm{U}_{2,3}$-Higgs bundle have degree zero, so its Toledo
invariant vanishes. Equality in Lemma \ref{lemma:higgsfieldnorm}  would imply
maximality and is therefore impossible. 
\end{proof}

\begin{remark}
    It is a little bit subtle in the proof that a maximal $\SO_{2,3}^0$-Higgs bundle is non-maximal as a $\mathrm{U}_{2,3}$-Higgs bundle. The reason is that the trivial embedding from the symmetric space of $\SO_{2,3}^0$ to that of $\mathrm{U}_{2,3}$ is totally real. Hence a Higgs bundle has different Toledo invariants in these two cases.
\end{remark}

\begin{remark}\label{rem:q=1}
    When $X$ is compact hyperbolic, $q=1$, the corresponding representation $\rho\colon\pi_1(X)\to\PU_{2,2}$ must be Anosov. Indeed, when $(\EE,\Phi)$ is maximal, $\rho$ is also maximal. Hence $\rho$ is $P_2$-Anosov by \cite{burger2010surface}. When $(\EE,\Phi)$ is non-maximal, $(\EE,\Phi)$ is of the form 
    \[\begin{tikzcd}
	{\EE_0} & {\EE_1} & {\EE_2} & {\EE_3}
	\arrow["{\mathds{1}}"', from=1-1, to=1-2]
	\arrow["\alpha"', from=1-2, to=1-3]
	\arrow["\beta"', from=1-3, to=1-4]
	\arrow["\gamma"', curve={height=19pt}, from=1-4, to=1-1]
\end{tikzcd}\] where $\beta$ is not an isomorphism. So \[\wedge^2(\EE,\Phi)=
\begin{tikzcd}
	{\mathcal{E}_0\wedge\mathcal{E}_3} & {\mathcal{E}_1\wedge\mathcal{E}_3} & {\mathcal{E}_2\wedge\mathcal{E}_3} & {\mathcal{E}_0\wedge\mathcal{E}_2} & {\mathcal{E}_1\wedge\mathcal{E}_2} \\
	&& {\mathcal{E}_0\wedge\mathcal{E}_1}
	\arrow["{\mathds{1}}"', from=1-1, to=1-2]
	\arrow["\alpha"', from=1-2, to=1-3]
	\arrow["\gamma"', from=1-2, to=2-3]
	\arrow["\gamma"', from=1-3, to=1-4]
	\arrow["\beta"', curve={height=18pt}, from=1-4, to=1-1]
	\arrow["{\mathds{1}}"', from=1-4, to=1-5]
	\arrow["\beta"', curve={height=18pt}, from=1-5, to=1-2]
	\arrow["\alpha"', from=2-3, to=1-4]
\end{tikzcd}.\] 
It follows that $\wedge^2\rho\colon\pi_1(X)\to\mathrm{PO}_{2,4}$ is $P_2$-Anosov by \cite[Remark 5.7]{zhang2025non}. Hence $\rho$ is $P_1$-Anosov.
\end{remark}

\subsection{Cyclic Higgs bundle of Type II}
We next consider cyclic $\mathrm{U}_{2,q+1}$-harmonic bundle of type II:
\[
    (\EE,\Phi)=\begin{tikzcd}
	{\mathcal{L}} & {\mathcal{L}\mathcal{K}_X^{-1}} & {\mathcal{V}} & {\mathcal{W}}
	\arrow["{\mathds{1}}"',  from=1-1, to=1-2]
	\arrow["\alpha"',  from=1-2, to=1-3]
	\arrow["\beta"', from=1-3, to=1-4]
	\arrow["\gamma"', curve={height=19pt}, from=1-4, to=1-1]
\end{tikzcd}
\]
over the universal cover $\widetilde{X}$ of $X$. Here $\LL$, $\LL\KK_X^{-1}$ and $\mathcal{V}$ are all line bundles and $\Wcal$ is of rank $q$.

We fix a basepoint $x_0\in\widetilde{X}$ and an oriented $2$-plane $\mathbf{S}$ in $\EE_{x_0}$. Similarly, we have the flat connection $\mathrm{D}^{h,\alpha}$ over $\widetilde{X}$ for some real $1$-form $\alpha$. Then $\mathbf{S}$ can be extended to a global oriented $2$-plane section, which we still denote by $\mathbf{S}$, on $\widetilde{X}$ flatly. Since the flat connection preserves the indefinite pairing $\langle\bullet,\ast\rangle$, the signature of $\langle\bullet,\ast\rangle|_{\mathbf{S}}$ is constant. Equivalently, via Pl\"ucker embedding, we can choose an ordered basis $(e_1,e_2)$ of $\mathbf{S}_{x_0}$. They can be extended to global flat sections, which we still denote by $e_1,e_2$, on $\widetilde{X}$. Then $\mathbf{S}$ is the same as the global flat section $e_1\wedge e_2$ in $(\wedge^2\EE,\wedge^2\Phi)$ with respect to the flat connection $\wedge^2\mathrm{D}^{h,\alpha}$. We will denote by $\pi_{j,k}$ the projection onto $\EE_j\wedge\EE_k$ and by $H$ the induced harmonic metric $\wedge^2h$ of $(\wedge^2\EE,\wedge^2\Phi)$.

Below we will consider the subspace $\mathbf{S}$ such that $\langle\bullet,\ast\rangle|_{\mathbf{S}}$ is positive semi-definite. In particular, we will choose perpendicular basis $e_1,e_2$ of $\mathbf{S}$ such that
\begin{itemize}
    \item $e_1,e_2$ are spacelike and have unit norm  when $\langle\bullet,\ast\rangle|_{\mathbf{S}}$ is positive definite;

    \item $e_1$ is spacelike and has unit norm, $e_2$ is isotropic when $\langle\bullet,\ast\rangle|_{\mathbf{S}}$ is neither positive definite nor $0$;

    \item $e_1,e_2$ are lightlike when $\langle\bullet,\ast\rangle|_{\mathbf{S}}$ is $0$.
\end{itemize}

We simply denote $e_1\wedge e_2$ by $\mathbf{e}$ and $\pi_{j,k}(\mathbf{e})$ by $\mathbf{e}_{j,k}$. Now we define the function $f_{\mathbf{S}}'$ over $\widetilde{X}$ as $-\|\mathbf{e}_{1,2}\|^2=\|\mathbf{e}_{1,2}\|_H^2$. 

\begin{theorem}\label{thm:TypeIIf} If \prettyref{ass:positive-gap} holds, then $f_{\mathbf{S}}'$ is a Morse function and satisfies (S1)(S2). Moreover, $f_{\mathbf{S}}'$ satisfies (S3) when $\mathbf{S}$ is spacelike. As a consequence, $f_{\mathbf{S}}'$ has at most one zero and has one zero when $\mathbf{S}$ is spacelike.
\end{theorem}
\begin{proof}
 From Section \ref{subsec:duality}, we can consider the dual harmonic bundle and see it is of type I. Moreover, the function $f_{\mathbf{S}}'$ coincides with the function $f_{\mathbf{S}}$ for the dual harmonic bundle. Thus all results follow from Lemma \ref{lemma:S23} and Lemma \ref{lemma:p1-morse} of $f_{\mathbf{S}}$.
\end{proof}

\begin{theorem}\label{thm:embedding1}
Assume that \prettyref{ass:positive-gap}  holds with $g=g_X$ and $\|q_4\|_{g_X}$ is bounded,
where $g_X$ is the complete hyperbolic metric on $X$, then the minimal map $h\colon\widetilde{X}\to\operatorname{Sym}(G)=G/K$  is a bi-Lipschitz embedding. 

Moreover, when $X$ is compact hyperbolic, then the minimal map is a bi-Lipschitz embedding. As a corollary, the orbit map of the corresponding representation $\rho\colon\pi_1(X)\to\PU_{2,q+1}$ is a quasi-isometric embedding.
\end{theorem}
\begin{proof}
From Section \ref{subsec:duality}, we can consider the dual harmonic bundle and see it is of type I. Thus from Theorem \ref{thm:embedding}, we obtain the  $\rho^\vee$-equivariant harmonic map $h^\vee$ is a bi-Lipschitz embedding. Since $h^\vee$ and $h$ are isometric, we then have the  $\rho$-equivariant harmonic map $h$ itself is again a bi-Lipschitz embedding. 
\end{proof}

\section{Geometric structures}\label{sec:GeometricStructure}
\subsection{Model spaces: non-negative 2-planes in \texorpdfstring{\(\mathbb{C}^{2,q+1}\)}{C2,q+1}}

Assume \(q\geq1\), and let
$\mathbb C^{2,q+1}$
be equipped with a non-degenerate Hermitian form
\(\langle\bullet,\ast\rangle\) of signature \((2,q+1)\). For
\(b\in\{0,1,2\}\), set
\[
\mathcal S_b
=
\left\{
U\in\operatorname{Gr}_2(\mathbb C^{2,q+1})\mid
\langle\bullet,\ast\rangle|_U\geq0,\ 
\operatorname{rk}\bigl(\langle\bullet,\ast\rangle|_U\bigr)=b
\right\}.
\]
The group \(\GR=\mathrm{U}_{2,q+1}\) acts transitively on each
\(\mathcal S_b\).

The open stratum $\mathcal S_2$ is the Hermitian symmetric space of noncompact type
 $\mathcal X_{2,n+1}$
of complex dimension \(2(q+1)\), and hence of real dimension $4q+4$.

For \(U\in\mathcal S_1\), let $N$ be the unique null line in $U$, denoted by $\operatorname{rad}(U)$. Let
\[
H=\operatorname{Stab}_{\GR}(U)
=\operatorname{Stab}_{\GR}(N\subset U).
\]
Then $\mathcal S_1\cong \GR/H.$ Let $\mathcal N
=
\left\{
N\in\mathbb P(\mathbb C^{2,q+1})\mid
\langle\bullet,\ast\rangle|_N=0
\right\}$. The projective null cone
\(\mathcal N\) is a real hypersurface in \(\mathbb P(\mathbb C^{2,q+1})\) of
dimension
$2q+3$. The map
\[
\begin{aligned}
    \mathcal S_1&\longrightarrow\mathcal N\\
U&\longmapsto\operatorname{rad}(U),
\end{aligned}\]
is a \(\GR\)-equivariant fiber bundle. 
For a null line \(N\), the quotient \(N^\perp/N\) has signature
\((1,q)\), and the fiber over \(N\) is the space of positive lines
in \(N^\perp/N\). Thus $\mathcal S_1\to\mathcal N$
has fiber \(\mathbb {CH}^q\), and
\[
\dim_{\mathbb R}\mathcal S_1
=
(2q+3)+2q
=
4q+3.
\]

The closed stratum
\[
\mathcal S_0
=
\left\{
W\in\operatorname{Gr}_2(\mathbb C^{2,q+1})\mid
\langle\bullet,\ast\rangle|_W=0
\right\}
\]
is the Grassmannian of totally isotropic \(2\)-planes. If
\(W_0\in\mathcal S_0\), \(P_0=\operatorname{Stab}_G(W_0)\) is a parabolic subgroup, then
\[
\mathcal S_0\cong G/P_0.
\] Consequently, $\dim_{\mathbb R}\mathcal S_0=4q.$
It is the Shilov boundary of the space
\(\mathcal S_2\).

Finally, the closure of \(\mathcal S_2\) in
\(\operatorname{Gr}_2(\mathbb C^{2,q+1})\) is
\[
\overline{\mathcal S_2}
=
\left\{
U\in\operatorname{Gr}_2(\mathbb C^{2,q+1})\mid
\langle\bullet,\ast\rangle|_U\geqslant0
\right\}
=
\mathcal S_2\sqcup\mathcal S_1\sqcup\mathcal S_0.
\]
Within this closure, \(\mathcal S_2\) is the open dense orbit,
\(\mathcal S_1\) is the intermediate boundary orbit, and
\(\mathcal S_0\) is the unique closed orbit. Moreover,
\[
\overline{\mathcal S_1}
=
\mathcal S_1\sqcup\mathcal S_0.
\]

\subsection{Developing map}
Consider the non-maximal cyclic Higgs bundle $(\mathcal E, \Phi)$ as in Section \ref{subsec:duality}. 
Since \((\mathcal E,\mathrm{D})\) is projectively flat, the induced
connection on \(\mathbb P(\wedge^2\mathcal E)\) is flat.  After lifting to
\(\widetilde X\), the projectively flat connection gives a flat
trivialization
\[
\mathbb P(\wedge^2\widetilde{\mathcal E})
\cong
\widetilde X\times \mathbb P(\wedge^2\mathbb C^{2,q+1}).
\]

For \(U\in \overline{\mathcal S_{2}}\), let
\[
[\omega_U]\in \mathbb P(\wedge^2 U)
\subset
\mathbb P(\wedge^2\mathbb C^{2,q+1})
\]
be the Plücker point of \(U\).  Via the above projectively flat
trivialization, \([\omega_U]\) determines a projectively flat section
\[
\widetilde X\longrightarrow \mathbb P(\wedge^2\widetilde{\mathcal E}),
\]
which we still denote by \([\omega_U]\). At \(x\in \widetilde X\), decompose
\[
\widetilde{\mathcal E}_x
=
\mathcal E_0(x)\oplus \mathcal E_1(x)
\oplus \mathcal E_2(x)\oplus \mathcal E_3(x).
\] This induces a decomposition of \(\wedge^2\widetilde{\mathcal E}_x\).

Suppose that $(\mathcal E, \Phi)$ is of type \upRoman{1}, that is, $\mathcal E_i$'s are line bundles except $\mathcal E_2$. Put
\[\mathcal{Q}=\mathcal E_0\oplus\mathcal E_3,\quad \WW=\mathcal E_1\oplus\mathcal E_2.\] 

Suppose $(\mathcal E, \Phi)$ is of type \upRoman{2}, that is, $\mathcal E_i$'s are line bundles except $\mathcal E_3$. Put \[\mathcal{Q}=\mathcal E_1\oplus\mathcal E_2,\quad \WW=\mathcal E_0\oplus\mathcal E_3.\]
Let $\Pi_{\mathcal{Q}}
:
\wedge^2\widetilde{\mathcal E}
\longrightarrow
\wedge^2\mathcal{Q}$ be the corresponding projection. Choose any local nonzero lift
$\widehat\omega_U(x)\in \wedge^2\widetilde{\mathcal E}_x$
of the projective class \([\omega_U](x)\).  Define
\[
s_U(x)
:=
\Pi_{\mathcal{Q}}
\bigl(\widehat\omega_U(x)\bigr).
\]
Although \(s_U(x)\) depends on the choice of lift
\(\widehat\omega_U(x)\), the condition $s_U(x)=0$
does not. Indeed, replacing \(\widehat\omega_U\) by
\(g\widehat\omega_U\), where \(g\) is a nowhere-zero function, replaces
\(s_U\) by \(g s_U\).

For \(b\in\{2,1,0\}\), define
\[
\widetilde{\mathcal D}_b
:=
\left\{
(x,U)\in \widetilde X\times \mathcal S_b
\;\middle|\;
s_U(x)=0
\right\}.
\]

The developing maps are the projections to the second factor:
\[
\begin{aligned}
    \Dev_b\colon\wt {\mathcal D_b}&\longrightarrow \mathcal S_b,\qquad
(x,U)&\longmapsto U,
\end{aligned}
\]

The construction is \(\pi_1(X)\)-equivariant.  The lifted decomposition is \(\pi_1(X)\)-equivariant: for every
\(\gamma\in\pi_1(X)\),
\[
\rho(\gamma)\mathcal E_i(x)=\mathcal E_i(\gamma x).
\]
Consequently the induced projections on \(\wedge^2\widetilde{\mathcal E}\)
are equivariant.
Therefore \(\widetilde{\mathcal D}_b\) is invariant under the diagonal action
\[
\gamma\cdot(x,U):=(\gamma x,\rho(\gamma)U).
\]
Hence the quotient
$\mathcal D_b
:=
\pi_1(X)\backslash \widetilde{\mathcal D}_b$ 
is well defined, and the projection
$\widetilde{\mathcal D}_b\to \widetilde X$ 
descends to a map $\mathcal D_b\longrightarrow X.$

\begin{theorem}
\label{thm:p1-geometric-structures}
\begin{enumerate}[label=(\roman*)]
    \item The developing map $\Dev_b:\wt {\mathcal D_b}\to \mathcal S_b$
    is a diffeomorphism onto an open subset \(\Omega_{\rho,b}\). 
    \item $\Dev_{2}$ is surjective onto $\mathcal S_{2}$.

    \item The action of $\rho(\pi_1(X))$ on $\Omega_{\rho,b}$ is properly discontinous.
\end{enumerate}
\end{theorem}
\begin{proof}
We use the existence, uniqueness of zeros of $f_U$ and the Morse function property to show the desired properties of $\Dev_b$.

We first prove injectivity for \(\Dev_b\). Suppose $\Dev_b(x_1,U_1)=\Dev_b(x_2,U_2).$
Then $U_1=U_2=:U.$
Since \((x_i,U)\in \wt {\mathcal D_b}\), we have
$s_U(x_i)=0,$ for $i=1, 2$. Equivalently,
$f_U(x_i)=0$, for $i=1, 2$. 
Because $U\in \mathcal S_b$, \prettyref{coro:indexfS} and Theorem \ref{thm:TypeIIf} imply that \(f_U\) has at most one  zero.
Therefore $x_1=x_2.$ Thus \(\Dev_b\) is injective.

Next we prove the local diffeomorphism of $\Dev_b$. We argue by the inverse function theorem. Consider the map
\[
\mathcal F_b:\widetilde X\times \mathcal S_b
\longrightarrow
\wedge^2Q_x
\]
defined by
$\mathcal F_b(x,U)
=
s_U(x)$.
By definition, $\widetilde{\mathcal D}_b=\mathcal F_b^{-1}(0).$
Let
$(x,U)\in \widetilde{\mathcal D}_b.$
Thus \(s_U(x)=0\). 
Suppose that the Higgs bundle is of type \upRoman{1}. Let \(e_0,e_1,e_2^1,\cdots, e_2^q, e_3\) be a local frame adapted to the splitting
\[
\mathcal E
=
\mathcal E_0\oplus \mathcal E_1\oplus \mathcal E_2\oplus \mathcal E_3.
\]
Write
\[
\omega_U
=
\sum_{i,j\neq 2, i<j} a_{ij}\,e_i\wedge e_j+\sum_{1\leq i\leq q, j=0,1,3} b_{ij}\,e_2^i\wedge e_j.
\]
Then
\[
s_U(x)=a_{03}(x,U)\,e_0(x)\wedge e_3(x).
\]
 We use the nondegeneracy of
$f_U=\|s_U\|^2=|a_{03}|^2$
at the zeros of $a_{03}$. Write locally $a_{03}=u+iv.$
Then $f_U=u^2+v^2.$
Let \(x\) be a zero of \(a_{03}\).  Thus $u(x)=v(x)=0.$
In particular \(\dd(f_U)_x=0\), and a direct computation gives
\[
\operatorname{Hess}_x(f_U)
=
2\bigl(\dd u_x\otimes \dd u_x+\dd v_x\otimes \dd v_x\bigr).
\]
Equivalently, for every \(V\in T_x\widetilde X\),
\[
\operatorname{Hess}_x(f_U)(V,V)
=
2\left|\dd_xa_{03}(V)\right|^2.
\]

Suppose that \(V\in T_x\widetilde X\) satisfies $\dd_xa_{03}(V)=0.$
Then
\[
\operatorname{Hess}_x(f_U)(V,V)=0.
\]
Since \(f_U\) is nondegenerate at \(x\) from Lemma \ref{lemma:p1-morse}, the quadratic form
\(\operatorname{Hess}_x(f_U)\) is nondegenerate.  Therefore
$V=0.$
Thus \(\dd_xa_{03}\) is injective.  Since both \(T_x\widetilde X\) and
\(\mathbb C\) are real two-dimensional vector spaces, it follows that
$\operatorname{rank}_{\mathbb R}(\dd_xa_{03})=2.$ Thus the differential in the
\(\widetilde X\)-direction
\[
\dd_xs_U:
T_x\widetilde X
\longrightarrow
\mathcal E_0(x)\wedge \mathcal E_3(x)
\]
is an isomorphism of real vector spaces at zeros of $s_U$.

In particular \(\dd\mathcal F_b\) is
surjective at \((x,U)\), and therefore \(\widetilde{\mathcal D}_b\) is a
smooth submanifold near \((x,U)\). The tangent space at \((x,U)\) is given by
\[
T_{(x,U)}\widetilde{\mathcal D}_b
=
\left\{
(\xi,\eta)\in T_x\widetilde X\oplus T_U\mathcal S_b
\;\middle|\;
\dd_xs_U(\xi)+\dd_Us_U(\eta)=0
\right\}.
\]
The differential of \(\Dev_b\) is the projection $\dd\Dev_b(\xi,\eta)=\eta.$

We now show that this map is an isomorphism.  Let
\(\eta\in T_U\mathcal S_b\).  Since \(\dd_xs_U\) is an isomorphism, there is a
unique vector \(\xi\in T_x\widetilde X\) satisfying
\[
\dd_xs_U(\xi)+\dd_Us_U(\eta)=0,
\]
namely
\[
\xi
=
-\bigl(\dd_xs_U\bigr)^{-1}\bigl(\dd_Us_U(\eta)\bigr).
\]
Hence \((\xi,\eta)\in T_{(x,U)}\widetilde{\mathcal D}_b\), and $\dd\Dev_b(\xi,\eta)=\eta.$
Thus \(\dd\Dev_b\) is surjective. To prove injectivity, suppose $\dd\Dev_b(\xi,\eta)=0.$
Then $\eta=0$.  Since $(\xi,0)\in T_{(x,U)}\widetilde{\mathcal D}_b$, we
have $\dd_xs_U(\xi)=0.$
Because \(\dd_xs_U\) is an isomorphism, this implies \(\xi=0\).  Therefore $(\xi,\eta)=(0,0),$
so \(\dd\Dev_b\) is injective.

Consequently
\[
\dd\Dev_b:
T_{(x,U)}\widetilde{\mathcal D}_b
\longrightarrow
T_U\mathcal S_b
\]
is an isomorphism.  Since \((x,U)\) was arbitrary, the inverse function
theorem implies that
$\Dev_b:\widetilde{\mathcal D}_b\longrightarrow \mathcal S_b$
is a local diffeomorphism.

Suppose the Higgs bundle $(\mathcal E,\Phi)$ is of type \upRoman{2}. The proof is similar to the type \upRoman{1} case. Mainly, the local diffeomorphism property of $\Dev_b$ follows from the property $f_U=\|s_U\|^2$ being nondegenerate at zeros. 

Finally, we show \(\Dev_{2}\) is surjective. Let $U\in \mathcal S_{2}$ be a positive \(2\)-plane. By \prettyref{lemma:p1-morse}, the function
$f_U(x)=\norm{s_U(x)}^2$
has a zero
$x_U\in \wt X.$
Therefore $s_U(x_U)=0,$
so $(x_U,U)\in \wt {\mathcal D_b}.$
Thus $\Dev_{2}(x_U,U)=U.$
Since \(U\in\mathcal S_{2}\) was arbitrary, \(\Dev_{2}\) is surjective on $\mathcal S_{2}$. Define the open domain
\[
\Omega_{\rho,b}:=\Dev_b(\wt {\mathcal D_b}).
\] Then $\Omega_{\rho,2}$ is the whole space $\mathcal S_{2}$. 
We first note that the $\pi_1(X)$-action on $\widetilde{\mathcal D_b}$ is properly
discontinuous. Indeed, if $K\subset \widetilde{\mathcal D_b}$ is compact, then its projection 
$p(K)$ to $\widetilde X$ is compact, and
\[
\{\gamma\in\pi_1(X)\mid\gamma K\cap K\neq\varnothing\}
\subset
\{\gamma\in\pi_1(X)\mid\gamma p(K)\cap p(K)\neq\varnothing\}.
\]
The latter set is finite because the deck action of $\pi_1(X)$ on
$\widetilde X$ is properly discontinuous.

Since $\Dev_b$ is injective and a local diffeomorphism, it identifies
$\widetilde{\mathcal D_b}$ diffeomorphically with $\Omega_{\rho,b}$. And this proves that the action of $\rho(\pi_1(X))$ on
$\Omega_{\rho,b}$ is properly discontinuous.
\end{proof}

\subsection{Topology of fiber bundles}
In this subsection, we will first show 
\[
\mathcal D_b:=\pi_1(X)\backslash \wt {\mathcal D_b}\to X
\] is a fiber bundle, and then study the fiber and the topology of the fiber bundle. 
\begin{lemma}
\[
s_U(x)=0
\quad\Longleftrightarrow\quad
U\cap \mathcal W_x\neq\{0\}.
\]
\end{lemma}
\begin{proof}
$s_U(x)=0$ is equivalent to the projection of $U$ to $\mathcal{Q}_x$ is at most of $1$-dimensional, which is equivalent to $U\cap W_x$ is of at least $1$-dimensional.
\end{proof}
Fix \(x_*\in\widetilde X\), and put
$W_*=\mathcal W(x_*)$. Set
\[
K=\operatorname{Stab}_{\PGR}(W_*)\cong \mathrm{P}(\mathrm{U}_{1,q}\times \mathrm{U}_{1,1}),
\]
where $\PGR=\PU_{2,q+1}$. Moreover,
\[
P
=
\pi_1(X)\backslash
\left\{
(x,g)\in\widetilde X\times \PGR\mid
gW_*=W_x
\right\},
\]
where
$\gamma\cdot(x,g)
=
(\gamma x,\rho(\gamma)g),$
is a principal \(K\)-bundle over \(X\).  For $b\in \{0,1,2\}$, define the space
\[
F_b
=
\left\{
U\in\mathcal S_b\mid
U\cap W_*\neq\{0\}
\right\}.
\]
Then \(K\) acts on \(F_b\) by \(k\cdot U=kU\), and we can define the associated
bundle $P\times_KF_b$.
\begin{proposition}\label{prop:FiberBundle}
1. There is a natural diffeomorphism
\[
\mathcal D_b
\cong
P\times_KF_b.
\]
2. The fibers have diffeomorphism types
\[
F_2\cong\mathbb R^{4q+2},
\qquad
F_0\cong S^{2q-1}\times S^{2q-1},
\qquad
F_1\cong S^{4q+1}\setminus j_0(S^{2q-1}\times S^{2q-1}),
\]
where $j_0\colon S^{2q-1}\times S^{2q-1}\to S^{4q+1}$ denotes the standard embedding $(u,w)\mapsto(0,u,w)/\sqrt{2}$.
\end{proposition}

\begin{proof} The map
\[
P\times_KF_b\longrightarrow\mathcal D_b,
\qquad
[[x,g],U]_K\longmapsto[x,gU]
\]
is well defined because
\[
U\cap W_*\neq\{0\}
\iff
gU\cap W_x\neq\{0\}.
\]
It is a diffeomorphism: every plane in the fiber over \(x\) is
carried to \(F_b\) by \(g^{-1}\), while changing \(g\) amounts
precisely to the \(K\)-equivalence relation.

Suppose that $(\mathcal E,\Phi)$ is of type \upRoman{1}. Write
\[
V=N\oplus P,
\qquad
N=\mathcal E_1(x_*)\oplus\mathcal E_3(x_*),
\qquad
P=\mathcal E_0(x_*)\oplus\mathcal E_2(x_*).
\]
Suppose that $(\mathcal E,\Phi)$ is of type \upRoman{2}. Write
\[
V=N\oplus P,
\qquad
N=\mathcal E_0(x_*)\oplus\mathcal E_2(x_*),
\qquad
P=\mathcal E_1(x_*)\oplus\mathcal E_3(x_*).
\]
In both cases, $\rank(N)=2, \rank(P)=q+1$. Every non-negative \(2\)-plane is uniquely the graph of a linear
map $A\colon N\to P$,
and
\[
\operatorname{graph}(A)\in\mathcal S_b
\iff
A^*A\leqslant I_2,
\qquad
\operatorname{rk}(I_2-A^*A)=b.
\]

If $(\mathcal E,\Phi)$ is of type \upRoman{1}, then \(F_b\) is the corresponding rank stratum in
\[
\left\{
A\colon N\to P\mid
A(\EE_1(x_*))\subset\mathcal E_2(x_*)
\right\}
\cong\mathbb C^{2q+1}.
\]
Write \(A=\begin{pmatrix}
    A_1&A_2
\end{pmatrix}\), where
\[
A_1\in\mathcal E_2(x_*)\cong \mathbb C^q,
\qquad
A_2\in P\cong \mathbb C^{q+1}.
\]
If $(\mathcal E,\Phi)$ is of type \upRoman{2}, then \(F_b\) is the corresponding rank stratum in
\[
\left\{
A\colon N\to P\mid
A(\EE_0(x_*))\subset\mathcal E_3(x_*)
\right\}
\cong\mathbb C^{2q+1}.
\]
Write \(A=\begin{pmatrix}
A_1&A_2
\end{pmatrix}\), where
\[
A_1\in\mathcal E_3(x_*)\cong \mathbb C^q,
\qquad
A_2\in P\cong \mathbb C^{q+1}.
\]

With respect to an adapted basis, we can let
\[
A=
\begin{pmatrix}
0 & a\\
u & v
\end{pmatrix},
\qquad
a\in\C,\quad u,v\in\C^q.
\]
Then
\[
A^*A=
\begin{pmatrix}
\|u\|^2 & u^*v\\
v^*u & |a|^2+\|v\|^2
\end{pmatrix}.
\]

We characterize the parameter spaces
\[
F_b
:=
\left\{
(a,u,v)\in\C\times\C^q\times\C^q: I_2\geqslant A^*A,
\rank(I_2-A^*A)=b
\right\},
\qquad b=0,1,2.
\]

Case 1: $b=0$. We have
$\rank(I_2-A^*A)=0$
if and only if
$I_2-A^*A=0.$
Therefore
\[
F_0
=
\left\{
(a,u,v):
\|u\|=1,\quad
u^*v=0,\quad
|a|^2+\|v\|^2=1
\right\}.
\]

For each point of \(F_0\), define
$w=au+v.$
Since \(u^*v=0\),
\[
\|w\|^2
=
|a|^2\|u\|^2+\|v\|^2
=
|a|^2+\|v\|^2
=
1.
\]
Thus the map
\[\begin{aligned}
    \Psi\colon F_0
&\longrightarrow
S^{2q-1}\times S^{2q-1}\\
(a,u,v)&\longmapsto(u,au+v)
\end{aligned}
\]
is a diffeomorphism. Its inverse is
\[
\Psi^{-1}(u,w)
=
\left(
u^*w,\,
u,\,
w-(u^*w)u
\right).
\]
Consequently,
$F_0
\cong
S^{2q-1}\times S^{2q-1}$.

Case 2: $b=2$. We have \(I_2-A^*A>0\). It is the intersection of $\|A\|_{op}<1$ and the vector space $\mathbb C^q\times \mathbb C^{q+1}$. Thus it is convex and 
$F_2\cong\mathbb R^{4q+2}.$

Case 3: $b=1$. The radial projection identifies the boundary of $F_2$ with a sphere:
$\partial F_2\cong S^{4q+1}.$

The boundary decomposes as
$\partial F_2=F_0\sqcup F_1,$
where
\[
F_0=\{A:A^*A=I_2\}.
\]

Recall the diffeomorphism $\Psi$ and its inverse $\Psi^{-1}\colon S^{2q-1}\times S^{2q-1}\to F_0$
defined by
\[
(u,w)\longmapsto
\left(u^*w,\ u,\ w-(u^*w)u\right).
\]
Since
\[
|a|^2+\|u\|^2+\|v\|^2
=
\operatorname{tr}(A^*A)=2
\]
on \(F_0\), radial projection sends \(F_0\) to the embedded submanifold
\[
j(u,w)
=
\frac{1}{\sqrt 2}
\left(u^*w,\ u,\ w-(u^*w)u\right)
\subset S^{4q+1}.
\]
This embedding is smoothly isotopic to the standard embedding
\[
j_0(u,w)=\frac{1}{\sqrt 2}(0,u,w).
\]
Indeed, define
\[
j_t(u,w)
=
\frac{1}{\sqrt 2}
\left(
\sin\left(\frac{\pi t}{2}\right)u^*w,\,
u,\,
w-\left(1-\cos\left(\frac{\pi t}{2}\right)\right)(u^*w)u
\right),
\qquad 0\leqslant t\leqslant 1.
\]
One can readily verify that \(j_t(u,w)\in S^{4q+1}\), while
\[
j_0(u,w)=\frac{1}{\sqrt 2}(0,u,w),
\qquad
j_1(u,w)=j(u,w).
\]
Hence, by the isotopy extension theorem,
\[
F_1
=
\partial F_2\setminus F_0
\cong
S^{4q+1}\setminus
j_0\bigl(S^{2q-1}\times S^{2q-1}\bigr).
\]
\end{proof}

\begin{proposition}\label{prop:topology of fiber bundles}
Let $d_i=\deg \EE_i$. Note that $d_1=d_0-(2g-2)$. The bundle
$\mathcal D_b\cong P\times_KF_b$
is determined, as an associated bundle of the principal $K$-bundle $P$, by the degrees
$[d_0,d_1,d_2,d_3]\in \mathbb{Z}^4/\mathbb{Z}\cdot(1,1,q,1)$ for $\tau=\upRoman{1}$ and $[d_0,d_1,d_2,d_3]\in \mathbb{Z}^4/\mathbb{Z}\cdot(1,1,1,q)$ for $\tau=\upRoman{2}$. 

Moreover, since $F_2\cong\mathbb R^{4q+2}$, the projection
$\mathcal D_2\rightarrow X$
is a homotopy equivalence.
\end{proposition}

\begin{proof}
The group \(K\) deformation retracts onto
$\mathrm{P}\big(\mathrm{U}_q\times (\mathrm{U}_1)^3\big),$
and principal \(K\)-bundles over the surface \(X\) are therefore
classified by
$H^2(X;\pi_1(K)).$
For \(P\), the four components are the first Chern classes of
\(\EE_0,\EE_1,\EE_2,\EE_3\). Thus the statement follows.
\end{proof}

\subsection{Proof of Theorem \ref{thm:intro-geometry}}
Part 1 and 2 follow from Theorem \ref{thm:embedding} and Theorem \ref{thm:embedding1}. Part 3 follows from Theorem \ref{thm:p1-geometric-structures} and Proposition \ref{prop:FiberBundle}. \qed

\subsection{Proof of Theorem \ref{thm:intro-domains}}
The statement follows from Theorem \ref{thm:p1-geometric-structures}, Proposition \ref{prop:FiberBundle} and Proposition \ref{prop:topology of fiber bundles}.
\qed

\section{\texorpdfstring{$\partial$}{partial}-alternating surfaces}\label{sec:AlternatingSurface}

\subsection{Complex pseudo-hyperbolic spaces}

Let $\C^{m,n+1}$ be equipped with a Hermitian form $H_0$ of signature
$(m,n+1)$, with $m,n\geqslant 1$. The complex pseudo-hyperbolic space of
signature $(m,n)$ is
\[
  \CH^{m,n}:=\left\{[v]\in\mathbb P(\C^{m,n+1}) \mid H_0(v,v)<0\right\}.
\]
Let $\Tcal\to\CH^{m,n}$ denote the tautological negative line bundle, which is a subbundle of the trivial vector bundle with rank $m+n+1$.
For a negative line $L\subset\C^{m,n+1}$, the orthogonal complement
$L^{\perp_{H_0}}$ is nondegenerate of signature $(m,n)$, and there is a
canonical identification
\[
  T^{1,0}_{[L]}\CH^{m,n}\cong\Hom(L,L^{\perp_{H_0}}).
\]
The induced pseudo-Hermitian metric on $T^{1,0}\CH^{m,n}$ is denoted
$\Hsczero$ and has signature $(m,n)$. The projective isometry group is
$\PU_{m,n+1}$ or $\PU_{n+1,m}$. 

Throughout, $X$ denotes a Riemann surface, not necessarily
compact, and $\KK_X$ denotes its canonical bundle.

Let $f\colon X\to\CH^{m,n}$ be a smooth map. Set
\[
  L_f:=f^*\Tcal,
  \qquad
  \Vf:=f^*T^{1,0}\CH^{m,n}
      \cong \Hom(L_f,L_f^{\perp_{H_0}}).
\]
We denote by $\nabla^{\Vf}$ the connection on $\Vf$ induced by the
pseudo-K\"ahler Levi-Civita connection.

With respect to the decompositions
$f^*T\CH^{m,n}\otimes\C=\Vf\oplus\overline{\Vf}$ and
$TX\otimes\C=T^{1,0}X\oplus T^{0,1}X$, the complexified differential
has components
\[
  \partial^{1,0}f\in\Omega^{1,0}(\Vf),
  \qquad
  \bar\partial^{1,0}f\in\Omega^{0,1}(\Vf),
\]
together with their complex conjugates $\bar\partial^{0,1}f$, $\partial^{0,1}f$ in
$\Omega^{0,1}(\overline{\Vf})$ and $\Omega^{1,0}(\overline{\Vf})$ respectively.
Using the pseudo-Hermitian metric $\Hsczero$ on $\Vf$, define
\[
  \delta_f:=\left(\bar\partial^{1,0}f\right)^{*_{\Hsczero}}
  \in\Omega^{1,0}(\Vf^\vee),
\]
where $\Vf^\vee$ is the dual bundle of $\Vf$. When $f$ is harmonic, the sections $\partial^{1,0}f$ and $\delta_f$ are
holomorphic sections of the holomorphic bundle $\K_X\otimes\VV_f$ and $\KK_X\otimes\VV_f^\vee$, respectively.
The Hopf differential is
\[
  q_2:=\delta_f\left(\partial^{1,0}f\right)\in H^0(X,\K_X^2).
\]
The map $f$ is conformal if and only if $q_2=0$. A conformal harmonic
spacelike immersion is a spacelike minimal immersion, that is, the mean curvature vanishes. In the sequel,
when minimality is included in a definition, it means conformal and
harmonic with respect to the given complex structure on $X$.

Let $f\colon X\to\CH^{m,n}$ be an immersion such that $\Hsczero|_{f_*T^{1,0}X}$ is positive definite. Whenever $\partial^{1,0}f$ is nowhere zero, set
\[
  \Tp:=\langle\partial^{1,0}f\rangle\subset\Vf.
\]
For such a map define
\[
  \sigma^{2,0}:=
  \left((\nabla^{\Vf})^{1,0}\partial^{1,0}f\right)^{\perp_{\Tp}}
  \in\Omega^0(\KK_X^2\otimes\Vf),
\]
where $\perp_{\Tp}$ denotes the $\Hsczero$-orthogonal projection to
$\Tp^{\perp_{\Hsczero}}\subset\Vf$. When $f$ is conformal and harmonic,
this is the usual $(2,0)$-part of the second fundamental form.

For a conformal harmonic immersion, the Chern--Wolfson cubic differential
is locally defined by
\[
  q_3=
  \Hsczero\!\left(
    \nabla^{\Vf}_Z\bigl(\partial^{1,0}f(Z)\bigr),
    \bar\partial^{1,0}f(\bar Z)
  \right)dz^3=\delta_f\circ \sigma^{2,0},
  \qquad Z=\frac{\partial}{\partial z}.
\]
This expression is independent of the holomorphic coordinate $z$: under
a coordinate change, the possible extra term is proportional to the Hopf
differential $q_2$, and hence vanishes for conformal maps. Since
$\CH^{m,n}$ is a pseudo-K\"ahler complex space form, the Chern--Wolfson
moving-frame computation gives that $q_3$ is holomorphic for minimal
surfaces; see \cite{chern1983minimal}.

\subsection{\texorpdfstring{$\partial$-alternating}{partial-alternating} surfaces}\label{sec:def}

We work in $\CH^{m,n}$ with $m\geqslant 2$ and $n\geqslant 1$.

\begin{definition}\label{def:alt}
Let $f\colon X\to\CH^{m,n}$ be a minimal immersion, i.e. a conformal harmonic immersion with respect to the given complex
structure on $X$. We call $f$ a \emph{$\partial$-alternating surface} if $\Vf$
admits an $\Hsczero$-orthogonal splitting by smooth complex subbundles
\[
  \Vf=\Tp\oplus\Nm\oplus\Np,
\]
where $\Tp$ is positive of rank $1$, $\Nm$ is negative of rank $n$, and
$\Np$ is positive of rank $m-1$, satisfying:
\begin{enumerate}[label=(A\arabic*)]
\item $\partial^{1,0}f$ is nowhere zero and takes values in $\Tp$;
      equivalently, $\Tp=\langle\partial^{1,0}f\rangle$.
\item $\sigma^{2,0}\in\Omega^{0}(\KK_X^2\otimes\Nm)$.
\item  $(\nabla^{\Vf})^{1,0}\left(\Omega^0(\Nm)\right)
  \subset \Omega^{1,0}(\Nm\oplus\Np).$
\item $\delta_f|_{\Tp\oplus \Nm}=0$.
\item $(\nabla^{\Vf})^{1,0}\left(\Omega^0(\Np)\right)
  \subset \Omega^{1,0}(\Np\oplus \Tp).$
\end{enumerate}
The data in (A2)--(A4) define smooth bundle maps 
\[
  \alpha\colon\Tp\longrightarrow\Nm\otimes \K_X,
  \qquad
  \beta\colon\Nm\longrightarrow\Np\otimes \K_X,
  \qquad
  \gamma\colon\Np\longrightarrow \K_X,
\]
where $\alpha$ is represented by $\sigma^{2,0}$, $\beta$ is the
$\Np$-component of $(\nabla^{\Vf})^{1,0}|_{\Nm}$, and
$\gamma=\delta_f|_{\Np}$. Each of these maps may vanish identically.
\end{definition}
Here $\delta_f|_{\Tp}=0$ follows from the conformality of $f$.

\begin{proposition}\label{prop:q2q3-vanish}
If $f$ is $\partial$-alternating, then $q_3=0.$
\end{proposition}
\begin{proof}
Let $Z=\partial/\partial z$ be a local holomorphic vector field. Since
$\Tp$ is generated by $\partial^{1,0}f$, we can write
\[
  \nabla^{\Vf}_Z\bigl(\partial^{1,0}f(Z)\bigr)
  =
  \lambda\,\partial^{1,0}f(Z)+\sigma^{2,0}(Z,Z)
\]
for a local function $\lambda$. The first term lies in $\Tp$, and by
(A2) the second term lies in $\Nm$. Therefore
\[
  \nabla^{\Vf}_Z\bigl(\partial^{1,0}f(Z)\bigr)
  \in \Tp\oplus\Nm .
\]
By (A4), $\delta_f$ annihilates $\Tp\oplus\Nm$, equivalently
$\Tp\oplus\Nm$ pairs trivially with
$\bar\partial^{1,0}f(\bar Z)$. Thus
\[
  q_3=
  \Hsczero\!\left(
    \nabla^{\Vf}_Z\bigl(\partial^{1,0}f(Z)\bigr),
    \bar\partial^{1,0}f(\bar Z)
  \right)dz^3
  =0.
\]
\end{proof}

\subsection{The flatness equations}

In this subsection let $f$ be a spacelike minimal immersion satisfying
(A1)--(A4) of Definition~\ref{def:alt}; for the moment we do not assume
(A5). Pull back the tautological line bundle and the ambient trivial bundle to $X$, we have a complex vector bundle over $X$,
\[
  E:=X\times\C^{m,n+1}
\]
comes with the flat connection $\mathrm{D}$ and the constant Hermitian form $H_0$.
The splitting of $\Vf$ gives an $H_0$-orthogonal decomposition
\[
  E=L_f\oplus L_T\oplus V\oplus W,
  \qquad
  L_T:=L_f\otimes\Tp,
  \quad
  V:=L_f\otimes\Nm,
  \quad
  W:=L_f\otimes\Np.
\]
By (A1), $\partial^{1,0}f$ gives a smooth isomorphism between
$\Tp$ and $\KK_X^{-1}$. Below we order the summands as
\[
  E_0=L_f,
  \qquad
  E_1=L_T,
  \qquad
  E_2=V,
  \qquad
  E_3=W.
\] The signs of the four summands are $-,+,-+$ respectively.

Let $h$ be the positive-definite Hermitian metric obtained from $H_0$ by
changing the sign on the two negative summands. Equivalently,
\[
  H_0(u,v)=h(Su,v),
  \qquad
  S=\diag(-1,+1,-1,+1).
\]

\begin{lemma}\label{lem:signed-adjoint}
Let $A\colon E_j\to E_i\otimes\KK_X$ be a bundle map. Then
$A^{*_{H_0}}=s_i s_j\,A^{*_h}.$
Consequently, if $\theta$ denotes the off-diagonal $(1,0)$-part of $\mathrm{D}$
and $\theta^\dagger$ denotes the corresponding adjoint of $\theta$ by $H_0$ together with conjugation on $1$-forms,
then $(\theta^\dagger)_{ji}=s_i s_j\,(\theta_{ij})^{*_h}.$

In particular, if $\theta$ only contains $(i,j)$-entries such that $|i-j|=1$ or $3$, then $\theta^{\dagger}=-\theta^{*_h}$.
\end{lemma}

\begin{proof}
For $u\in E_j$ and $v\in E_i$,
\[
  H_0(Au,v)=s_i h(Au,v)=s_i h(u,A^{*_h}v),
\]
whereas
\[
  H_0(u,A^{*_{H_0}}v)=s_j h(u,A^{*_{H_0}}v).
\]
Since $s_i,s_j\in\{\pm1\}$, the formula follows. 
\end{proof}

Let $\nabla^{H_0}$ be the block-diagonal Chern connection associated with
$H_0$ and with the holomorphic structures on the four summands induced by
projecting the $(0,1)$-part of $\mathrm{D}$ onto each summand. Since changing signs do not change Chern connection, we also write $\nabla^{H_0}$ as $\nabla^h$, as it is the Chern connection uniquely determined by the holomorphic structure and the Hermitian metric $h$. We will use the curl letters $\LL_f$, $\LL_T$, $\VV$, $\WW$, $\mathcal{T}_+$, $\mathcal{N}_-$, $\mathcal{N}_+$ to denote the corresponding holomorphic bundles over $X$. Then
\begin{equation}\label{eq:structure}
\mathrm{D}=\nabla^h+\theta-\theta^\dagger
\end{equation}
since this decomposition is orthogonal.

Write $\theta_{ij}\colon E_j\to E_i\otimes\KK_X$ for the off-diagonal
$(1,0)$-blocks. Conformality and Conditions (A1)--(A3) give
\[\theta=\begin{pmatrix}
0&0&z&\gamma\\
\id&0&0&x\\
0&\alpha&0&y\\
0&0&\beta&0
\end{pmatrix},\qquad
\theta^{\dagger}=\begin{pmatrix}
0&\id^\dagger&0&0\\
0&0&\alpha^\dagger&0\\
z^\dagger&0&0&\beta^\dagger\\
\gamma^\dagger&x^\dagger&y^\dagger&0
\end{pmatrix}.\] 

Condition (A4) means $z=0$. Condition (A5) means $y=0$.

Since $\mathrm{D}^2=0$, we have \begin{equation}
 F(\nabla^h)+ \dbar_{\End(E)}\theta
  -\partial^h_{\End(E)}\theta^\dagger
  -(\theta\wedge\theta^\dagger)
  -(\theta^\dagger\wedge\theta)=0.
\end{equation}
That is,
\begin{equation}\label{eq:offdiag-master}
\scalebox{0.8}{$
\begin{pmatrix}
*&-\partial(\id^\dagger)-\gamma\wedge x^\dagger
&\bar\partial z-\gamma\wedge y^\dagger
&\bar\partial\gamma-\id^\dagger\wedge x-z\wedge\beta^\dagger\\
\bar\partial(\id)-x\wedge\gamma^\dagger
&*
&-\partial(\alpha^\dagger)-x\wedge y^\dagger
&\bar\partial x-\alpha^\dagger\wedge y\\
-\partial(z^\dagger)-y\wedge\gamma^\dagger
&\bar\partial\alpha-y\wedge x^\dagger
&*
&\bar\partial y-\partial(\beta^\dagger)-z^\dagger\wedge\gamma\\
-\partial(\gamma^\dagger)-x^\dagger\wedge\id-\beta\wedge z^\dagger
&-\partial(x^\dagger)-y^\dagger\wedge\alpha
&-\partial(y^\dagger)+\bar\partial\beta-\gamma^\dagger\wedge z
&*
\end{pmatrix}
$}=0.\end{equation}

\begin{lemma}\label{lem:basic-holomorphicity}
Assume that $f$ is a $\partial$-spacelike minimal immersion satisfying
(A1)--(A4). Then, in the notation above,
\[
  x=0, \qquad z=0,
  \qquad
  \bar\partial(\id)=0,
  \qquad
  \bar\partial\alpha=0,
  \qquad
  \bar\partial\gamma=0,
\]
and
\[
  \gamma\wedge y^\dagger=0,
  \qquad
  \alpha^\dagger\wedge y=0,
  \qquad
  \bar\partial\beta=\partial(y^\dagger).
\]
 In particular
$\id$, $\alpha$ and $\gamma$ are holomorphic.  

If moreover, (A5) holds, i.e., $y=0$, hence $\beta$ is holomorphic as well; equivalently,
$\theta$ is holomorphic. Moreover, 
\[\theta=\begin{pmatrix}
0&0&0&\gamma\\
\id&0&0&0\\
0&\alpha&0&0\\
0&0&\beta&0
\end{pmatrix},\]
and $\mathrm{D}=\nabla^h+\theta+\theta^{*_h}.$
\end{lemma}

\begin{proof}
With respect to the splitting $\Vf=\Tp\oplus\Nm\oplus\Np,$ the induced $(0,1)$-connection on $\Vf$ has the form
\[
  (\nabla^{\Vf})^{0,1}s
  =
  (\nabla^h)^{0,1}s-
  \begin{pmatrix}
    0 & \alpha^\dagger & 0\\
    0 & 0 & \beta^\dagger\\
    x^\dagger & y^\dagger & 0
  \end{pmatrix}s.
\]
Since $f$ is harmonic,
\[
  0=(\nabla^{\Vf})^{0,1}(\partial^{1,0}f)
  = (\nabla^{\Vf})^{0,1}\begin{pmatrix}
    \id\\ 0\\ 0
  \end{pmatrix}=
  \begin{pmatrix}
    \bar\partial(\id)\\
    0\\
    -x^\dagger\wedge\id
  \end{pmatrix}.
\]
Because $\id$ is nowhere zero by (A1), this gives
\[
  \bar\partial(\id)=0,
  \qquad
  x=0.
\]

Condition (A4) is exactly $z=0$. Therefore, reading the $(1,3)$ and $(1,4)$-entries of Equation (\ref{eq:offdiag-master}),
\[
  \bar\partial\gamma=0,
  \qquad
  \gamma\wedge y^\dagger=0.
\]

Reading the $(2,4)$, $(3,2)$ and $(4,3)$-entries of Equation (\ref{eq:offdiag-master}),
\[
  \bar\partial\alpha-y\wedge x^\dagger=0,
  \qquad
  \bar\partial x -\alpha^\dagger\wedge y=0,
  \qquad
  \bar\partial\beta-\partial(y^\dagger)-\gamma^\dagger\wedge z=0.
\]
Since $x=0$ and $z=0$, these reduce to
\[
  \bar\partial\alpha=0,
  \qquad
  \alpha^\dagger\wedge y=0,
  \qquad
  \bar\partial\beta=\partial(y^\dagger).
\]
Thus $\alpha$ is holomorphic. If (A5) holds, then $y=0$, and the last
identity gives $\bar\partial\beta=0$. Hence $\beta$ is holomorphic, and
therefore all nonzero blocks of $\theta$ are holomorphic.

The last claim follows from Lemma \ref{lem:signed-adjoint} and $\theta^{\dagger}=-\theta^{*_h}$.
\end{proof}

\begin{lemma}\label{lem:A5-automatic-low-rank}
Assume that $f$ is a $\partial$-spacelike minimal immersion satisfying
(A1)--(A4). Under one of the following conditions, 
\begin{enumerate}[label=($\partial$\roman*)]
\item $\rank(\Nm)=1$ and $\alpha\not\equiv0$,
\item $\rank(\Np)=1$ and $\gamma\not\equiv0$,
\end{enumerate}
then $y\equiv 0$.
Consequently, condition (A5) is automatic
in the signatures
\[
  (m,n)=(q+1,1)
  \qquad\text{or}\qquad
  (m,n)=(2,q).
\]
Moreover, the splitting is uniquely determined by $f$.

Moreover, in Case ($\partial$\romannumeral1), one can replace (A4) condition by $q_3=0$.
\end{lemma}

\begin{proof}
By Lemma~\ref{lem:basic-holomorphicity}, $\alpha$ and $\gamma$ are
holomorphic, and
\[
  \alpha^\dagger\wedge y=0,
  \qquad
  \gamma\wedge y^\dagger=0.
\]

In either cases, either $\alpha$ or $\gamma$ are holomorphic sections of line bundles which is not constantly zero. Hence $y=0$ everywhere.

\smallskip\emph{Uniqueness of the splitting in case (i).}\enspace
With $n=1$, $\Nm$ has rank $1$. By (A2), $\sigma^{2,0}$ takes values in
$\K_X^{2}\otimes\Nm$, and the hypothesis $\alpha\not\equiv 0$ makes
$\sigma^{2,0}\not\equiv 0$. Hence $\Nm$ is the saturation of the image line
of $\sigma^{2,0}$ in $\Tp^{\perp_{\Hsc_0}}$ away from its isolated zero set,
and is determined by $f$. The summand $\Np$ is then the
$\Hsczero$-orthogonal complement of $\Tp\oplus\Nm$ in $\Vf$, also
determined by $f$.

\smallskip\emph{Uniqueness of the splitting in case (ii).}\enspace
With $m=2$, $\Np$ has rank $1$. By (A4), $\delta_f$ vanishes on
$\Tp\oplus\Nm$ and is generically nonzero on $\Np$, so
$\ker\delta_f=\Tp\oplus\Nm$ on the open set where $\delta_f\neq 0$.
This subbundle is nondegenerate of signature $(1,n)$, so its
$\Hsc_0$-orthogonal complement is a positive line, which by the splitting
is $\Np$. Since $\delta_f$ is holomorphic, $\Np$ extends uniquely
across the isolated zeros of $\delta_f$ by saturation. The summand
$\Nm$ is then forced by orthogonality.

To prove the final claim, assume $f$ satisfies (A1)--(A3) and $q_3=0$, and that $\rank(\Nm)=1$. As shown in the proof of Lemma~\ref{lem:basic-holomorphicity}, harmonicity and (A1) alone imply $x=0$. Flatness then gives $\bar\partial\alpha=0$, so $\alpha$ is holomorphic. The cubic differential is given by $q_3 = \delta_f\bigl((\nabla^{\Vf})^{1,0}\partial^{1,0}f\bigr)$. Therefore $q_3 = z \alpha$. Since $q_3=0$ and $\alpha \not\equiv 0$ is a holomorphic section of a line bundle, $z$ must vanish identically, which is exactly (A4).
\end{proof}

\begin{remark}\label{rmk:special-bake}
Outside these signatures, the preceding argument no longer makes (A5) automatic, and the splitting becomes part of the data.
\end{remark}

\begin{definition}
Under the condition ($\partial$\romannumeral1), we call $f$ \emph{non-degenerate} if $f$ does not lie in a totally geodesic copy of $\CH^{q, 1}$. \\
Under the condition ($\partial$\romannumeral2), we call $f$ \emph{non-degenerate} if $f$ does not lie in a totally geodesic copy of $\CH^{2, q-1}$. \\
Under both conditions ($\partial$\romannumeral1)($\partial$\romannumeral2), we call $f$ \emph{regular} if either $f$ does not lie in a totally geodesic copy of $\CH^{2, 1}$ or $\beta$ has at least one zero. 
\end{definition}

\subsection{Special cases}
\begin{proposition}
Let $f\colon X\to\CH^{m,n}$ be a $\partial$-alternating spacelike immersion, with associated 
splitting $\Vf=\Tp\oplus\Nm\oplus\Np$ and data 
$\alpha,\beta,\gamma$. The following are equivalent:
\begin{enumerate}[label=(\roman*)]
\item $\gamma\equiv 0$;
\item $f$ is a holomorphic curve.
\end{enumerate}
\end{proposition}

\begin{proposition}\label{prop:degenerate-totally-geodesic}
Let $f\colon X\to\CH^{m,n}$ be a $\partial$-alternating spacelike
immersion, with associated splitting
$\Vf=\Tp\oplus\Nm\oplus\Np$ and data
$\alpha,\beta,\gamma$. The following are equivalent:
\begin{enumerate}[label=(\roman*)]
\item $\alpha\equiv 0$ and $\beta\equiv 0$;
\item there exists a totally geodesic complex submanifold
  $Y\subset\CH^{m,n}$, biholomorphically isometric to the complex
  hyperbolic space $\CH^{m,0}=\mathbf{H}^{m}_{\C}$, such that
  $f(X)\subset Y$ and $T^{1,0}Y|_{f(X)}=\Tp\oplus\Np$.
\end{enumerate}
Under these equivalent conditions, the sub-bundle
$\Tp\oplus\Np\subset\Vf$ is parallel for the pullback Levi-Civita
connection $\nabla^{\Vf}$, and the remaining datum $\gamma$ encodes
the intrinsic geometry of $f$ as a spacelike minimal immersion
$X\to\mathbf{H}^{m}_{\C}$.
\end{proposition}
\begin{proof}
\textbf{(i)$\Rightarrow$(ii).}\enspace We first show that
$E_+:=\Tp\oplus\Np$ is $\nabla^{\Vf}$-parallel, and then use parallelism to produce the totally geodesic
$Y$.

\emph{Parallelism.}\enspace
For $\Tp$: by (A2), $\sigma^{2,0}$ takes values in $\KK_X^{2}\otimes\Nm$,
so $(\nabla^{\Vf})^{1,0}\Tp$ has no $\Np$-component, and its
$\Nm$-component is $\alpha\equiv 0$. Hence
\[(\nabla^{\Vf})^{1,0}\Tp\subset\Tp\otimes \K_X.\] For $\Nm$: by (A3),
\[(\nabla^{\Vf})^{1,0}\Nm\subset(\Nm\oplus\Np)\otimes \K_X,\] and its
$\Np$-component is $\beta_T\equiv 0$, so
\[(\nabla^{\Vf})^{1,0}\Nm\subset\Nm\otimes\K_X.\] For $\Np$: by (A5),
\[(\nabla^{\Vf})^{1,0}\Np\subset\Np\otimes \K_X.\] Taking adjoint with respect to $\Hsczero$, we obtain that $\nabla^{\Vf}$ preserves
each summand. In particular $E_+=\Tp\oplus\Np$ is parallel, and so is
$E_+^{\perp_{H_0}}=\Nm$. We denote by $\EE_+$ the corresponding holomorphic subbundle.

\emph{Construction of $Y$.}\enspace
The sub-bundle $\EE_+$ is parallel, holomorphic, and $H_0$-positive of
complex rank $m$, and by (A1) contains
$\langle\partial^{1,0}f\rangle=\Tp$. Lift $f$ locally to
$\tilde f\colon X\to\C^{m,n+1}\setminus\{0\}$ with image in the
negative cone, and identify $\Vf$ with
$\Hom(L_f,L_f^{\perp_{H_0}})$. Fibrewise over $f(X)$, $\EE_+$
corresponds to a complex $(m+1)$-dimensional subspace
\[
  W_x\;:=\;\tilde f(x)\;\oplus\;(E_+)_x\;\subset\;\C^{m,n+1}
\]
of signature $(m,1)$. Parallelism of $E_+$, together with
\[\partial^{1,0}f\in\mathrm{H}^0(\mathcal{T}_+\otimes\K_X)\subset \mathrm{H}^0(\EE_+\otimes \K_X)\] and the analogous
$(0,1)$-statement, implies that $W_x$ is independent of $x\in X$; call
it $W$. Set
\[
  Y\;:=\;\bigl\{[v]\in\PP(W)\;:\;H_0(v,v)<0\bigr\}\;\subset\;\CH^{m,n}.
\]
Since $H_0|_W$ has signature $(m,1)$, $Y$ is biholomorphically
isometric to $\CH^{m,0}=\mathbf{H}^{m}_{\C}$, and totally geodesic in
$\CH^{m,n}$ as the fixed-point set of the holomorphic isometric
involution that is $+1$ on $W$ and $-1$ on $W^{\perp_{H_0}}$. By
construction $f(X)\subset Y$ and
\[T^{1,0}Y|_{f(X)}=E_+=\Tp\oplus\Np.\]

\textbf{(ii)$\Rightarrow$(i).}\enspace Assume $f$ factors through a
totally geodesic $Y\cong\mathbf{H}^{m}_{\C}\subset\CH^{m,n}$ with
\[T^{1,0}Y|_{f(X)}=\Tp\oplus\Np.\] Then
$\Nm=(T^{1,0}Y)^{\perp_{H_0}}$ is the normal bundle of $Y$ along
$f(X)$. Total geodesy makes both $T^{1,0}Y$ and its $H_0$-orthogonal
complement parallel along $Y$ for the ambient Levi-Civita connection;
pulling back, $\Tp\oplus\Np$ and $\Nm$ are parallel under
$\nabla^{\Vf}$. In particular \[(\nabla^{\Vf})^{1,0}\Tp\subset(\Tp\oplus\Np)\otimes \K_X,\]
so its $\Nm$-component $\alpha$ vanishes, and
$(\nabla^{\Vf})^{1,0}\Nm\subset\Nm\otimes K$, so its $\Np$-component
$\beta$ vanishes.

\smallskip
The pullback argument of (ii)$\Rightarrow$(i) also establishes the
parallelism of $\Tp\oplus\Np$ asserted in the "moreover" clause.
Finally, under these conditions $f$ is a spacelike minimal immersion
into $Y\cong\mathbf{H}^{m}_{\C}$; the splitting of $\Vf$ restricts on
$Y$ to the splitting of $T^{1,0}Y$ into the tangent line $\Tp$ and its
$H_0$-orthogonal complement $\Np$ in $T^{1,0}Y$, and
$\gamma\colon\Np\to \K_X$ is the restriction of $\delta_f$ to $\Np$,
which is precisely the $(0,1)$-part of $\dd f$ relative to the tangent
line of $f$ inside $Y$.
\end{proof}

From the following, we see that the $\partial$-alternating surfaces in $\CH^{2,q}$ are natural extensions of maximal surfaces in $\mathbb H^{2,q}$.
\begin{proposition}\label{prop:real-part}
Let $\mathbb H^{2,q}\subset\CH^{2,q}$ denote the totally geodesic real
form fixed by complex conjugation, and let
$f\colon X\to\mathbb H^{2,q}$ be a spacelike conformal immersion,
regarded also as an immersion into $\CH^{2,q}$.
A spacelike conformal immersion
$f\colon X\to\mathbb H^{2,q}$ is maximal if and only if, as a map into
$\CH^{2,q}$, it is $\partial$-alternating for the splitting
\[
  \Vf
  =
  \Tp\oplus\Nm\oplus\Np,
  \qquad
  \Tp=\C\,\dd f(T^{1,0}X),\quad
  \Nm=N_f^{\C},\quad
  \Np=\C\,\dd f(T^{0,1}X),
\]
where $N_f$ is the real normal bundle of $f$ in $\mathbb H^{2,q}$.
\end{proposition}
\begin{proof}
Let $M=\mathbb H^{m,n}$ be the real form of $\CH^{m,n}$. The real form
is totally geodesic. Hence the Levi--Civita connection and second
fundamental form of an immersion into $M$ agree with the corresponding
objects obtained after viewing the immersion as an immersion into
$\CH^{m,n}$ and complexifying along $M$.

Choose a local holomorphic coordinate $z$ on $X$, put
$Z=\partial/\partial z$, and write
\[
  u=\dd f(Z),\qquad \bar u=\dd f(\bar Z).
\]
For a conformal spacelike immersion in $M$, the complex lines
$\C u$ and $\C\bar u$ are positive and $H_0$-orthogonal. Moreover, if
$N_f$ denotes the real normal bundle in $M$, then
\[
  \delta_f|_{\C u\oplus N_f^{\C}}=0,
  \qquad
  \delta_f(\bar u)(Z)=H_0(\bar u,\bar u)>0.
\]
Thus, for the real splitting, the map $\gamma_T$ is exactly the
projection of $\delta_f$ to the $\C\bar u$-factor.

We first treat $M=\mathbb H^{2,q}$. If $f$ is maximal, then, since
$M$ is totally geodesic in $\CH^{2,q}$, $f$ is harmonic as a map into
$\CH^{2,q}$. Therefore $\partial^{1,0}f$ is holomorphic, giving (A1).
The splitting
\[
  \Tp=\C u,\qquad
  \Nm=N_f^{\C},\qquad
  \Np=\C\bar u
\]
has signatures $+,-,+$ of ranks $1,q,1$. Since the second fundamental
form $B$ of an immersion into the real space form is normal,
$\sigma^{2,0}=B(Z,Z)\dd z^2$ takes values in $N_f^{\C}$, giving (A2).

Let $\eta$ be a local section of $N_f^{\C}$. The Weingarten formula
gives
\[
  \nabla^M_Z\eta=\nabla^\perp_Z\eta-A_\eta Z,
\]
where $A$ denotes the shape operator. Because $f$ is maximal, $B(Z,\bar Z)=0$. Hence
\[
  \langle A_\eta Z,\bar u\rangle
  =
  \langle B(Z,\bar Z),\eta\rangle
  =
  0.
\]
Thus $A_\eta Z$ is proportional to $\bar u$ and
\[
  \nabla^M_Z\eta\in N_f^{\C}\oplus\C\bar u,
\]
which is (A3). The identity for $\delta_f$ above gives (A4). Finally,
\[
  \nabla^M_Z\bar u
  =
  \dd f(\nabla^X_Z\bar Z)+B(Z,\bar Z)=\dd f(\nabla^X_Z\bar Z)
\]
and in a conformal coordinate one has
$\nabla^X_Z\bar Z=0$. Thus the line $\C\bar u$ is preserved.
Thus (A5) holds. Hence a maximal surface in $\mathbb H^{2,q}$ is
$\partial$-alternating in $\CH^{2,q}$.

Conversely, if a real immersion into $\mathbb H^{2,q}$ is
$\partial$-alternating as a map into $\CH^{2,q}$, then
by definition it is a spacelike minimal
surface in $\CH^{2,q}$. Since the real form is totally geodesic, its
mean curvature as a surface in $\mathbb H^{2,q}$ is the same. Hence
the real immersion is maximal.
\end{proof}

\subsection{Equivariance and equivalence}\label{sec:equiv}
\begin{definition}
An \emph{equivariant alternating surface} on $X$ is a pair $(f,\rho)$
where
\[
  \rho\colon\pi_1(X)\longrightarrow\PU_{m,n+1}
\]
is a group homomorphism and $f\colon\widetilde X\to\CH^{m,n}$ is a
$\rho$-equivariant $\partial$-alternating spacelike minimal immersion
in the sense of Definition~\ref{def:alt}; here equivariance means
\[
  f(\gamma\cdot p)\;=\;\rho(\gamma)\cdot f(p)\qquad
  \forall\gamma\in\pi_1(X),\ p\in\widetilde X.
\]
The alternating splitting $\Vf=\Tp\oplus\Nm\oplus\Np$ on $\widetilde
X$ is $\pi_1(X)$-equivariant.
\end{definition}

\begin{definition}
Two equivariant alternating surfaces $(f,\rho)$ and $(f',\rho')$ on
$X$ are \emph{equivalent} if there exists $g\in\PU_{m,n+1}$ such that
\[
  f'\;=\;g\circ f
  \qquad\text{and}\qquad
  \rho'(\gamma)\;=\;g\,\rho(\gamma)\,g^{-1}\quad\forall\gamma\in\pi_1(X).
\]
\end{definition}

\begin{definition}
Two equivariant alternating surfaces with splittings $(f,\rho, \Vf=\Tp\oplus\Nm\oplus\Np)$ and $(f',\rho',V_f'=\Tp'\oplus\Nm'\oplus\Np')$ on
$X$ are \emph{equivalent} if there exists $g\in\PU_{m,n+1}$ such that
\[
  f'\;=\;g\circ f
  \qquad\text{and}\qquad
  \rho'(\gamma)\;=\;g\,\rho(\gamma)\,g^{-1}\quad\forall\gamma\in\pi_1(X).
\] and $g$ preserves the splitting. 
\end{definition}

\subsection{Correspondence}\label{sec:base}
\begin{theorem}\label{thm:base}
On any Riemann surface $X$, there is a natural bijection between
isomorphism classes of:
\begin{enumerate}
\item $4$-cyclic $\PU_{m,n+1}$-harmonic bundles on $X$ as in
  \eqref{eq:fourcycle} (any of $\alpha,\beta,\gamma$ may vanish
  identically), modulo graded unitary gauge equivalence;
\item equivariant $\partial$-alternating immersions
  $f\colon\widetilde X\to\CH^{m,n}$ together with splitting $\Vf=\Tp\oplus\Nm\oplus\Np$, modulo equivalence preserving splittings.
\end{enumerate}
\end{theorem}

\begin{proof}
We give the construction in both directions. 

\medskip
\noindent\emph{From a harmonic bundle to a surface.}
Let
\[
(\EE,\Phi)=\begin{tikzcd}
	{\mathcal{L}} & {\mathcal{L}\mathcal{K}_X^{-1}} & {\mathcal{V}} & {\mathcal{W}}
	\arrow["{\mathds{1}}"', from=1-1, to=1-2]
	\arrow["\alpha"', from=1-2, to=1-3]
	\arrow["\beta"', from=1-3, to=1-4]
	\arrow["\gamma"', curve={height=19pt}, from=1-4, to=1-1]
\end{tikzcd}
\]
be a representative $4$-cyclic $\mathrm{PU}_{m,n+1}$-harmonic bundle with
diagonal harmonic metric $h$, and let
\[
    \Hsc=-h_\Lcal\oplus h_{\Lcal \K_X^{-1}}\oplus(-h_\Vcal)\oplus h_\Wcal .
\]
By Lemma~\ref{lem:signed-adjoint}, the connection
\[
    \mathrm{D}^h:=\nabla^h+\Phi+\Phi^{*h}
\]
preserves $\Hsc$, and the Hitchin's self dual equation gives 
\[F(\mathrm{D}^h)=-2\pi\iu\mu\cdot\omega\otimes\id_\EE.\]
After pulling back to the universal cover, $(\EE,\mathrm{D}^h)$ becomes a projectively flat Hermitian vector bundle
of the appropriate signature. The negative line $\Lcal$ then defines
\[
    f=[\Lcal]:\widetilde X\longrightarrow \CH^{m,n},
\]
which is equivariant for the monodromy representation of $\mathrm{D}^h$, viewed
projectively as $\rho\colon\pi_1(X)\to\PU_{m,n+1}$.

The orthogonal complement of $\Lcal$ is
\[
    \Lcal^\perp=\Lcal \K_X^{-1}\oplus \Vcal\oplus \Wcal,
\]
and therefore
\[
    f^*T^{1,0}\CH^{m,n}\cong\Hom(\Lcal,\Lcal^\perp)
      =\Tp\oplus \Nm\oplus \Np,
\]
where
\[
    \Tp:=\Hom(\Lcal,\Lcal \K_X^{-1})\cong \K_X^{-1},\qquad
    \Nm:=\Hom(\Lcal,\Vcal),\qquad
    \Np:=\Hom(\Lcal,\Wcal).
\]
The signs are $+,-,+$. The arrow $\mathds 1:\Lcal\to\Lcal \K_X^{-1}\otimes \K_X$
identifies with $\partial^{1,0}f\in H^0(\K_X\otimes \Tp)$, so $f$ is
unbranched. The arrows $\alpha$ and $\beta$ identify with
\[
    \alpha\colon\Tp\to \Nm\otimes \K_X,
    \qquad
    \beta\colon\Nm\to \Np\otimes \K_X.
\]
Moreover
\[
    \sigma^{2,0}=\alpha(\partial^{1,0}f)\in \mathrm{H}^0(\K_X^2\otimes\mathcal{N}_{-}),
\]
and $\gamma$ identifies with
\[
    \delta_f|_{\Np}\colon\Np\to \K_X.
\]
Since $\delta_f$ vanishes on $\Tp\oplus \Nm$, the Hopf differential
$q_2=\delta_f(\partial^{1,0}f)$ vanishes. Thus the harmonic map
defined by the negative line is conformal; being harmonic and
spacelike, it is a spacelike minimal immersion. The vanishing of the
off-diagonal Higgs blocks $\Phi_{24}$ and $\Phi_{34}$ in the
$4$-cyclic shape is precisely the surface-side condition (A5):
\[(\nabla^{V_f})^{1,0}\Np\subset(\Np\oplus\Tp)\otimes\K_X.\] The conditions (A1)--(A5) of
Definition~\ref{def:alt} are therefore all satisfied by the resulting
$(f,\Tp\oplus\Nm\oplus\Np)$.

Twisting $(\EE,\Phi,h)$ by a holomorphic line bundle $\mathcal F$
multiplies $f$ by a fibrewise unitary scalar in $\widetilde
X\times\C^{m,n+1}$; projectively this is the identity, so the
projective map $f\colon\widetilde X\to\CH^{m,n}$ depends only on the
$\PU_{m,n+1}$-class.

\medskip
\noindent\emph{From an equivariant $\partial$-alternating surface to a harmonic bundle.}
Conversely, suppose that $(f,\rho)$ is an equivariant $\partial$-alternating surface. Pulling back the projectivization of the tautological line bundle gives the $\pi_1(X)$-equivariant section of the flat bundle $\widetilde{X}\times\mathbb{P}(\C^{m,n+1})$, hence descends to a section $[\widetilde{L}_f]$ of the flat bundle $\widetilde{X}\times_{\rho}\mathbb{P}(\C^{m,n+1})$ over $X$. By \cite[Lemma 8]{deroin2022toledo}, since $\mathrm{H}^3(X;\mathbb{Z})$ vanishes, we can then lift it as a line subbundle $L_f$ in a projectively flat bundle $(E\to X,\mathrm{D})$, whose fibre is $\mathbb{C}^{m,n+1}$.

Use evaluation $L_f\otimes\Hom(L_f,L_f^\perp)\cong L_f^\perp$
to write
\[
    \EE=\Lcal_f\oplus (\Lcal_f\otimes \mathcal{T}_+)\oplus(\Lcal_f\otimes \mathcal{N}_-)
        \oplus(\Lcal_f\otimes\mathcal{N}_+).
\]
$\EE$ is the same as $E$ as a smooth complex bundle, but equipped with a holomorphic structure different from $\mathrm{D}^{0,1}$.
We set
\[
    \LL:=\LL_f,\qquad\Vcal:=\Lcal_f\otimes \mathcal{N}_-,
    \qquad
    \Wcal:=\Lcal_f\otimes \mathcal{N}_+,
\]
then $\Lcal \K_X^{-1}=\Lcal_f\otimes \mathcal{T}_+$ and
order these summands as $\EE_0=\Lcal$, $\EE_1=\Lcal \K_X^{-1}$, $\EE_2=\Vcal$, $\EE_3=\Wcal$, with
signs $-,+,-,+$. Let $h$ be obtained from the flat pseudo-Hermitian
form $\Hsc$ by changing sign on the negative summands $\Lcal$ and
$\Vcal$. Write $\mathrm{D}=\nabla^h+\theta-\theta^{\dagger}$ as \prettyref{eq:structure}. By Lemma \ref{lem:basic-holomorphicity}, we obtain that $\theta$ is holomorphic and $\theta^\dagger=-\theta^{*_h}$. 

The $(1,1)$-block-diagonal part of the projectively flatness of $\mathrm{D}$ yields that
\[
    F(\nabla^h)+[\theta,\theta^{*h}]=-2\pi\iu\mu\cdot\omega\otimes\id_\EE
\]
for the fixed K\"ahler form $\omega$ on $X$, which is Hitchin's self-dual equation. Therefore $(\EE,\theta,h)$ is a
$\mathrm{U}_{m,n+1}$-harmonic bundle of the required type. Note that two lifts $(E,\mathrm{D})$ are differed by a smooth Hermitian line bundle equipped with a compatible connection. Hence the resulting two $\mathrm{U}_{m,n+1}$-harmonic bundles are differed by an Hermitian holomorphic line bundle. Thus their images in the same $\PU_{m,n+1}$-equivalence class.

Finally, we prove that the graded unitary gauge equivalence of $4$-cyclic $\PU_{m,n+1}$-harmonic bundles is the same as the equivalence of $\partial$-alternating immersions. Suppose first that $u$ is a graded unitary equivalence between two harmonic bundles. Thus $u$ intertwines their holomorphic structures and Higgs fields and satisfies $u^*h_2=h_1$. By uniqueness of the Chern connection, $u\nabla^{h_1}=\nabla^{h_2}u$. Unitarity also gives $u\theta_1^{*h_1}=\theta_2^{*h_2}u$, hence $u$ also intertwines their flat connections. Because $u$ preserves the grading, it preserves the pseudo-Hermitian forms obtained by reversing the prescribed signs. Consequently, in projectively flat trivializations over $\widetilde{X}$, its projectivization is a constant element $g\in \PU_{m,n+1}$. Since $u$ preserves the grading, hence $g$ gives an equivalence between the corresponding $\partial$-alternating immersions $f_1,f_2$ on $\widetilde{X}$. Now $u$ descends to $X$ gives $\rho_2=\operatorname{Ad}(g)\circ\rho_1$. The proof of converse is similar.
\end{proof}


\subsection{Gauss maps}

Let
\[
V=\mathbb C^{m+n+1},\qquad
h_0(z,w)=z^*I_{m,n+1}w,\qquad
I_{m,n+1}=\operatorname{diag}(I_m,-I_{n+1}),
\]
so that \(h_0\) has signature \((m,n+1)\). The symmetric space of
\(\mathrm{U}_{m,n+1}\) admits the Grassmannian realization
\[
\mathcal X_{m,n+1}
=
\frac{\mathrm{U}_{m,n+1}}{\mathrm{U}_m\times\mathrm{U}_{n+1}}
\cong
\operatorname{Gr}_m^+(V),
\]
where
\[
\operatorname{Gr}_m^+(V)
=
\{P\subset V:\dim_{\mathbb C}P=m,\ h_0|_P>0\},
\]
via
\[
g(\mathrm{U}_m\times\mathrm{U}_{n+1})\longmapsto gP_0,
\qquad
P_0=\mathbb C^m\oplus0.
\]

Equivalently, \(\mathcal X_{m,n+1}\) is the space
\[
\mathcal H_{h_0}
=
\{M\in\operatorname{Herm}^+(V):C_M^2=\operatorname{Id}\},
\]
where \(C_M\in\operatorname{End}_{\mathbb C}(V)\) is defined by
$h_0(v,w)=M(v,C_Mw).$
Indeed, $C_M$ is an $M$-self-adjoint involution whose
$(+1)$- and $(-1)$-eigenspaces are respectively $h_0$-positive
and $h_0$-negative. The correspondence
$\operatorname{Gr}_m^+(V)\cong\mathcal H_{h_0}$ is given by
\[
P\longmapsto M_P,\qquad
M_P(v,w)
=
h_0(v_P,w_P)-h_0(v_N,w_N),
\]
where
\[
V=P\oplus P^{\perp_{h_0}},\qquad
v=v_P+v_N,\quad w=w_P+w_N.
\]
Its inverse sends $M$ to the $(+1)$-eigenspace of $C_M$.

For $P\in\operatorname{Gr}_m^+(V)$, the holomorphic tangent space is
naturally identified with
\[
T_P^{1,0}\operatorname{Gr}_m^+(V)
\cong
\operatorname{Hom}_{\mathbb C}(P,P^{\perp_{h_0}}).
\]
The invariant Hermitian metric is
\[
\mathbf h_P(A,B)
=
-\sum_{i=1}^m h_0(Ae_i,Be_i),
\qquad
A,B\in
\operatorname{Hom}_{\mathbb C}(P,P^{\perp_{h_0}}),
\]
where \(\{e_1,\ldots,e_m\}\) is any \(h_0\)-orthonormal basis of \(P\).
This expression is independent of the choice of basis. Since
$h_0|_{P^{\perp_{h_0}}}<0$, it defines a positive-definite Hermitian form.
Its real part is, up to an overall normalization constant, the
standard $\mathrm{U}_{m,n+1}$-invariant Riemannian metric on
$\mathcal X_{m,n+1}$.

Let $f:X\to\mathbb {CH}^{m,n}$ be a $\partial$-alternating immersion, and
let
\[
P_x=\LL_{f,x}\otimes(\mathcal T_{+,x}\oplus\mathcal N_{+,x})
\in\operatorname{Gr}_m^+(V)
\]
be the associated positive \(m\)-plane. Its Gauss map is
\[
\mathcal G_f:X\longrightarrow\operatorname{Gr}_m^+(V),
\qquad
\mathcal G_f(x)=P_x.
\]
Let \(M_x\) denote the harmonic metric determined by the corresponding
conformal harmonic bundle in Theorem \ref{thm:base}. By construction,
\[
C_{M_x}|_{P_x}=\operatorname{Id},
\qquad
C_{M_x}|_{P_x^{\perp_{h_0}}}=-\operatorname{Id};
\]
equivalently, \(M_x=M_{P_x}\). Hence, under the identification $\mathcal H_{h_0}\cong\operatorname{Gr}_m^+(V)$, the associated harmonic map coincides with the Gauss map $\mathcal G_f$.

In particular, its target metric is, up to an overall normalization,
the symmetric-space metric of $\mathrm{U}(m,n+1)$. In summary, we obtain the following proposition.

\begin{proposition}\label{prop:GaussMaps}Under the correspondence in Theorem \ref{thm:base}, 
the associated harmonic map to the cyclic harmonic bundle is the Gauss map $\mathcal G_f\colon \widetilde{X}\to\mathcal{X}_{m,n+1}$ of the $\partial$-alternating surface $f:\widetilde X\to\mathbb {CH}^{m,n}$. 
\end{proposition}

\subsection{Proof of Theorem \ref{thm:baseSpecialSignatureintro}}
For the correspondence for general surface, it follows from Theorem \ref{thm:base}. The Gauss map statement follows from Proposition \ref{prop:GaussMaps}.
\qed

\section{Stability and moduli space: closed surface case}\label{sec:Moduli}

In this section, we assume that $X$ is a closed Riemann surface of genus $g\geqslant2$. We apply the general theory of the moduli space of cyclic Higgs bundles studied in \prettyref{apdx:cyclic} to our case. 

Recall that a $4$-cyclic Higgs bundle of rank $(r_j)$ is a Higgs bundle of the form:
\[
    \begin{aligned}
    \EE&=\EE_{\overline{0}}\oplus\EE_{\overline{1}}\oplus\EE_{\overline{2}}\oplus\EE_{\overline{3}},\\
    \Phi&=\begin{pmatrix}
        &&&\varphi_{\overline{3}}\\
        \varphi_{\overline{0}}&&&\\
        &\varphi_{\overline{1}}&&\\
        &&\varphi_{\overline{2}}&
    \end{pmatrix},
\end{aligned}
\]
where $\EE_{\overline{j}}$ is the holomorphic vector bundle of rank $r_j$ and $\varphi_{\overline{j}}\colon\EE_{\overline{j}}\to\EE_{\overline{j+1}}\otimes\KK_X$. Then the moduli space of such cyclic Higgs bundles is identified with the moduli space $\mathcal{M}(X;G_0,\mathfrak{g}_{\overline{1}})$ of $(G_0,\mathfrak{g}_{\overline{1}})$-Higgs pairs over $X$ for suitable $G_0<\mathrm{GL}_r\mathbb{C}$ and $\mathfrak{g}_{\overline{1}}\subset\mathfrak{gl}_r\mathbb{C}$, where $\mathfrak{g}_{\overline{j}}$ arose from a $\mathbb{Z}/4\mathbb{Z}$-grading on $\mathfrak{g}=\mathfrak{gl}_r\mathbb{C}$. 

From now on, we assume that $r_0=r_1=1$, $r_2=n$, $r_3=m-1$ and either $m=2$ or $n=1$. We use $(\mathbb{E},\Phi)$ to denote the corresponding $(G_0,\mathfrak{g}_{\overline{1}})$-Higgs pair and $C^\bullet(\mathbb{E},\Phi)$ to denote the deformation complex. 

The space of infinitesimal automorphism of weight $\overline{j}$ is defined as:
\[\operatorname{aut}_{\overline{j}}(\EE,\Phi):=\{s\in\mathbb{E}[\mathfrak{g}_{\overline{j}}]\mid[s\wedge\Phi]=0\}=\left\{s\in\mathrm{H}^0(\operatorname{End}(\EE))\mid[s\wedge\Phi]=0, s(\EE_{\overline{k}})\subset\EE_{\overline{k+j}}\right\}.\]

\subsection{Smoothness and dimension of the moduli space}

Let $(\EE,\Phi)$ be a $4$-cyclic Higgs bundle of rank $(r_j)$. Then this defines a continuous and locally constant map
\[\begin{aligned}
    \mathbf{d}_{(G_0,\mathfrak{g}_{\overline{1}})}\colon\mathcal{M}(X;G_0,\mathfrak{g}_{\overline{1}})&\longrightarrow\mathbb{Z}^4\\
    [(\mathbb{E},\Phi)]&\longmapsto(\deg(\EE_{\overline{j}}))_{j=0}^3.
\end{aligned}\]
We define $\mathcal{M}_{(d_j)}(X;G_0,\mathfrak{g}_{\overline{1}})$ as the preimage of $(d_j)\in\mathbb{Z}^4$ along $\mathbf{d}_{(G_0,\mathfrak{g}_{\overline{1}})}$, which is the union of some connected components. We denote by 
\[\mathcal M_{(d_j), \tau}(X)\subset\mathcal{M}_{(d_j)}(X;G_0,\mathfrak{g}_{\overline{1}})\] the open subspace consisting of isomorphism classes of $4$-cyclic Higgs bundles $(\EE,\Phi)$ (with corresponding $(G_0,\mathfrak{g}_{\overline{1}})$-Higgs pair $(\mathbb{E},\Phi)$) over $X$ satisfying that 
\[\deg(\EE_{\overline{j}})=d_j\quad\mbox{ and }\quad\varphi_{\overline{0}}\colon\EE_{\overline{0}}\to\EE_{\overline{1}}\otimes\KK_X\mbox{ is an isomorphism},\]
where $\tau=$ \upRoman{1} when $m=2$ and $\tau=$ \upRoman{2} when $n=1$. Then the associated $d_0=d_1+(2g-2)$. We denote by 
\[\mathcal M_{(d_j), \tau}^{\mathrm{s}}(X)=\mathcal M_{(d_j), \tau}(X)\cap\mathcal{M}^{\mathrm{s}}(X;G_0,\mathfrak{g}_{\overline{1}})\]
its stable part. We denote by $\mathcal M_{\tau}(X)$ and $\mathcal M_{\tau}^{\mathrm{s}}(X)$ their union among all $(d_j)$'s respectively. We will use $\mathcal {M}_{(d_j), \tau}^{\star}(X)$, $\mathcal {M}_{(d_j), \tau}^{\mathrm{s},\star}(X)$, $\mathcal M_{\tau}^{\star}(X)$ and $\mathcal M_{\tau}^{\mathrm{s},\star}(X)$ to denote their non-maximal parts, which are also open subspaces, respectively. Also we will omit the $X$ in the notation if there is no abuse of notion. 

\begin{proposition}\label{prop:H2vanish}
    Suppose that $[(\mathbb{E},\Phi)]\in\mathcal M_{\tau}^{\mathrm{s}}$, then $\operatorname{aut}_{\overline{j}}(\EE,\Phi)=0$ for $\overline{j}=\overline{1},-\overline{1}$. As a corollary, the second hypercohomology group $\mathbb{H}^2(C^\bullet(\mathbb{E},\Phi))=0$. If we further assume that $[(\mathbb{E},\Phi)]\in\mathcal M_{\tau}^{\mathrm{s},\star}$, then we also have $\operatorname{aut}_{\overline{2}}(\EE,\Phi)=0$.
\end{proposition}

\begin{proof}
    By \prettyref{prop:Schur}, any nonzero $s\in\operatorname{aut}_{-\overline{1}}(\EE,\Phi)$ or $\operatorname{aut}_{\overline{1}}(\EE,\Phi)$ is an isomorphism. Hence all $\EE_{\overline{j}}$ are isomorphic. This yields that $\EE_{\overline{0}}\cong\EE_{\overline{1}}$, which contrasts with $\varphi_0\colon\EE_{\overline{0}}\to\EE_{\overline{1}}\otimes\KK_X$ is an isomorphism, since $\KK_X$ is non-trivial. Therefore, $s$ must be $0$. By \prettyref{prop:dual}, we obtain that $\mathbb{H}^2(C^\bullet(\mathbb{E},\Phi))=0$.

    Now we assume that $[(\mathbb{E},\Phi)]\in\mathcal M_{\tau}^{\mathrm{s},\star}$ and let $s\in\operatorname{aut}_{\overline{2}}(\EE,\Phi)$. By \prettyref{prop:Schur} again, if $s\neq0$, then $s$ is an isomorphism and it forces that $\varphi_{\overline{2}}$ is an isomorphism, and this implies that $[(\EE,\Phi)]$ is maximal by \prettyref{lemma:higgsfieldnorm} (also by a direct degree computation), contradiction. Hence $\operatorname{aut}_{\overline{2}}(\EE,\Phi)=0$.
\end{proof}

\begin{proposition}\label{prop:smooth}
    If $[(\mathbb{E},\Phi)]\in\mathcal M_{\tau}(X)^{\mathrm{s}}$, then a neighborhood of $[(\mathbb{E},\Phi)]$ in the moduli space $\mathcal{M}(G_0,\mathfrak{g}_{\overline{1}})$ is the complex linear space $\mathbb{H}^1(C^\bullet(\mathbb{E},\Phi))$. As a corollary, $\mathcal M_{\tau}(X)^{\mathrm{s}}$ is smooth.
\end{proposition}

\begin{proof}
    By \prettyref{prop:H2vanish}, $\mathbb{H}^2(C^\bullet(\mathbb{E},\Phi))=0$ and by \prettyref{coro:automorphism}, the automorphism group $\operatorname{Aut}(\mathbb{E},\Phi)=\mathbb{C}^*\cdot\operatorname{id}_{\EE}$ acts trivially on $\mathbb{H}^1(C^\bullet(\mathbb{E},\Phi))$ by adjoint. Hence a neighborhood of $[(\EE,\Phi)]$ in the moduli space $\mathcal{M}(G_0,\mathfrak{g}_{\overline{1}})$ is the complex linear space $\mathbb{H}^1(C^\bullet(\mathbb{E},\Phi))$, which is smooth.
\end{proof}
The following proposition gives the expected dimension of $\mathcal{M}(G_0,\mathfrak{g}_{\overline{1}})$.
\begin{proposition}\label{prop:Cdim}
    If $[(\mathbb{E},\Phi)]\in\mathcal{M}_{(d_j),\tau}^{\mathrm{s}}$, then \[\begin{aligned}
    &\dim_{\mathbb{C}}T_{[(\mathbb{E},\Phi)]}\mathcal{M}(G_0,\mathfrak{g}_{\overline{1}})\\
    =&1+(2+n^2+(m-1)^2+m(n+1))(g-1)
        +(m-2)(d_0-d_2)+(1-n)(d_1-d_3).
    \end{aligned}\]
\end{proposition}

\begin{proof}
    Note that
    \[\begin{aligned}    \dim_{\mathbb{C}}T_{[(\mathbb{E},\Phi)]}\mathcal{M}(G_0,\mathfrak{g}_{\overline{1}})
        =&\dim_{\mathbb{C}}\mathbb{H}^1(C^\bullet(\mathbb{E},\Phi))\quad\mbox{(by \prettyref{prop:smooth})}\\
        =&\dim_{\mathbb{C}}\mathbb{H}^0(C^\bullet(\mathbb{E},\Phi))-\chi(C^\bullet(\mathbb{E},\Phi))+\dim_{\mathbb{C}}\mathbb{H}^2(C^\bullet(\mathbb{E},\Phi))\\
        =&1-\chi(C^\bullet(\mathbb{E},\Phi))\quad\mbox{(by \prettyref{prop:Schur} and \prettyref{prop:H2vanish})}\\
        =&1-\chi(\mathbb{E}[\mathfrak{g}_0])+\chi(\mathbb{E}[\mathfrak{g}_{\overline{1}}]\otimes\KK_X)\quad\mbox{(by \prettyref{eq:LES})}
    \end{aligned}\]
    where $\chi$ denotes the Euler characteristic. It suffices to compute $\chi(\mathbb{E}[\mathfrak{g}_0])$ and $\chi(\mathbb{E}[\mathfrak{g}_{\overline{1}}]\otimes\KK_X)$.

    Since \[\begin{aligned}
        \chi(\mathbb{E}[\mathfrak{g}_0])
        =&\sum_{j=0}^3\chi(\operatorname{End}(\EE_{\overline{j}}))\\
        =&\sum_{j=0}^3r_j^2(1-g)+(r_jd_j-r_jd_j)\quad\mbox{(by Riemann--Roch theorem)}\\
        =&(1-g)(1^2+1^2+n^2+(m-1)^2)\\
        =&(1-g)(2+n^2+(m-1)^2)
    \end{aligned}\]
    and
    \[\begin{aligned}
        &\chi(\mathbb{E}[\mathfrak{g}_{\overline{1}}]\otimes\KK_X)\\
        =&\sum_{j=0}^3\chi(\operatorname{Hom}(\EE_{\overline{j}},\EE_{\overline{j+1}})\otimes\KK_X)\\
        =&\sum_{j=0}^3r_jr_{j+1}(1-g)+(r_jd_{j+1}-r_{j+1}d_j)+r_jr_{j+1}\cdot\deg(\KK_X)\quad\mbox{(by Riemann--Roch theorem)}\\
        =&(1\cdot1+1\cdot n+n\cdot (m-1)+(m-1)\cdot 1)(g-1)\\&+(d_1-d_0)+(d_2-n\cdot d_1)+(n\cdot d_3-(m-1)\cdot d_2)+((m-1)\cdot d_0-d_3)\\
        =&m(n+1)(g-1)+(m-2)d_0+(1-n)d_1+(2-m)d_2+(n-1)d_3.
    \end{aligned}\]
    Therefore, 
    \[\begin{aligned}
        \dim_{\mathbb{C}}T_{[(\mathbb{E},\Phi)]}\mathcal{M}(G_0,\mathfrak{g}_{\overline{1}})
        =&1-\chi(\mathbb{E}[\mathfrak{g}_0])+\chi(\mathbb{E}[\mathfrak{g}_{\overline{1}}]\otimes\KK_X)\\
        =&1+(2+n^2+(m-1)^2+m(n+1))(g-1)\\
        &+(m-2)(d_0-d_2)+(1-n)(d_1-d_3).
    \end{aligned}\]
\end{proof}

\begin{remark}\label{rem:tensorline}
    There is an isomorphism between 
$\mathcal{M}_{(d_j),\tau}$ and $\mathcal{M}_{(d_j'),\tau}$ when
$(d_j)-(d_j')\in\mathbb{Z}\cdot(1,1,n,m-1)$ given by tensoring a line bundle.
\end{remark}

\subsection{Non-emptyness of \texorpdfstring{$\mathcal M_{(d_i),\tau}$}{M(di)tau}}\label{sec:stability}
In this subsection, we still assume that either $m=2$ or $n=1$. Consider a $4$-cyclic Higgs bundle $(\mathcal E,\Phi)$ either of type \upRoman{1} or of type \upRoman{2}, with $\rank(\EE)=q+3$,
\[
d_j=\deg \EE_{\overline{j}},\qquad \varphi_{\overline{0}}\colon\EE_{\overline{0}}\longrightarrow\EE_{\overline{1}}\KK_X\text{ isomorphism},\]
then 
\[d_1=d_0-2(g-1),
\qquad \mu:=\mu(\EE)=\frac{d_0+d_1+d_2+d_3}{q+3}.
\] 

If $(\EE,\Phi)$ is of type \upRoman{1}, then its ranks are $(1,1,q,1)$. Set \[
u_{\upRoman{1}}=d_1-d_3,
\qquad
v_{\upRoman{1}}=d_0-d_2+(q-1)\mu.
\]
Similarly, for a $4$-cyclic Higgs bundle $(\mathcal E,\Phi)$ of type \upRoman{2} with ranks $(1,1,1,q)$, set
\[
u_{\upRoman{2}}=d_2-d_0,
\qquad
v_{\upRoman{2}}=d_3-d_1-(q-1)\mu.
\]

For each $\tau$, let $\mathscr D_\tau$ denote the set of degree vectors
satisfying \begin{equation}\label{eq:a}-4(g-1)\leqslant u_\tau\leqslant v_\tau\leqslant  4(g-1);\end{equation}
let $\mathscr D_\tau^{\mathrm{s}}$ denote the set of degree vectors
satisfying  
\begin{equation}\label{eq:b}\begin{cases}
-4(g-1)\leqslant u_\tau\leqslant v_\tau\leqslant4(g-1), & q=1,\\[1mm]
-4(g-1)\leqslant u_\tau<v_\tau<4(g-1), & q\geqslant2.
\end{cases}\end{equation}

Consider a $4$-cyclic Higgs bundle of type \upRoman{1}. If $u_\upRoman{1}=d_1-d_3=-4(g-1)$ and $\varphi_{\overline{3}}\colon\EE_{\overline{3}}\rightarrow\EE_{\overline{0}}\KK_X$ is an isomorphism, then $(\mathcal E, \Phi)$ is also a Higgs bundle of type \upRoman{2}. For $q\geqslant 2$, this is the equivalent condition for a stable Higgs bundle is of both types. 

Consider a $4$-cyclic Higgs bundle $(\mathcal E,\Phi)$ of type \upRoman{2}, that is, the ranks are $(1,1,1,q)$ and degrees are $(d_0,d_1,d_2,d_3)$ and slope $\mu$.
By considering its dual Higgs bundle, we obtain a $4$-cyclic Higgs bundle of type \upRoman{2} of degree $(d_i^\vee)$ with slope $\mu^\vee=-\mu$, where 
\[d_0^\vee=-d_1, \quad d_1^\vee=-d_0,\quad d_2^\vee =-d_3, \quad d_3^\vee=-d_2.\] Then
\[u_\upRoman{1}(\EE^\vee,\Phi^\vee)=d_1^\vee-d_3^\vee=d_2-d_0=u_\upRoman{2}(\EE,\Phi)\]
and similarly $v_\upRoman{1}(\EE^\vee,\Phi^\vee)=v_\upRoman{2}(\EE,\Phi)$. The following lemma holds directly.

\begin{lemma}\label{lem:DualDegree}
A $4$-cyclic Higgs bundle $(\mathcal E,\Phi)$ of type \upRoman{2} with degree vector satisfies $(d_i)\in \mathscr D_{\upRoman{2}}$ (resp. $\mathscr D_{\upRoman{2}}^s)$ if and only if its dual Higgs bundle $(\mathcal E^\vee, \Phi^\vee)$ with degree vector satisfies $(d_i^\vee)\in \mathscr D_{\upRoman{1}}$ (resp. $\mathscr D_{\upRoman{1}}^s)$.
\end{lemma}

\begin{remark}

We can explicitly calculate the the number of the set of degree vectors in $\mathscr D_\tau^{\mathrm{s}}$ (resp. $\mathscr D_\tau$) and the number of the subset of degree vectors in $\mathscr D_\tau^{\mathrm{s}}$ (resp. $\mathscr D_\tau$) satisfying $d_0+d_1+d_2+d_3=(q+3)\mu$.  See Lemma \ref{lem:calculationDs} for details.
\end{remark}

\begin{lemma}\label{lem:StableNecessary}
Let $(\mathcal E,\Phi)$ be a $4$-cyclic Higgs bundle of type $\tau$. Then, $(\mathcal E,\Phi)$ is polystable only if $(d_j)\in \mathscr D_\tau$;
$(\mathcal E,\Phi)$ is stable only if $(d_i)\in \mathscr D_\tau^{\mathrm{s}}$.

\begin{itemize}
    \item If $\tau=\upRoman{1}$ and $(d_i)\in \mathscr D_\tau\setminus\mathscr D_\tau^{\mathrm{s}}$, then $ \mathcal E_{\overline{2}}= \mathcal E_{\overline{2}}^1\oplus  \mathcal E_{\overline{2}}^2$ as a holomorphic splitting, where $ \mathcal E_{\overline{2}}^1$ is of rank $1$, $\operatorname{im}(\alpha)\subset \mathcal E_{\overline{2}}^1\otimes \KK_X$, and $\beta(\mathcal E_{\overline{2}}^2)=0$. Moreover, \begin{itemize}
        \item $\beta\colon\mathcal E_{\overline{2}}^1\rightarrow \mathcal E_{\overline{3}}\otimes \KK_X$ is an isomorphism if $u_\upRoman{1}=v_\upRoman{1}$;
        
        \item $\alpha\colon\mathcal E_{\overline{1}}\rightarrow \mathcal E_{\overline{2}}^1\otimes \KK_X$ is an isomorphism if $v_\upRoman{1}=4(g-1)$. 
    \end{itemize}  

    \item If $\tau=\upRoman{2}$ and $(d_i)\in \mathscr D_\tau\setminus\mathscr D_\tau^{\mathrm{s}}$, then $\mathcal E_{\overline{3}}=\mathcal E_{\overline{3}}^1\oplus \mathcal E_{\overline{3}}^2$ as a holomorphic splitting, where $\mathcal E_{\overline{3}}^1$ is of rank $1$, $\operatorname{im}(\beta)\subset \mathcal E_{\overline{3}}^1\otimes\KK_X$, and $\gamma(\mathcal E_{\overline{3}}^2)=0$. Moreover, 
    
    \begin{itemize}
        \item $\beta\colon\mathcal E_{\overline{2}}\rightarrow \mathcal E_{\overline{3}}^1\otimes \KK_X$ is an isomorphism if $u_\upRoman{2}=v_\upRoman{2}$; 
        \item $\gamma\colon\mathcal E_{\overline{3}}^1\rightarrow \mathcal E_{\overline{0}}\otimes\KK_X$ is an isomorphism if $v_\upRoman{2}=4(g-1)$. 
    \end{itemize}
\end{itemize}
\end{lemma}

\begin{proof}
By Lemma \ref{lem:DualDegree}, it is enough to prove for $4$-cyclic Higgs bundles of type \upRoman{1}.

\begin{itemize}
    \item We first prove that $u_\upRoman{1}\geqslant-4(g-1)$. 
    
    \begin{itemize}
        \item Suppose that
 $\gamma\colon\mathcal E_{\overline{3}}\to \mathcal E_{\overline{0}}\otimes\KK_X$ is nonzero. Since both $\mathcal E_{\overline{3}}$ and $\mathcal E_{\overline{0}}\otimes\KK_X$ are line bundles, then 
$d_3\leqslant d_0+2(g-1)$. Since $d_1=d_0-2(g-1)$, this is exactly $u_{\upRoman{1}}\geqslant-4(g-1)$.

\item Suppose that $\gamma\equiv0$. Then the bundles $\mathcal E_3, \mathcal E_{\overline{2}}\oplus \mathcal E_{\overline{3}}, \mathcal E_{\overline{1}}\oplus \mathcal E_{\overline{2}}\oplus \mathcal E_{\overline{3}}$ are $\Phi$-invariant.  Polystability gives 
\[d_3\leqslant \mu, \quad d_2+d_3\leqslant (q+1)\mu,\quad  d_1+d_2+d_3\leqslant (q+2)\mu.\] The second inequality is equivalent to $d_0+d_1=2d_1+(2g-2)\geqslant 2\mu$. Thus 
\[u_{\upRoman{1}}=d_1-d_3\geqslant-(g-1)\geqslant -4(g-1).\]
    \end{itemize}

\item We next prove that $v_{\upRoman{1}}\leqslant4(g-1)$ and test what happens when the equality holds. 

\begin{itemize}
    \item Suppose that $\alpha$ is a nonzero map. Let $A\subset \mathcal E_{\overline{2}}$ be the saturated image of the map \[\alpha\colon \mathcal E_{\overline{1}}\KK_X^{-1}=\LL\KK_X^{-2}\to \mathcal E_{\overline{2}}.\] Then $A$ is a line subbundle and
$\deg(A)\geqslant d_0-4(g-1)$.  If $q=1$, the same nonzero line-bundle map gives
$d_2\geqslant d_0-4(g-1)$, hence $v_{\upRoman{1}}=d_0-d_2\leqslant4(g-1)$.  If $q\geqslant2$, the rank-four bundle
\[
C=\mathcal E_{\overline{0}}\oplus \mathcal E_{\overline{1}}\oplus A\oplus \mathcal E_{\overline{3}}
\]
is a proper $\Phi$-invariant subbundle.  Polystability gives $\deg(C)\leqslant4\mu$. Stability gives $\deg(C)<4\mu$. But
$\deg(C)\geqslant d_0+d_1+d_3+d_0-4(g-1)$. The identity
\[
4\mu-\left(d_0+d_1+\left(d_0-4(g-1)\right)+d_3\right)=4(g-1)-v_{\upRoman{1}}
\]
therefore yields $v_{\upRoman{1}}\leqslant 4(g-1)$ if it is polystable and $v_{\upRoman{1}}<4(g-1)$ if it is stable. 

If $q\geqslant 2$ and $v_{\upRoman{1}}=4(g-1)$, we obtain $\deg(C)=4\mu$ and by polystability,  $\mathcal E_{\overline{2}}=A\oplus \mathcal E_{\overline{2}}^2$ as a holomorphic splitting. Moreover, $\deg(A)=d_0-4(g-1)$, implying that $\alpha$ is injective.

\item Suppose that $\alpha$ is a zero map. Then the bundles $\mathcal E_{\overline{1}}$, $\mathcal E_{\overline{0}}\oplus \mathcal E_{\overline{1}}$, $\mathcal E_{\overline{3}}\oplus \mathcal E_{\overline{0}}\oplus \mathcal E_{\overline{1}}$ are $\Phi$-invariant. Polystability gives 
\[d_1\leqslant \mu, \quad d_0+d_1=2d_1+2(g-1)\leqslant 2\mu,\quad  d_3+d_0+d_1\leqslant 3\mu.\] The last inequality is equivalent to $d_2\geqslant q\mu$. Thus \[v_{\upRoman{1}}=d_0-d_2+(q-1)\mu=d_1-d_2+(2g-2)+(q-1)\mu\leqslant 2(g-1)< 4(g-1).\]
\end{itemize}

\item We finally prove that $u_{\upRoman{1}}\leqslant v_{\upRoman{1}}$ and test what happens when the equality holds.
\begin{itemize}
    \item Suppose that $\beta$ is a nonzero map. Let $P=\ker(\beta)$.  If $q=1$, $\beta\colon\mathcal E_{\overline{2}}\to \mathcal E_{\overline{3}}\otimes\KK_X$ being nonzero gives 
    \[u_{\upRoman{1}}-v_{\upRoman{1}}=(d_1-d_3)-(d_0-d_2)=(d_2-d_3)-2(g-1)\leqslant 0.\]
    If $q\geqslant2$, then $P$ has rank $(q-1)$, is $\Phi$-invariant, and \[\deg(P)\geqslant d_2-d_3-2(g-1).\]
    Also, 
    \[v_{\upRoman{1}}-u_{\upRoman{1}}=(q-1)\mu-d_2+d_3+2(g-1)\geqslant(q-1)\mu-\deg(P).\] Polystability gives $\deg(P)\leqslant (q-1)\mu$ and thus $v_{\upRoman{1}}-u_{\upRoman{1}}\geqslant 0$. Stability gives
$\deg(P)<(q-1)\mu$ and thus $v_{\upRoman{1}}-u_{\upRoman{1}}>0$. 

If $q\geqslant 2$ and $u_{\upRoman{1}}=v_{\upRoman{1}}$, we obtain $\deg(P)=(q-1)\mu$ and by polystability,  $\mathcal E_{\overline{2}}=P\oplus \mathcal E_{\overline{2}}^1$ as a holomorphic splitting. Moreover, $d_2=d_3+2(g-1)+(q-1)\mu$, implying that $\beta$ is surjective.

    \item Suppose that $\beta$ is a zero map. Then the bundles $\mathcal E_{\overline{2}}$, $\mathcal E_{\overline{1}}\oplus \mathcal E_{\overline{2}}$, $\mathcal E_{\overline{0}}\oplus \mathcal E_{\overline{1}}\oplus \mathcal E_{\overline{2}}$ are $\Phi$-invariant. Polystability gives 
\[d_2\leqslant q\mu,\quad d_1+d_2\leqslant (q+1)\mu,\quad d_0+d_1+d_2\leqslant
(q+2)\mu.\] The last inequality is equivalent to $d_3\geqslant \mu$.
Therefore, \[v_{\upRoman{1}}-u_{\upRoman{1}}=(q-1)\mu-d_2+d_3+2(g-1)\geqslant 2(g-1)>0.\]
\end{itemize}
\end{itemize}
\end{proof}

\begin{remark}
Let $(\mathcal E,\Phi)$ be a $4$-cyclic Higgs bundle of type $\upRoman{1}$. As a $\mathrm{U}_{2, q+1}$-Higgs bundle, the Toledo invariant is 
\[
\operatorname{Tol}=\frac{2}{q+3}\Bigl((q+1)(d_1+d_3)-2(d_0+d_2)\Bigr).
\]The Milnor--Wood type inequality for $\mathrm{U}_{2,q+1}$-Higgs bundles (due to Domic and
Toledo \cite{domic1987gromov}) says $|\operatorname{Tol}|\leqslant4(g-1)$. The maximal Higgs bundles are those that satisfy $|\operatorname{Tol}|=4(g-1)$.
A direct use of $(q+3)\mu=d_0+d_1+d_2+d_3$ gives
\[
\operatorname{Tol}=v_\upRoman{1}-u_\upRoman{1}-4(g-1).
\]
Then $|\operatorname{Tol}|\leqslant4(g-1)$ is equivalent to $0\leqslant v_\upRoman{1}-u_\upRoman{1}\leqslant8(g-1)$. Hence, the Milnor-Wood inequality is already
contained in the degree criterion: $u_\upRoman{1}\leqslant v_\upRoman{1}$ gives the lower inequality, while
$u_\upRoman{1}\geqslant-4(g-1)$ and $v_\upRoman{1}\leqslant4(g-1)$ give the upper one. 

For $q\geqslant2$, the strict conditions
$u_\upRoman{1}<v_\upRoman{1}<4(g-1)$ imply the strict Milnor--Wood inequality $|\operatorname{Tol}|<4(g-1)$, meaning maximal Higgs bundles are never stable.

When $q=1$, both boundary cases are allowed by the theorem. At
$\operatorname{Tol}=-4(g-1)$ one has $u_\upRoman{1}=v_\upRoman{1}$, so the nonzero line map
$\beta\colon\mathcal E_2\to \mathcal E_3\KK_X$ is an isomorphism. At $\operatorname{Tol}=4(g-1)$ one has
$u_\upRoman{1}=-4(g-1)$ and $v_\upRoman{1}=4(g-1)$, so the nonzero line maps $\gamma\colon\mathcal E_{\overline{3}}\to \mathcal E_{\overline{0}}\KK_X$ and
$\mathcal E_{\overline{1}}\KK_X^{-1}\to \mathcal E_{\overline{2}}$ induced by $\alpha$ are isomorphisms. The explicit construction
above produces stable examples on both maximal points.
\end{remark}

Below we construct stable Higgs bundles for $(d_j)\in\mathscr{D}_\tau$. We will need the following lemma.

\begin{lemma}\label{lem:Hecke}
Let $T$ be a torsion sheaf on $X$ determined by a reduced divisor, i.e., $T=\oplus_{i=1}^t\mathbb C_{x_i}$ for distinct $x_i$'s.  For every $r\geq 1$ and $e\in\mathbb Z$, there exist stable bundles $P$ and
$R$ of rank $r$ and degrees $e$ and $e+t$, together with an exact sequence
\[0\longrightarrow P\longrightarrow R\longrightarrow T\longrightarrow0.\]
\end{lemma}
\begin{proof} Note that $r=1$ is obvious. We restrict ourselves to $r\geq 2$.
A positive elementary transformation of a rank-$r$ bundle $E$ at a point
$x\in X$ is an exact sequence
\[
0\longrightarrow E\longrightarrow E^{+}\longrightarrow \mathcal O_x\longrightarrow0.
\]
It keeps the rank and raises the degree by one.   The inverse operation is called a negative elementary transformation.
By \cite[Lemma~2.5]{Ballico2000}, a general positive or negative
elementary transformation of a general stable bundle is stable.

The torsion sheaf $T$ admits a filtration whose
successive quotients are skyscraper sheaves $\mathcal O_{x_i}$; the same point only occur
once.  Perform the corresponding positive elementary transformations
successively. Thus the final quotient is precisely $T$.
We obtain stable bundles
$E_0,\ldots,E_t$ fitting into
\[
0\longrightarrow E_{i-1}\longrightarrow E_i
\longrightarrow\mathcal O_{x_i}\longrightarrow0.
\]
Taking $P=E_0$ and $R=E_t$ gives
$0\to P\to R\to T\to0$. \end{proof}

\begin{proposition}\label{prop:connectedness}
\begin{enumerate}[label=(\roman*)]
 \item  $\mathcal M_{(d_j),\tau}(X)$ is nonempty if and only if $(d_i)\in \mathscr D_\tau$.
    \item $\mathcal M^{\mathrm{s}}_{(d_j),\tau}(X)$ is nonempty if and only if $(d_i)\in \mathscr D_\tau^{\mathrm{s}}$.
   \item For $(d_i)\in \mathscr D_\tau\setminus \mathscr D_\tau^{\mathrm{s}}$,  $\mathcal M_{(d_j),\tau}(X)\cong\mathcal M_{(d_i'),\tau}(X)\times \mathcal M_\mu^{m+n-3}(X)$,
   where 
   \[d_0'=d_0,\quad d_1'=d_1, \quad  d_2'=d_2-(n-1)\mu,\quad  d_3'=d_3-(m-2)\mu\]
   and $\mathcal M_\mu^{m+n-3}(X)$ denotes the moduli space of polystable vector bundles of rank $n+m-3$ with slope $\mu$. That is, every element decomposes as the direct sum of a $\operatorname{U}_{2,2}$-Higgs bundle in $\mathcal M_{(d_i'),\tau}(X)$ and a polystable vector bundle of rank $n+m-3$ with slope $\mu$.
    \item For $(d_i)\in \mathscr D_\tau^{\mathrm{s}}$, the image of $\mathcal M_{(d_j),\tau}^{\mathrm{s}}(X)$ by the Hitchin fibration contains $H^0_{simple}(\KK_X^4)\setminus\{0\}$, where $H^0_{simple}(\KK_X^4)\setminus\{0\}$ consists of quartic differentials with only simple zeros. 
\end{enumerate}
\end{proposition}
\begin{proof}
By \prettyref{lem:DualDegree}, it suffices to prove the statements for $4$-cyclic Higgs bundles of type \upRoman{1}. The ``only if'' direction of (\romannumeral1) and (\romannumeral2) holds directly by \prettyref{lem:StableNecessary}.

Suppose that $q=1$. Let $(d_j)\in\mathscr{D}_\tau=\mathscr{D}_\tau^\mathrm{s}$.
Set
\[
a=u_\upRoman{1}+4(g-1),\qquad b=4(g-1)-v_\upRoman{1},\qquad c=v_\upRoman{1}-u_\upRoman{1}.
\]
The assumed inequalities \prettyref{eq:b} say that $a,b,c$ are nonnegative integers and
$a+b+c=8(g-1)=\deg(\KK_X^4)$. Given a nonzero quartic differential
$Q\in H^0(\KK_X^4)$, we split its zero divisor as
$\operatorname{div}(Q)=D_\gamma+D_\alpha+D_\beta$ with degrees
$a,b,c$. Let $\mathcal E_{\overline{0}}$ be a line bundle $\LL$ of degree $d_0$, put
\[\mathcal E_{\overline{1}}:=\LL\KK_X^{-1},\qquad 
\mathcal E_{\overline{3}}:=\LL\KK_X(-D_\gamma),\qquad\mathcal E_{\overline{2}}:=\LL\KK_X^{-2}(D_\alpha).
\] Thus the degrees of $\mathcal E_{\overline{i}}$ are $(d_i)$.
The canonical sections of the three effective divisors define nonzero maps $\mathds 1\colon\EE_{\overline{0}}\to\EE_{\overline{1}}\KK_X$, 
$\alpha\colon\EE_{\overline{1}}\to\EE_{\overline{2}}\KK_X$, $\beta\colon\EE_{\overline{2}}\to\EE_{\overline{3}}\KK_X$, and $\gamma\colon\EE_{\overline{3}}\to\EE_{\overline{0}}\KK_X$, and their product is $Q$. Let $\EE=\oplus_{i=0}^3 \mathcal E_{\overline{i}}$ and the Higgs field $\Phi$ is given by the canonical sections $\mathds 1, \alpha,\beta$ and $\gamma$. 

Suppose that $q\geqslant 2$. Let $(d_j)\in\mathscr{D}_\tau^\mathrm{s}$.
Then 
\[-4(g-1)\leqslant u_\upRoman{1}<v_\upRoman{1}<4(g-1),\] and put $r:=q-1$,
$a:=u_\upRoman{1}+4(g-1)$, and $t:=4(g-1)-u_\upRoman{1}$. Let $\EE_{\overline{0}}$ be a line bundle $\LL$ of degree $d_0$ and $\EE_{\overline{1}}=\EE_{\overline{0}}\KK_X^{-1}$. Choose $Q\neq 0\in H^0_{simple}(\KK_X^4)$ and write
$\operatorname{div}(Q)=D_\gamma+D$, where $\deg D_\gamma=a$ and
$\deg D=t$.  Set
\[
\EE_{\overline{3}}:=\LL\KK_X(-D_\gamma),\qquad A:=\LL\KK_X^{-2},\qquad B:=\EE_{\overline{3}}\KK_X=\LL\KK_X^2(-D_\gamma).
\]
The differential $Q$ gives an injection $f\colon A\to B$ whose cokernel is a
torsion sheaf $T$ of length $t$ deterined by a reduced divisor $D$ since $\div(Q)$ is reduced.

Let 
$p:=d_2-d_3-2(g-1)$ and $e:=d_2-d_0+4(g-1)$.  Since $e-p=t$, Lemma \ref{lem:Hecke} gives
stable rank-$r$ bundles $P,R$ and an exact sequence
$0\to P\to R\to T\to0$ such that $\deg(P)=p$, $\deg(R)=e$.  Form the fiber product
\[
\mathcal{V}=\ker\bigl(R\oplus B\longrightarrow T\bigr),
\]
where the map is the difference of the two quotient maps.  It is locally free and
fits into the two exact sequences
\[
0\to A\to \VV\to R\to0,
\qquad
0\to P\to \VV\stackrel{\beta}{\to}B\to0.
\]
Thus $\deg(\VV)=\deg(A)+\deg (R)=\deg(A)+e=d_2$. Set $\EE_2:=\VV$. The inclusion $A\to \VV$ defines
$\alpha$, the second sequence defines $\beta$, and the natural inclusion
$\EE_3\to \EE_0\KK_X$ defines $\gamma$.   Thus the degrees of $\mathcal E_{\overline{i}}$ are $(d_i)$.
 Let $\EE=\oplus_{i=0}^3 \mathcal E_{\overline{i}}$ and the Higgs field $\Phi$ be given by the  nonzero maps $\mathds 1\colon\EE_{\overline{0}}\to\EE_{\overline{1}}\KK_X$, 
$\alpha\colon\EE_{\overline{1}}\to\EE_{\overline{2}}\KK_X$, $\beta\colon\EE_{\overline{2}}\to\EE_{\overline{3}}\KK_X$, and $\gamma\colon\EE_{\overline{3}}\to\EE_{\overline{0}}\KK_X$. Their cyclic product is $Q\neq 0$.

It remains to prove the stability of $(\EE,\Phi)$.  Put $b:=4(g-1)-v_\upRoman{1}>0$ and $c:=v_\upRoman{1}-u_\upRoman{1}>0$. From the
definitions,
\[
\deg(P)=r\mu-c,
\qquad
\deg(R)=r\mu+b,
\qquad
\deg(C)=4\mu-b,
\]
where  $C=\EE_0\oplus \EE_1\oplus A\oplus \EE_3$.

We only need to check invariant graded subbundles, that is, either contained in
$P$ or contains $C$. Let $\FF$ be a proper $\Phi$-invariant subbundle. In the first case,
$\FF\subset P$, and stability of $P$ gives
$\mu(F)\leqslant\mu(P)=\mu-c/r<\mu$.  In the second case, let $\mathcal{G}$ be the image of
$\FF$ in $R=\EE/C$ and write $k=\operatorname{rank}(\mathcal{G})$.  Properness gives $0\leqslant k<r$.
Since $\deg(\FF\cap C)\leqslant\deg(C)$ and $R$ is stable,
\[
\deg(\FF)\leqslant\deg(C)+\deg(\mathcal{G})
\leqslant4\mu-b+k\mu+\frac{kb}{r}
=(4+k)\mu-b\left(1-\frac{k}{r}\right)<(4+k)\mu.
\]
Thus every proper $\Phi$-invariant subbundle has slope strictly smaller than $\mu$.
The constructed Higgs bundle is stable, completing the proof of the statement (\romannumeral2)(\romannumeral4).

For part (\romannumeral1), it follows from the characterization of polystable elements as in \prettyref{lem:StableNecessary}.

For part (\romannumeral3), we just need to consider $q\geqslant 2$ and construct the strictly polystable elements as direct sum of $(\mathcal E',\Phi')\in \mathcal M_{(d_i'),\tau}(X)$ and $(\mathcal E'',0)$ for $\mathcal E''\in \mathcal M_{\mu}^{n+m-3}(X)$.
\end{proof}

\subsection{The real form}

In this section, we consider the moduli space of $\GR$-Higgs bundles for $\GR=\mathrm{U}_{m,n+1}$. We first recall the general definition. Let $\GR$ be a real reductive Lie group with maximal compact subgroup $H_{\mathbb{R}}$ and Cartan decomposition $\mathfrak{g}_{\mathbb{R}}=\mathfrak{h}_{\mathbb{R}}\oplus\mathfrak{m}_{\mathbb{R}}$. Let $H$ be the complexification of $H_{\mathbb{R}}$ and we have the complexified Cartan decomposition $\mathfrak{g}=\mathfrak{h}\oplus\mathfrak{m}$. We also have the isotropy representation $H\to\mathrm{GL}(\mathfrak{m})$ which is the restriction of the adjoint action.

\begin{definition}
    A $\GR$-Higgs bundle is an $(H,\mathfrak{\mathfrak{m}})$-Higgs pair.
\end{definition}
Thus we also have (semi,poly)-stability and the moduli space for $\GR$-Higgs bundles. We denote by $\mathcal{M}(\GR):=\mathcal{M}^\mu(H,\mathfrak{m})$ the moduli space of $\mu$-polystable $\GR$-Higgs bundles and by $\mathcal{M}^{\mathrm{s}}(\GR)$ its stable part.

For $\GR=\mathrm{U}_{m,n+1}<\mathrm{GL}_{m+n+1}\mathbb{C}=:G$, where either $m=2$ or $n=1$, it easy to verify that a $\GR$-Higgs bundle $(\mathbb{E},\Phi)$ is equivalent to a $2$-cyclic Higgs bundle $(\EE_{\mathrm{even}},\EE_{\mathrm{odd}})$ of rank $(m,n+1)$. This defines a continuous and locally constant map \[\begin{aligned}
    \mathbf{d}_{\GR}\colon\mathcal{M}(\GR)&\longrightarrow\mathbb{Z}\oplus\mathbb{Z}\\
    [(\mathbb{E},\Phi)]&\longmapsto(\deg(\EE_{\mathrm{even}}),\deg(\EE_{\mathrm{odd}})).
\end{aligned}\]
We denote by $\mathcal{M}_{(a,b)}(\GR)$ the preimage of $(a,b)$ along $\mathbf{d}_{\GR}$, which is the union of some connected components. For the stable part, it is known by \cite[Theorem 2.3]{gothen2001components} and \cite[Proposition 3.13]{garcia2009hitchin} that $\mathcal{M}^{\mathrm{s}}(\GR)$ is smooth.

Now we return to our case for $4$-cyclic Higgs bundle $(\EE,\Phi)$ satisfying that $\varphi_{\overline{0}}\colon\EE_{\overline{0}}\to\EE_{\overline{1}}\otimes\KK_X$ is an isomorphism. Since we are considering a $\mathbb{Z}/4\mathbb{Z}$-grading on $\mathfrak{g}$, the grading descends to $\mathbb{Z}/2\mathbb{Z}$. For instance, we have
\[\mathfrak{g}_{\mathrm{odd}}=\mathfrak{g}_{\overline{1}}\oplus\mathfrak{g}_{-\overline{1}},\quad\mathfrak{g}_{\mathrm{even}}=\mathfrak{g}_{\overline{0}}\oplus\mathfrak{g}_{\overline{2}}.\]
Let $\tau$ denote the negative conjugate transpose, which is the Cartan involution on $\mathfrak{g}$. We obtain a real form $\sigma$ by setting
\[\sigma|_{\mathfrak{g}_{\mathrm{even}}}=\tau|_{\mathfrak{g}_{\mathrm{even}}},\quad\sigma|_{\mathfrak{g}_{\mathrm{odd}}}=-\tau|_{\mathfrak{g}_{\mathrm{odd}}}.\]
It fixes the subalgebra $\mathfrak{g}_\mathbb{R}\cong\mathfrak{u}_{m,n+1}$ and $\mathfrak{g}=\mathfrak{g}_{\mathrm{even}}\oplus\mathfrak{g}_{\mathrm{odd}}$ gives the complexified Cartan decomposition of $\mathfrak{g}_{\mathbb{R}}$. We also have the corresponding reductive Lie group $\GR\cong\mathrm{U}_{m,n+1}$.

When using the standard representation with respect to the decomposition $\mathbb{C}^r=V_{\mathrm{even}}\oplus V_{\mathrm{odd}}$, where $V_{\mathrm{even}}=V_0\oplus V_2$ and $V_{\mathrm{odd}}=V_1\oplus V_3$, we obtain a $2$-cyclic Higgs bundle 
\[\begin{aligned}
    \EE&=(\EE_{\overline{0}}\oplus\EE_{\overline{2}})\oplus(\EE_{\overline{1}}\oplus\EE_{\overline{3}})\\
    \Phi&=\begin{pmatrix}
        &\begin{pmatrix}
            0&\varphi_{\overline{3}}\\
            \varphi_{\overline{1}}&0
        \end{pmatrix}\\
        \begin{pmatrix}
            \varphi_{\overline{0}}&0\\
            0&\varphi_{\overline{2}}
        \end{pmatrix}&
    \end{pmatrix}.
\end{aligned}\]
This coincides with the trivial extension of structure group from a $(G_0,\mathfrak{g}_{\overline{1}})$-Higgs pair to an $(H,\mathfrak{m})$-Higgs pair. Furthermore, it extends to a $G$-Higgs bundle. By \cite[Proposition 5.7]{garcia2019involution}, we obtain that such maps descend to the level of moduli space.

\begin{proposition}
    The natural map $\mathcal{M}_{\tau}\to\mathcal{M}(\GR)\to\mathcal{M}(G)$ always has finite fibre, and it restricts to a smooth immersion $\mathcal{M}_{\tau}^{\mathrm{s},\star}\to\mathcal{M}^{\mathrm{s}}(\GR)\to\mathcal{M}^{\mathrm{s}}(G)$. Moreover, by adding the degree restriction, we obtain the map 
\[\mathcal{M}_{(d_j)}(G_0,\mathfrak{g}_{\overline{1}})\longrightarrow\mathcal{M}_{(d_0+d_2,d_1+d_3)}(\GR)\longrightarrow\mathcal{M}(G)\]
restricts to an injective smooth immersion
\[\mathcal{M}_{(d_j),\tau}^{\mathrm{s},\star}\longrightarrow\mathcal{M}_{(d_0+d_2,d_1+d_3)}^{\mathrm{s}}(\GR)\longrightarrow\mathcal{M}^{\mathrm{s}}(G).\]
\end{proposition}

\begin{proof}
   The finiteness of the fibre follows from \prettyref{prop:finite} since both of them naturally mapped into the moduli spce of $G$-Higgs bundle with finite fibre.
    
    Let $(\mathbb{E}_{G_0},\Phi_{\mathfrak{g}_{\overline{1}}})$ be a non-maximal $\mu$-stable $(G_0,\mathfrak{g}_{\overline{1}})$-Higgs pair. By \cite[Proposition 5.7]{garcia2019involution}, its corresponding $\GR$-Higgs bundle $(\mathbb{E}_{H},\Phi_{\mathfrak{m}})$ is $\mu$-polystable. By \prettyref{coro:infautpoly}, to prove that $(\mathbb{E}_{H},\Phi_{\mathfrak{m}})$ is $\mu$-stable, it suffices to prove that $\mathbb{H}^0(C^\bullet(\mathbb{E}_{H},\Phi_{\mathfrak{m}}))=\mathbb{C}\cdot\operatorname{id}_{\EE}$. Since
    \[\mathbb{H}^0(C^\bullet(\mathbb{E}_{H},\Phi_{\mathfrak{m}}))=\operatorname{aut}(\EE,\Phi)\oplus\operatorname{aut}_{\overline{2}}(\EE,\Phi),\] this follows from \prettyref{prop:Schur} and \prettyref{prop:H2vanish}. Now a similar argument shows that the corresponding $G$-Higgs bundle $(\mathbb{E}_{G},\Phi_{\mathfrak{g}})$ is also $\mu$-stable.

    To prove that the map is an immersion, it suffices to prove that the tangent map is injective. This follows from that $C^\bullet(\mathbb{E}_{G_0},\Phi_{\mathfrak{g}_{\overline{1}}})$ is a direct summand of $C^\bullet(\mathbb{E}_{G},\Phi_{\mathfrak{g}})$, hence $\mathbb{H}^1(C^\bullet(\mathbb{E}_{G_0},\Phi_{\mathfrak{g}_{\overline{1}}}))\to\mathbb{H}^1(C^\bullet(\mathbb{E}_{G},\Phi_{\mathfrak{g}}))$ is injective and the obstructions vanish due to \prettyref{prop:H2vanish}.

    To prove the injectivity, take any two cyclic Higgs bundles $(\EE,\Phi)$ and $(\EE',\Phi')$ in $\mathcal{M}_{(d_j),\tau}^{\mathrm{s},\star}$. If they are isomorphic as $G$-Higgs bundles, then let $f$ be an isomorphism between them. Using the grading, $f$ can be decomposed as the sum of $f_{\overline{k}}$'s such that $f_{\overline{k}}|_{\EE_{\overline{j}}}\colon\EE_{\overline{j}}\to\EE_{\overline{j+k}}$. Also $f_{\overline{k}}$ intertwines $\Phi,\Phi'$ for every $k$. Hence by \prettyref{prop:Schur}, $f_{\overline{k}}$ is either an isomorphism or zero. A direct degree computation shows that $f_{\overline{2}}\neq0$ contradicts to non-maximality and either $f_{\overline{1}}$ or $f_{-\overline{1}}\neq0$ implies that $d_0=d_1$, also contradiction. Therefore, 
    $(\EE,\Phi)$ and $(\EE',\Phi')$ are isomorphic as $G$-Higgs bundles if and only if they are isomorphic as $(G_0,\mathfrak{g}_{\overline{1}})$-Higgs pair. Therefore, the restricted map $\mathcal{M}_{(d_j),\tau}^{\mathrm{s},\star}\to\mathcal{M}^{\mathrm{s}}(G)$ is injective. So is $\mathcal{M}_{(d_j),\tau}^{\mathrm{s},\star}\to\mathcal{M}_{(d_0+d_2,d_1+d_3)}^{\mathrm{s}}(\GR)$.
\end{proof}

Since $\mathrm{PU}_{m,n+1}$ is obtained from $\mathrm{U}_{m,n+1}$ modulo the center, we have the following definition of stability condition for $\PU_{m,n+1}$-Higgs bundles.

\begin{definition}
A $4$-cyclic $\PU_{m,n+1}$-Higgs bundle is stable (resp. polystable), if some, and hence every, $\mathrm{U}_{m,n+1}$-lift is stable (resp. polystable).
\end{definition}

For the correspondence between $4$-cyclic Higgs bundles and $\partial$-alternating surfaces in the case of closed Riemann surfaces, combining Theorem \ref{thm:base}, we have
\begin{theorem}\label{thm:baseSpecialSignature2}
Let $m=2$ or $n=1$. 
On a closed Riemann surface $X$, there is a natural bijection between
isomorphism classes of:
\begin{enumerate}
\item $4$-cyclic stable $\PU_{m,n+1}$-Higgs bundles on $X$ as in
  \eqref{eq:fourcycle} with $\gamma\neq0$ when $m=2$ and $\alpha\neq0$ when $n=1$, modulo graded unitary gauge equivalence;
\item equivariant non-degenerate $\partial$-alternating immersions
  $f\colon\widetilde X\to\CH^{m,n}$.
\end{enumerate}
\end{theorem}
\begin{proof}
The equivalence between the polystability and harmonicity follows from \cite[Theorem 2.24]{garcia2009hitchin} or \cite[Theorem 5.4]{garcia2019involution}, that the solution of Hitchin equation could be chosen as a diagonal solution as in \prettyref{defn:harmonic}. Hence it suffices to give the correspondence between being stable and being non-degenerate. Equivalently, it suffices to give the correspondence between being strictly polystable and being degenerate. Here we only prove this when $m=2$. When $n=1$, the proof is similar. If the $4$-cyclic $\PU_{m,n+1}$-Higgs bundle is strictly polystable, then the lifting $4$-cyclic $\mathrm{U}_{m,n+1}$-Higgs bundle $(\EE,\Phi)$ must decompose as the direct sum of polystable $4$-cyclic Higgs bundles
\begin{equation}\label{eq:strict-poly}
    \left(\begin{tikzcd}[ampersand replacement=\&]
      {\mathcal{E}_{\overline{0}}}
      \& {\mathcal{E}_{\overline{1}}}
      \& {\mathcal{E}_{\overline{2}}'}
      \& {\mathcal{E}_{\overline{3}}} 
      \arrow["{\mathds{1}}"', from=1-1, to=1-2]
      \arrow["\alpha"', from=1-2, to=1-3]
      \arrow["{\beta}"', from=1-3, to=1-4]
      \arrow["\gamma"', curve={height=19pt}, from=1-4, to=1-1]
    \end{tikzcd}\right)\oplus\left({\mathcal{F}_{\overline{2}}},0\right)
\end{equation}
    for some nonzero $\FF_{\overline{2}}$ since both $\mathds{1}$ and $\gamma\neq0$. Hence the polystable piece
    \[\begin{tikzcd}[ampersand replacement=\&]
      {\mathcal{E}_{\overline{0}}}
      \& {\mathcal{E}_{\overline{1}}}
      \& {\mathcal{E}_{\overline{2}}'}
      \& {\mathcal{E}_{\overline{3}}} 
      \arrow["{\mathds{1}}"', from=1-1, to=1-2]
      \arrow["\alpha"', from=1-2, to=1-3]
      \arrow["{\beta}"', from=1-3, to=1-4]
      \arrow["\gamma"', curve={height=19pt}, from=1-4, to=1-1]
    \end{tikzcd}\]
    implies that $f$ lies in a totally geodesic copy of $\CH^{2,n-1}$, hence degenerates. Conversely, if $f$ degenerates, then the corresponding harmonic bundle decomposes as \prettyref{eq:strict-poly}, which is strictly polystable and $\gamma\neq0$.
\end{proof}

\subsection{Character variety}

We recall some general facts about representations of a surface group
in $\mathrm{U}_{m, n+1}$ or $\PU_{m, n+1}$. See \cite{goldman1985representations}, \cite[Section 2]{bradlow2003surface} for references. Let $\Sigma$ be a closed oriented surface of genus $g(\Sigma)\geqslant2$. 

\begin{definition}
    The character variety, which is the moduli space of representations, of a finitely presented group $\pi$ in a reductive Lie group $\GR$ with Lie algebra $\mathfrak{g}_{\mathbb{R}}$ is defined as 
    \[\begin{aligned}
        \mathfrak{X}(\pi,\GR):&=\operatorname{Hom}(\pi,\GR)\sslash \GR=\operatorname{Hom}^{\rm{red}}(\pi,\GR)/\GR,
    \end{aligned}\]
    where $\operatorname{Hom}^{\rm{red}}(\pi,\GR)$ denotes the subvariety consisting of all completely reducible representations, which are the representations such that $\operatorname{Ad}\circ\rho\colon\pi\to\mathrm{GL}(\mathfrak{g}_{\mathbb{R}})$ is completely reducible and $\GR$ acts on $\operatorname{Hom}(\pi,\GR)$ by the adjoint action.
\end{definition}

The character variety $\mathfrak{X}(\pi_1(\Sigma),\GR)$ can be described more concretely as follows.
From the standard presentation
\[\pi_1(\Sigma) = \Big\langle a_1,b_1, \cdots ,a_{g(\Sigma)},b_{g(\Sigma)}\mid\prod^{g(\Sigma)}
_{i=1}[a_i,b_i] =1\Big\rangle,\]
we see that $\operatorname{Hom}^{\rm{red}}(\pi_1(\Sigma),\GR)$ can be embedded in $(\GR)^{2g(\Sigma)}$ via \[\rho\mapsto(\rho(a_1),\cdots,\rho(b_{g(\Sigma)})).\]
We give $\operatorname{Hom}^{\rm{red}}(\pi_1(\Sigma),\GR)$ the subspace topology and $\mathfrak{X}(\pi_1(\Sigma),\GR)$ the quotient topology. This topology is Hausdorff because we have restricted attention to completely reducible representations.

Below we focus on the case that $\GR=\mathrm{U}_{m,n+1}$ and we denote $\PGR:=\PU_{m,n+1}$. Clearly any representation of $\pi_1(\Sigma)$ in $\GR$ gives rise to a representation in $\PGR$; however, not all representations in $\PGR$ lift to $\GR$. We thus consider representations of the universal central extension
\[0\longrightarrow\mathbb{Z}\longrightarrow\Gamma\longrightarrow\pi_1(\Sigma)\longrightarrow1.\]
Such extensions are defined by the generators $a_1,b_1, \dots ,a_{g(\Sigma)},b_{g(\Sigma)}$ and a central
element $J$ subject to the relation $\prod^{g(\Sigma)}_{i=1}[a_i,b_i] = J$. With $\Gamma$ thus defined, any representation of $\pi_1(\Sigma)$ in $\PGR$ can be lifted to a representation of $\Gamma$ in $\GR$. 

Note that $\mathfrak{X}(\Gamma,\GR)$ can be identified with the moduli space of connections with central curvature on a fixed $\GR$-bundle on $X$. Taking a reduction to the maximal compact $K_{\mathbb{R}}=\mathrm{U}_m\times\mathrm{U}_{n+1}$, we thus associate to each class $[\tilde\rho]\in\mathfrak{X}(\Gamma,\GR)$ a vector bundle of the form $V\oplus W$, where $V$ and $W$ are rank $m$ and $n+1$ respectively, and thus a pair of integers $(a,b)= (\deg(V), \deg(W))$. There is a map \[\begin{aligned}
    \tilde{\mathbf{d}}\colon\mathfrak{X}(\Gamma,\GR)&\longrightarrow\mathbb{Z}\oplus\mathbb{Z},\qquad
    [\tilde\rho]&\longmapsto(a,b)
\end{aligned}\]
which is continuous and locally constant. We denote by $\mathfrak{X}_{(a,b)}^\Gamma\subset\mathfrak{X}(\Gamma,\GR)$ the preimage of $(a,b)\in\mathbb{Z}\oplus\mathbb{Z}$ along $\tilde{\mathbf{d}}$, which is a union of connected components. 

Since any flat $\PGR$-bundle lifts to a $\GR$-bundle with projectively flat connection up to twisting with a
line bundle $L$ with a unitary connection of constant curvature. Therefore, there is a well-defined map
\[\begin{aligned}
    {\mathbf{d}}\colon\mathfrak{X}(\pi_1(\Sigma),\PGR)&\longrightarrow\mathbb{Z}\oplus\mathbb{Z}/(m,n+1)\mathbb{Z}\\
    [\rho]&\longmapsto[\tilde{\mathbf{d}}([\tilde\rho])]
\end{aligned}\]
where $\tilde\rho\colon\Gamma\to\GR$ is an arbitrary lift of $\rho$. Let $\mathfrak{X}_{[a,b]}:=\mathbf{d}^{-1}([a,b])\subset\mathfrak{X}(\pi_1(\Sigma),\PGR)$, which is also a union of connected components.   The following proposition was proven in \cite[Proposition 2.5]{bradlow2003surface}:
\begin{proposition}
    The natural projection $\mathfrak{X}_{(a,b)}^\Gamma\to\mathfrak{X}_{[a,b]}$ defines a $(\mathrm{U}_1)^{2g(\Sigma)}$-fibration.
\end{proposition}

Note that such degree argument also works for the moduli space $\mathcal{M}(X;\PGR)$ of $\mu$-polystable $\PGR$-Higgs bundles over $X$ and we can define $\mathcal{M}_{[a,b]}(X;\PGR)$.

\begin{proposition}\label{prop:GtoPGfibre}
    There is a natural projection $\mathcal{M}_{(a,b)}(X;\GR)\to\mathcal{M}_{[a,b]}(X;\PGR)$ whose fibre is isomorphic to the identity component $\operatorname{Pic}_0(X)$ of the Picard group of $X$.
\end{proposition}

The following proposition was proven in \cite[Proposition 3.13]{bradlow2003surface}:
\begin{proposition}\label{prop:UNAH}
    There is a homeomorphism between $\mathcal{M}_{(a,b)}(\GR)$ and $\mathfrak{X}^\Gamma_{(a,b)}$, which covers the non-Abelian Hodge correspondence $\mathcal{M}_{[a,b]}(\PGR)$ and $\mathfrak{X}_{[a,b]}$.
\end{proposition}

It is computed in \cite[Proposition 3.20]{bradlow2003surface} and by \prettyref{prop:GtoPGfibre} that 
\begin{proposition}
    If $\mathcal{M}_{[a,b]}^{\mathrm{s}}(\PGR)$ is non-empty, then \[\dim_{\mathbb{R}}\mathcal{X}_{[a,b]}^{\mathrm{irr}}=2\cdot\dim_{\mathbb{C}}\mathcal{M}_{[a,b]}^{\mathrm{s}}(\PGR)=((m+n+1)^2-1)(2g(\Sigma)-2).\]
\end{proposition}

\subsection{Varying complex structure}

In this section, we allow the complex structure on the closed surface $\Sigma$ varies among its Teichm\"uller space $\mathcal{T}(\Sigma)$. Recall that $\mathcal{M}_{(d_j),\tau}^{\mathrm{s},\star}(X)$ (resp. $\mathcal{M}_{(d_j),\tau}^{\star}(X)$) denotes the non-maximal part of the moduli space of stable (resp. polystable) $4$-cyclic Higgs bundles of type $\tau=\upRoman{1},\upRoman{2}$ over $X$. The ranks $(r_j)=(1,1,n,m-1)$ of the $4$-cyclic bundle satisfy $n=q,m=2$ when $\tau=\upRoman{1}$ and $n=1,m=q+1$ when $\tau=\upRoman{2}$. We consider the following fibration
\[\mathcal{M}_{(d_j),\tau}^{\mathrm{s},\star}(\Sigma):=\bigsqcup_{[X]\in\mathcal{T}(\Sigma)}\mathcal{M}_{(d_j),\tau}^{\mathrm{s},\star}(X),\quad\mathcal{M}_{(d_j),\tau}^{\star}(\Sigma):=\bigsqcup_{[X]\in\mathcal{T}(\Sigma)}\mathcal{M}_{(d_j),\tau}^{\star}(X)\]
over $\mathcal{T}(\Sigma)$. This is the moduli space of all $\partial$-alternating surfaces on $\Sigma$, which can be viewed as a subvariety in the joint moduli space $\mathcal{M}(\Sigma;G)$ of $G$-Higgs bundles constructed by Simpson algebro-geometrically in \cite{simpson1994moduliI,simpson1994moduliII} and by Collier--Toulisse--Wentworth, \'{A}lvarez-C\'{o}nsul--Garcia-Fernandez--Garc\'{i}a-Prada--Trautwein and Hitchin gauge-theoretically in different ways in \cite{collier2025higgs,alvarez2025universal,hitchin2026universal} respectively. Using the non-Abelian Hodge correspondence on each fibre and \prettyref{prop:UNAH}, we obtain the joint non-Abelian Hodge map 
\[\operatorname{NAH}|_{\mathcal{M}_{(d_j),\tau}^{\mathrm{s},\star}(\Sigma)}\colon\mathcal{M}_{(d_j),\tau}^{\mathrm{s},\star}(\Sigma)\to\mathfrak{X}^\Gamma_{(d_0+d_2,d_1+d_3)}.\]
Since a point in $\mathcal{M}_{(d_j),\tau}^{\mathrm{s},\star}(\Sigma)$ is an isomorphism class of a $4$-cyclic Higgs bundle with $\varphi_{\overline{0}}$ nowhere-vanishing, by \cite[Proposition 6.10 \& Theorem 7.11]{collier2025higgs} (also see \cite[Theorem 1.1]{li2026infinitesimal}), we obtain that:
\begin{proposition}\label{prop:immersion}$\operatorname{NAH}|_{\mathcal{M}_{(d_j),\tau}^{\mathrm{s},\star}(\Sigma)}\colon\mathcal{M}_{(d_j),\tau}^{\mathrm{s},\star}(\Sigma)\to\mathfrak{X}^\Gamma_{(d_0+d_2,d_1+d_3)}$ is an immersion.
\end{proposition}

Let $[d_j]$ denote the equivalent class of $(d_j)$ in $\mathbb{Z}^4/\mathbb{Z}\cdot(1,1,n,m-1)$. We denote by $\mathcal{R}_{[d_j],\tau}$ the image of the map
\[\mathcal{M}_{(d_j),\tau}^{\star}(\Sigma)\to\mathfrak{X}^\Gamma_{(d_0+d_2,d_1+d_3)}\to\mathfrak{X}_{[d_0+d_2,d_1+d_3]}\] and $\mathcal{R}^{\mathrm{irr}}_{[d_j],\tau}$ its subset which is the image of the map
\[\mathcal{M}_{(d_j),\tau}^{\mathrm{s},\star}(\Sigma)\to\mathfrak{X}^\Gamma_{(d_0+d_2,d_1+d_3)}\to\mathfrak{X}_{[d_0+d_2,d_1+d_3]}.\] They are independent of the choice of $(d_j)$ in $[d_j]$ by \prettyref{rem:tensorline}.

\begin{corollary} \label{cor:Dimension}
\begin{enumerate}[label=(\roman*)]
    \item $\mathcal{R}^{\mathrm{irr}}_{[d_j],\tau}$ is non-empty if and only if $(d_i)\in \mathscr D_{\tau}^{s,\star}$.

    \item For $(d_i)\in \mathscr D_{\tau}^{s,\star}$, \[\begin{aligned}
        &\dim_{\mathbb{R}}\mathcal{R}^{\mathrm{irr}}_{[d_j],\tau}\\=&(4+n^2+(m-1)^2+m(n+1))(2g(\Sigma)-2)
        +2(m-2)(d_0-d_2)+2(1-n)(d_1-d_3).
    \end{aligned}\]
    In particular, when $\tau=\upRoman{1}$ and $d_1-d_3=-4(g(\Sigma)-1)$ or $\tau=\upRoman{2}$ and $d_2-d_0=-4(g(\Sigma)-1)$, it achieves the maximum dimension as $\mathbb R$-codimension $10(g(\Sigma)-1)$.
    
    \item For $(d_i)\in \mathscr D_{\tau}^{\star}\setminus\mathscr D_{\tau}^{s,\star}$, 
    \[\begin{aligned}
    \dim_{\mathbb{R}}\mathcal{R}_{[d_j],\tau}=&\dim_{\mathbb{R}}\mathcal{R}^{\mathrm{irr}}_{[d_j']}+2\dim_{\mathbb{C}}\mathcal{M}_{\mu}^{n+m-3}(X)
    \\=&[20+2(m+n-3)^2](g(\Sigma)-1)+2.
    \end{aligned}\]
\end{enumerate}\end{corollary}

\begin{proof}
Part (i) follows from Proposition \ref{prop:connectedness}.
    
For Part (ii), note that
    \[\begin{aligned}
        \dfrac{1}{2}\dim_{\mathbb{R}}\mathcal{R}^{\mathrm{irr}}_{(d_j),\tau}
=&\dim_{\mathbb{C}}\mathcal{M}_{(d_j),\tau}^{\mathrm{s},\star}(\Sigma)-\dim_{\mathbb{C}}(\operatorname{Pic}_0(X))\quad(\mbox{by \prettyref{prop:GtoPGfibre}})\\
        =&\dim_{\mathbb{C}}\mathcal{M}_{(d_j),\tau}^{\mathrm{s},\star}(X)+\dim_{\mathbb{C}}\mathcal{T}(\Sigma)-\dim_{\mathbb{C}}(\operatorname{Pic}_0(X))\quad(\mbox{by \prettyref{prop:immersion}})
    \end{aligned}\]
    Then the result follows from \prettyref{prop:Cdim}.

For Part (iii), the equality follows from Proposition \ref{prop:connectedness}. 
\end{proof}

\subsection{Proof of Theorem \ref{thm:intro-moduli}}
Part (i) and (ii) follow from Proposition \ref{prop:connectedness}. Part (iii) follows from Part (ii) in Corollary \ref{cor:Dimension} by applying $n=q,m=2$ when $\tau=\upRoman{1}$ and $n=1,m=q+1$ when $\tau=\upRoman{2}$. 
\qed

\subsection{Proof of Theorem \ref{thm:intro-slices}}
It is enough to prove the statement for all nonpositive Toledo invariants are achieved by Higgs bundles of Type $\upRoman{1}$, since one can consider the dual representation to obtain positive Toledo invariant to obtain the statement for nonpositive Toledo invariants since $\operatorname{Tol}(\rho^{\vee})=-\operatorname{Tol}(\rho)$. From Section \ref{subsec:duality},  the dual representations of type $\upRoman{1}$ are of type $\upRoman{2}$. So all the nonnegative Toledo invariants are achieved by Higgs bundles of Type $\upRoman{2}$,

Note that the Toledo invariant is $v_{\upRoman{1}}-u_{\upRoman{1}}-4(g-1)$. 

If $(n+3)t\in4\mathbb Z$, choose \[(d_i)=(0,\quad -2(g-1),\quad -\frac{n+3}{4}t,\quad 2(g-1)).\] The associated Toledo invariant is $t$ and $u_{\upRoman{1}}=-4(g-1)$. In this case, $v_{\upRoman{1}}=t<4g-4$ for $q\geq 2$ and $v_{\upRoman{1}}=t\leq 4g-4$ for $q=1$. Thus, we can achieve all possible $t$'s such that $t<4g-4$. 

If $(n+3)t\in4\mathbb Z+2$, choose \[(d_i)=(0,\quad -2(g-1),\quad -\frac{n+3}{4}t-\frac{n+1}{2},\quad 2(g-1)-1).\] The associated Toledo invariant is $t$ and $u_{\upRoman{1}}=-4(g-1)+1$. In this case, $v_{\upRoman{1}}=t+1<4g-4$ for $q\geq 2$ and $v_{\upRoman{1}}=t+1\leq 4g-4$ for $q=1$. Thus, we can achieve all possible $t$'s such that $t<4g-5$ for $q\geq 2$ and $t\leq 4g-5$ for $q=1$. Although in the case $q\geq 2$, we do not achieve $t=4g-5$, but it is enough for our purpose, since we only need to achieve all possible $t$'s such that $t\leq 0$.

For each $t$, we choose  $M_t$ to be the space $\mathcal M^{\mathrm s}_{[d_j],\upRoman{1}}$, $\iota_t$ be the the non-Abelian Hodge map, and $\mathscr S_t$ to be the image $\mathcal{R}^{\mathrm{irr}}_{[d_j],\upRoman{1}}$  inside the character variety. The immersion property of $\iota_t$ is from \prettyref{prop:immersion}. The property of being not relatively compact follows from that each $\mathcal M^{\mathrm s}_{[d_j],\upRoman{1}}(X)$ is diverging to infinity along $\mathbb C^*$-flow as long as $q_4\neq 0$. The property of elements in $\mathscr S_t$ follows from Theorem \ref{thm:intro-geometry}.  
The smooth locus of the character variety has real
dimension
\[
N_0
=
(2g-2)\dim_{\mathbb R}\PU_{2,n+1}
=
(n^2+6n+8)(2g-2).
\] Following from Theorem \ref{thm:intro-moduli}, the dimension of $\mathscr S_t$ is  
\[
N_I=(n^2+2n+7)(2g-2)-2(n-1)u_I.
\]
Hence for $(n+3)t\in4\mathbb Z$,
\[
N_0-N_\tau
=
5(2g-2)
=
10(g-1);
\]
for $(n+3)t\in4\mathbb Z+2$,
\[
N_0-N_\tau
=
5(2g-2)+2(n-1)
=
10(g-1)+2(n-1).
\]\qed

\appendix

\section{The moduli space of cyclic Higgs bundles}\label{apdx:cyclic}

Let $X$ be a closed Riemann surface of genus $g\geqslant2$. We will reduce the notion of stability conditions of general Higgs pairs to cyclic Higgs bundles and explain the local structure of the moduli space in this appendix. See \cite{garcia2009hitchin} and \cite{garcia2024cyclic} for references.

\subsection{Higgs pairs and stability}

Let $G$ be a connected complex reductive Lie group with Lie algebra $\mathfrak{g}$ and $\rho\colon G\to \mathrm{GL}(V)$ a holomorphic representation. Let $\LL$ be a line bundle over $X$. The following definitions are introduced in \cite{garcia2009hitchin}.

\begin{definition}
    An $\LL$-twisted $(G,V)$-Higgs pair is a pair $(\bEE,\Phi)$, where $\bEE$ is a holomorphic principal $G$-bundle on $X$, and $\Phi\in\mathrm{H}^0(X,\bEE[V]\otimes\LL)$, where $\bEE[V]:=\bEE\times_G V$ is the associated vector bundle to $\bEE$ via $\rho$. A $(G,V)$-Higgs pair is a $\KK_X$-twisted $(G, V)$-Higgs pair.
\end{definition}

\begin{definition}
    An isomorphism between two $\LL$-twisted $(G,V)$-Higgs pairs $(\mathbb{E},\Phi)$ and $(\mathbb{E}',\Phi')$ is an isomorphism $f\colon\mathbb{E}\to\mathbb{E}'$ such that $\Phi'=(\rho(f)\otimes\operatorname{id}_{\LL})(\Phi)$.
\end{definition}

Fix a maximal compact subgroup $K\leqslant G$. Let $\mathfrak{k}$ be its Lie algebra, a real subalgebra of $\mathfrak{g}$. Define for $s\in\iu\mathfrak{k}$ the spaces
\[V_s^0 =\{v\in V \mid \forall t\in\mathbb{R}, \rho(\mathrm{e}^{ts})(v) = v\},\quad V_s =\{v\in V \mid \rho(\mathrm{e}^{ts})(v)\mbox{ is bounded as }t\to\infty\}\]
and the subgroups
\[L_s =\{g\in G\mid\operatorname{Ad}(g)(s) = s\}, P_s = \{g\in G\mid\mathrm{e}^{ts}g\mathrm{e}^{-ts}\mbox{ is bounded as }t\to\infty\}.\]

\begin{definition}
    Fix a parameter $\alpha\in\iu\mathfrak{z}$. A $(G,V)$-Higgs pair $(\bEE,\Phi)$ is:
    \begin{itemize}
        \item $\alpha$-semistable, if for any element $s\in\iu\mathfrak{k}$ and reduction $\sigma\in\mathrm{H}^0(X,\bEE[G/P_s)]$ such that $\Phi\in\mathrm{H}^0(\bEE_\sigma[V_s]\otimes\KK_X)$, we have $\deg\bEE(\sigma,s)\geqslant B(\alpha,s)$;

        \item $\alpha$-stable, if it is $\alpha$-semistable and, for any element $s\in\iu\mathfrak{k}$ but $s\notin\operatorname{ker}(\dd\rho|_{\mathfrak{k}})$ and reduction $\sigma\in\mathrm{H}^0(X,\bEE[G/P_s])$ such that $\Phi\in\mathrm{H}^0(\bEE_\sigma[V_s]\otimes\KK_X)$, we have $\deg\bEE(\sigma, s) > B(\alpha, s)$;

        \item $\alpha$-polystable, if it is $\alpha$-semistable and, for the $s\in\iu\mathfrak{k}$ and $\sigma\in \mathrm{H}^0(X,\bEE[G/P_s])$ such that $\Phi\in\mathrm{H}^0(\bEE_\sigma[V_s]\otimes\KK_X)$ and we have $\deg\bEE(\sigma, s) = B(\alpha, s)$, there exists a reduction $\sigma'\in\mathrm{H}^0(\bEE_{\sigma}[P_s/L_s])$ of $\bEE_\sigma$ to $L_s$ such that $\Phi\in\mathrm{H}^0(\bEE_{\sigma'}[V_s^0]\otimes\KK_X)$.
    \end{itemize}
\end{definition}

\begin{definition}
    The moduli space $\mathcal{M}^\alpha(X;G,V)$ of $(G,V)$-Higgs pairs over $X$ is the space of isomorphism classes of $\alpha$-polystable $(G,V)$-Higgs pairs. We denote by $\mathcal{M}^{\alpha,\mathrm{s}}(X;G,V)$ the subset consisting of isomorphism classes of $\alpha$-stable $(G,V)$-Higgs pairs. 
\end{definition}
 We will omit $X$ in the notation is there is no abuse of notion.

\subsection{Stability of cyclic Higgs bundles}\label{apdx:stab-cyclic}

Let $G=\mathrm{GL}_{r}\mathbb{C}$. $\mathbb{C}^r=\bigoplus_{j=0}^{c-1}V_j$ an orthogonal decomposition with $\dim(V_j)=r_j$. Suppose that $V_j=0$ when $j<0$ or $j>c-1$. 
We have the $\mathbb{Z}$-grading \[\mathfrak{g}_{k}=\bigoplus_{j=0}^{c-1}\operatorname{Hom}(V_j,V_{j+k})\]
on the Lie algebra $\mathfrak{g}=\mathfrak{gl}_r\mathbb{C}$. This gives a reductive subgroup
\[G_0=\prod_{j=0}^{c-1}\mathrm{GL}(V_j)\cong\prod_{j=0}^{c-1}\mathrm{GL}_{r_j}\mathbb{C}\]
whose Lie algebra is $\mathfrak{g}_0$. The $\mathbb{Z}$-grading descends to the $\mathbb{Z}/c\mathbb{Z}$-grading $\mathfrak{g}_{\overline{j}}=\bigoplus_{c|j-k}\mathfrak{g}_k.$
Let $\iota\colon G_0\to\mathrm{GL}(\mathfrak{g}_{\overline{1}})$ denote the adjoint action. We fix the maximal compact subalgebra
$\mathfrak{k}_0=\bigoplus_{j=0}^{c-1}\mathfrak{u}_{r_j}$
in $\mathfrak{g}_0$. The kernel of $\dd\iota|_{\mathfrak{k}_0}$ is $\iu\mathbb{R}\cdot\operatorname{id}_{\mathbb{C}^r}$.

A $(G_0,\mathfrak{g}_{\overline{1}})$-Higgs pair $(\mathbb{E},\Phi)$ is equivalent to the following Higgs bundle, which we call a ($c$-)cyclic Higgs bundle of rank $(r_j)$:
\[
    \begin{aligned}
    \EE&=\EE_{\overline{0}}\oplus\EE_{\overline{1}}\oplus\cdots\oplus\EE_{\overline{c-1}},\\
    \Phi&=\begin{pmatrix}
        &&&&\varphi_{\overline{c-1}}\\
        \varphi_{\overline{0}}&&&&\\
        &\varphi_{\overline{1}}&&&\\
        &&\ddots&&\\
        &&&\varphi_{\overline{c-2}}&
        
    \end{pmatrix},
\end{aligned}
\]
where $\EE_{\overline{j}}$ is the holomorphic vector bundle $\mathbb{E}[V_j]$ of rank $r_j$ and $\varphi_{\overline{j}}\colon\EE_{\overline{j}}\to\EE_{\overline{j+1}}\otimes\KK_X$. We will simply call a $(G_0,\mathfrak{g}_{\overline{1}})$-Higgs pair $(\mathbb{E},\Phi)$ is $\mu$-(semi,poly)-stable if it is $\mu(\EE)\cdot\operatorname{id}_{\mathbb{C}^r}$-(semi,poly)-stable.

\begin{definition}
    An isomorphism between two cyclic Higgs bundles $(\EE,\Phi)$ and $(\EE',\Phi')$ is an $m$-tuple $(g_j\colon\EE_{\overline{j}}\to\EE_{\overline{j}}')_{j=0}^{m-1}$ of isomorphisms such that $\Phi'=(\iota(g)\otimes\operatorname{id}_{\KK_X})(\Phi)$ for $g=\prod_{j=0}^{m-1}g_j$.
\end{definition}

Correspondingly, two $(G_0,\mathfrak{g}_{\overline{1}})$-Higgs pairs are isomorphic is equivalent to that their corresponding cyclic Higgs bundles are isomorphic.

\begin{definition}
    A holomorphic subbundle $\mathcal{F}\subset\EE=\bigoplus_{j=0}^{m-1}\EE_{\overline{j}}$ is called graded if $\mathcal{F}=\bigoplus_{j=0}^{m-1}\mathcal{F}_{\overline{j}}$ for some holomorphic subbundles $\mathcal{F}_{\overline{j}}\subset\EE_{\overline{j}}$.
\end{definition}

\begin{definition}\label{def:polystable}
Let $(\EE,\Phi)$ be a cyclic Higgs bundle, we call it is
    \begin{itemize}
        \item semistable as a cyclic Higgs bundle if for any $\Phi$-invariant graded subbundle $\mathcal{F}\subset\EE$ with $\FF\neq0$, we have $\mu(\mathcal{F})\leqslant\mu(\EE)$;

        \item stable as a cyclic Higgs bundle if for any $\Phi$-invariant graded subbundle $\mathcal{F}\subset\EE$ with $\FF\neq0,\EE$, we have $\mu(\mathcal{F})<\mu(\EE)$;

        \item polystable as a cyclic Higgs bundle if for any $\Phi$-invariant graded subbundle $\mathcal{F}\subset\EE$ with $\FF\neq0$, we have $\mu(\FF)\leqslant\mu(\EE)$ and if $\mu(\FF)=\mu(\EE)$, then there exists a $\Phi$-invariant graded subbundle $\mathcal{S}\subset\EE$ with $\mu(\mathcal{S})=\mu(\EE)$ such that $(\EE,\Phi)=(\FF,\Phi|_\FF)\oplus(\mathcal{S},\Phi|_{\mathcal{S}})$.
    \end{itemize}
\end{definition}

One can verify the other characterization of the polystability in the following proposition for cyclic Higgs bundles as this for usual Higgs bundles, so we omit its proof.

\begin{proposition}\label{prop:polystable-equiv}
Let $(\EE,\Phi)$ be a cyclic Higgs bundle. The following are equivalent:
\begin{enumerate}[label=(\roman*)]
    \item $(\EE,\Phi)$ is polystable as a cyclic Higgs bundle;

    \item $(\EE,\Phi)=\bigoplus_{k}(\FF^k,\Phi|_{\FF^k})$
    for some $\Phi$-invariant graded subbundle $\FF^k\subset\EE$ such that $(\FF^k,\Phi|_{\FF^k})$ is stable as a cyclic Higgs bundle for every $k$ and all $\mu(\FF^k)$'s are equal.
\end{enumerate}
    
\end{proposition}



The following proposition illustrates the equivalence of between the (semi,poly-)stability of $(G_0,\mathfrak{g}_{\overline{1}})$-Higgs pairs and that of cyclic Higgs bundles.

\begin{proposition}\label{prop:equivstability}
    Let $(\mathbb{E},\Phi)$ be a $(G_0,\mathfrak{g}_{\overline{1}})$-Higgs pair and $(\EE,\Phi)$ be its associated cyclic Higgs bundle. The following conditions are equivalent:
    \begin{enumerate}[label=(\roman*)]
        \item $(\mathbb{E},\Phi)$ is $\mu$-semistable as a $(G_0,\mathfrak{g}_{\overline{1}})$-Higgs pair;
        \item $(\EE,\Phi)$ is semistable as a cyclic Higgs bundle.
    \end{enumerate}
We also have that 
    \begin{enumerate}[label=(\roman*')]
        \item $(\mathbb{E},\Phi)$ is $\mu$-stable as a $(G_0,\mathfrak{g}_{\overline{1}})$-Higgs pair;
        \item $(\EE,\Phi)$ is stable as a cyclic Higgs bundle.
    \end{enumerate}
    are equivalent and
    \begin{enumerate}[label=(\roman*'')]
        \item $(\mathbb{E},\Phi)$ is $\mu$-polystable as a $(G_0,\mathfrak{g}_{\overline{1}})$-Higgs pair;
        \item $(\EE,\Phi)$ is polystable as a cyclic Higgs bundle.
    \end{enumerate}
    are equivalent.
\end{proposition}

\begin{proof}
    For (\romannumeral1) $\implies$ (\romannumeral2), given a  $\Phi$-invariant graded subbundle $\mathcal{F}\subset\EE$, we define an element $s\in\iu\mathfrak{k}_0$ as $s|_{V_j}=\operatorname{diag}(\underbrace{-\lambda,\cdots,-\lambda}_{\operatorname{rank}(\FF_{\overline{j}})},0,\cdots,0)$ with $\lambda>0$. Let $J_\FF:=\{{\overline{j}}\mid \FF_{\overline{j}}=0\}$. Now the graded subbundle $\FF\oplus\bigoplus_{{\overline{j}}\in J_\FF}\EE_{\overline{j}}$ gives a reduction $\sigma\in\mathrm{H}^0(X,\mathbb{E}_\sigma[G_0/P_s])$, where $P_s<G_0$ denotes the parabolic subgroup associated with $s$. Also $\FF$ is $\Phi$-invariant implies that $\Phi\in\mathrm{H}^0(\mathbb{E}_\sigma[(\mathfrak{g}_1)_s]\otimes\KK_X)$. Therefore, $(\EE,\Phi)$ is $\mu(\EE)$-semistable yieds that $\deg\mathbb{E}(\sigma,s)
        \geqslant B(\mu(\EE)\cdot\operatorname{id}_{\mathbb{C}^r},s)$. Since
    \[\begin{aligned}
        \deg\mathbb{E}(\sigma,s)
        =&\sum_{j=0}^{c-1}\left(-\lambda\cdot\deg(\FF_{\overline{j}})+0\cdot\deg(\EE_{\overline{j}}/\FF_{\overline{j}})\right)\\
        =&-\lambda\cdot\sum_{j=0}^{c-1}\deg(\FF_{\overline{j}})\\
        =&-\lambda\cdot\deg(\FF) 
    \end{aligned}\]
and
    \[\begin{aligned}
        B(\mu(\EE)\cdot\operatorname{id}_{\mathbb{C}^r},s)
        =&\mu(\EE)\cdot\sum_{j=0}^{c-1}\left(-\lambda\cdot\rank(\FF_{\overline{j}})+0\cdot\rank(\EE_{\overline{j}}/\FF_{\overline{j}})\right)\\
        =&-\lambda\cdot\mu(\EE)\cdot\rank(\FF).
    \end{aligned}\]
we obtain that
\[\deg(\FF)\leqslant\mu(\EE)\cdot\rank(\FF)\]
due to $\lambda>0$. Therefore, $\mu(\FF)\leqslant\mu(\EE)$ since $\rank(\FF)>0$.

Conversely, for (\romannumeral2) $\implies$ (\romannumeral1), let $\sigma$ be the holomorphic reduction from $G_0$ to $P_s$ for some $s\in\iu\mathfrak{k}_0$ such that $\Phi\in\mathrm{H}^0(\mathbb{E}_\sigma[(\mathfrak{g}_1)_s]\otimes\KK_X)$. Note that $s$ can be diagonalized with real eigenvalues. We rearrange them as 
$
\lambda_1<\lambda_2<\cdots<\lambda_{t}$ and assume that  $\lambda_{t+1}=0$. Define $U^k=\operatorname{ker}(\lambda_k\operatorname{id}_{\mathbb{C}^r}-s)$ and $W^k:=\bigoplus_{l=1}^k U^l$. We obtain holomorphic vector bundles $\mathcal{W}^j=\mathbb{E}_\sigma[W^j]$ and this gives the filtration
\[
    0=\mathcal{W}^0\subset\mathcal{W}^1\subset\cdots\mathcal{W}^{t}=\EE.
\]
Now for every $k$, $\mathcal{W}^k$ is a graded subbundle of $\mathcal{E}$ since it is the associated bundle of $\mathbb{E}$, and $\mathcal{W}^k$ is $\Phi$-invariant since $\Phi\in\mathrm{H}^0(\mathbb{E}_\sigma[(\mathfrak{g}_1)_s]\otimes\KK_X)$. Thus $\mu(\WW^k)\leqslant\mu(\EE)$. Therefore, following from \cite[Lemma 2.12]{garcia2009hitchin}, we obtain that 
\begin{equation}\label{eq:degfiltration}
    \begin{aligned}
    &\deg\mathbb{E}(\sigma,s)-B(\mu(\EE)\cdot\operatorname{id}_{\mathbb{C}^r},s)\\
    =&\sum_{k=1}^{t}(\lambda_{k}-\lambda_{k+1})\cdot\left(\deg(\mathcal{W}^{k})-\mu(\EE)\cdot\rank(\WW^k)\right)\\
    =&\sum_{k=1}^{t-1}(\lambda_{k}-\lambda_{k+1})\cdot\rank(\WW^k)\cdot\left(\mu(\mathcal{W}^{k})-\mu(\EE)\right)\geqslant0
\end{aligned}
\end{equation}
since $\lambda_k<\lambda_{k+1}$ when $1\leqslant k\leqslant t-1$, $\rank(\WW^k)>0$, $\mu(\WW^k)\leqslant\mu(\EE)$ and $\mu(\WW^t)=\mu(\EE)$.

(\romannumeral1') $\iff$ (\romannumeral2') is similar to (\romannumeral1) $\iff$ (\romannumeral2). The only difference is that if $\FF\neq\EE$, then the $s$ we defined as above could not lie in $\iu\operatorname{ker}\dd\iota|_{\mathfrak{k}_0}$.

Finally we consider (\romannumeral1'') $\iff$ (\romannumeral2''). We first assume (\romannumeral1'') holds. From the proof of (\romannumeral1) $\implies$ (\romannumeral2), we obtain that if there exists a $\Phi$-invariant graded subbundle $\FF\subset\EE$ with $\FF\neq0$, we can define an element $s\in\iu\mathfrak{k}$ and a reduction $\sigma\in\mathbb{E}[G/P_s]$ such that $\Phi\in\mathrm{H}^0(\mathbb{E}_\sigma[(\mathfrak{g}_1)_s]\otimes\KK_X)$ and
\[\deg\mathbb{E}(\sigma,s)=-\lambda\cdot\deg(\FF),\quad B(\mu(\EE)\cdot\operatorname{id}_{\mathbb{C}^r},s)=-\lambda\cdot\mu(\EE)\cdot\rank(\FF)\]
for a given positive number $\lambda>0$. Therefore, $\mu(\FF)=\mu(\EE)$ implies that
\[\deg\mathbb{E}(\sigma,s)=B(\mu(\EE)\cdot\operatorname{id}_{\mathbb{C}^r},s).\]
By polystability, there exists a further reduction $\sigma'\in\mathrm{H}^0(\mathbb{E}_\sigma[P_s/L_s])$ such that $\Phi\in\mathrm{H}^0(\mathbb{E}_{\sigma'}[V_s^0]\otimes\KK_X)$. Hence the Levi reduction $\mathbb{E}_{\sigma'}$ defines the decomposition $\EE_{\overline{j}}=\mathbb{E}_{\sigma'}[V_j]=\FF_{\overline{j}}\oplus\mathcal{S}_{\overline{j}}$ and 
$\Phi\in\mathrm{H}^0(\mathbb{E}_{\sigma'}[V_s^0]\otimes\KK_X)$
implies that $\mathcal{S}:=\bigoplus_{j=0}^{m-1}\mathcal{S}_{\overline{j}}$ is $\Phi$-invariant. Also the decomposition implies that $\mu(\mathcal{S})=\mu(\EE)$.

Conversely, when assuming (\romannumeral2''), let $\sigma$ be the holomorphic reduction from $G_0$ to $P_s$ for some $s\in\iu\mathfrak{k}_0$ such that $\Phi\in\mathrm{H}^0(\mathbb{E}_\sigma[(\mathfrak{g}_1)_s]\otimes\KK_X)$ and
\[\deg\mathbb{E}(\sigma,s)=B(\mu(\EE)\cdot\operatorname{id}_{\mathbb{C}^r},s).\]Hence, from \prettyref{eq:degfiltration} we obtain that $\mu(\WW^k)=\mu(\EE)$ for every $k$. This gives the direct sum decomposition $(\EE,\Phi)=(\WW^k,\Phi|_{\WW^k})\oplus(\mathcal{R}^k,\Phi|_{\mathcal{R}^k})$ where $\mathcal{R}^k$'s are $\Phi$-invariant graded subbundles with $\mu(\mathcal{R}^k)=\mu(\EE)$. Then since $\WW^k\subset\WW^{k+1}$, we obtain that $\WW^{k+1}=\WW^k\oplus(\WW^{k+1}\cap\mathcal{R}^k)$. And $\mathcal{U}^{k+1}:=\WW^{k+1}\cap\mathcal{R}^k$ is a holomorphic $\Phi$-invariant graded subbundle with $\mu(\WW^{k+1}\cap\mathcal{R}^k)=\mu(\EE)$. This gives the direct sum decomposition 
\[(\WW^{k},\Phi|_{\WW^k})=\bigoplus_{l=1}^k(\mathcal{U}^{l},\Phi|_{\mathcal{U}^{l}})\]
which induces the Levi reduction $\sigma'\in\mathrm{H}^0(X,\mathbb{E}_\sigma[P_s/L_s])$ such that $\Phi\in\mathrm{H}^0(\mathbb{E}_{\sigma'}[(\mathfrak{g}_1)_s^0]\otimes\KK_X)$.
\end{proof}

\subsection{The moduli space and its local structure}

We consider the moduli space $\mathcal{M}(G_0,\mathfrak{g}_{\overline{1}})$ of $\mu$-polystable $(G_0,\mathfrak{g}_{\overline{1}})$-Higgs pairs and its stable part $\mathcal{M}^{\mathrm{s}}(G_0,\mathfrak{g}_{\overline{1}})$ . By \prettyref{prop:equivstability}, it can be also identified with the moduli space of polystable cyclic Higgs bundles of rank $(r_j)$. The construction of this moduli space of Higgs pairs given by Schmitt \cite{schmitt2005moduli,schmitt2008geometric} using Geometric Invariant Theory applies to this moduli space, hence it has the structure of a complex analytic variety.

To study the local structure of $\mathcal{M}(G_0,\mathfrak{g}_{\overline{1}})$ at a point, we recall some deformation theory for Higgs pairs, for more details see \cite{biswas1994infinitesimal} and \cite{garcia2009hitchin}. We consider the deformation complex of a $(G_0,\mathfrak{g}_{\overline{1}})$-Higgs pair $(\mathbb{E},\Phi)$, which is the following complex of sheaves:
\[C^\bullet(\mathbb{E},\Phi)\colon\mathbb{E}[\mathfrak{g}_0]\stackrel{\dd\iota(\Phi)}{\longrightarrow}\mathbb{E}[\mathfrak{g}_{\overline{1}}]\otimes\KK_X.\]
The deformation complex governs infinitesimal deformations of $(\mathbb{E},\Phi)$ and it induces a natural long exact sequence
\begin{equation}\label{eq:LES}    
\begin{aligned}
    0&\longrightarrow\mathbb{H}^0(C^\bullet(\mathbb{E},\Phi))\longrightarrow\mathrm{H}^0(\mathbb{E}[\mathfrak{g}_0])\stackrel{\dd\iota(\Phi)}{\longrightarrow}\mathrm{H}^0(\mathbb{E}[\mathfrak{g}_{\overline{1}}]\otimes\KK_X)\\
    &\longrightarrow\mathbb{H}^1(C^\bullet(\mathbb{E},\Phi))\longrightarrow\mathrm{H}^1(\mathbb{E}[\mathfrak{g}_0])\stackrel{\dd\iota(\Phi)}{\longrightarrow}\mathrm{H}^1(\mathbb{E}[\mathfrak{g}_{\overline{1}}]\otimes\KK_X)\longrightarrow\mathbb{H}^2(C^\bullet(\mathbb{E},\Phi)),
\end{aligned}
\end{equation}
where $\mathbb{H}^i$ denotes the $i$-th hypercohomology group. By definition, 
\[\mathbb{H}^0(C^\bullet(\mathbb{E},\Phi))=\{s\in\mathrm{H}^0(\mathbb{E}[\mathfrak{g}_0])\mid[s\wedge\Phi]=0\},\]
which is the Lie algebra of the automorphism group of $(\mathbb{E},\Phi)$. We call it the infinitesimal automorphism space $\operatorname{aut}(\mathbb{E},\Phi)$. We have the following standard Schur's lemma:
\begin{proposition}\label{prop:Schur}
    If two cyclic Higgs bundles $(\EE,\Phi)$ and $(\EE',\Phi')$ are both stable as cyclic Higgs bundles and $\mu(\EE)=\mu(\EE')$, then for any bundle map $f\colon\EE\to\EE'$ satisfying that
    \begin{itemize}
        \item there exists $k$ such that $f(\EE_{\overline{j}})\subset \EE_{\overline{j+k}}'$ for every $j$;

        \item $\Phi'\circ f=(f\otimes\operatorname{id}_{\KK_X})\circ\Phi$,
    \end{itemize}
    $f$ is either $0$ or an isomorphism. Moreover, $\operatorname{aut}(\mathbb{E},\Phi)=\mathbb{C}\cdot\operatorname{id}_\EE$.
\end{proposition}

\begin{proof}

    We consider the following short exact sequence
    \[0\to\operatorname{ker}(f)\to\EE\to\operatorname{im}(f)\to 0.\] Since $\Phi'\circ f=(f\otimes\operatorname{id}_{\KK_X})\circ\Phi$, $\operatorname{ker}(f)$ and $\operatorname{im}(f)$ are $\Phi$-invariant subsheaves of $\EE$ and $\EE'$ respectively. Thus their saturation $\operatorname{Sat}(\operatorname{ker}(f))$ and $\operatorname{Sat}(\operatorname{im}(f))$ must be $\Phi$-invariant and graded subbundles of $\EE$ and $\EE'$ respectively since $f(\EE_{\overline{j}})\subset\EE_{\overline{j+k}}'$. Therefore, by stability we obtain that if both $\operatorname{Sat}(\operatorname{ker}(f))$ and $\operatorname{Sat}(\operatorname{im}(f))$ are neither $0$ nor $\EE$, $\EE'$, then
    \[\mu(\operatorname{ker}(f))\leqslant\mu(\operatorname{Sat}(\operatorname{ker}(f)))<\mu(\EE)\]
    and 
    \[\mu(\operatorname{im}(f))\leqslant\mu(\operatorname{Sat}(\operatorname{im}(f)))<\mu(\EE').\]
    This contradicts to the short exact sequence. Therefore, $f$ must be either $0$ or an isomorphism.

    Now let $s\in\operatorname{aut}(\mathbb{E},\Phi)$. The characteristic polynomial of $s$ has holomorphic coefficients on the compact Riemann surface $X$, hence constant coefficients. Choose an eigenvalue $\lambda\in\mathbb{C}$, then by the above discussion we obtain that $s-\lambda\cdot\operatorname{id}_\EE$ must be $0$.
\end{proof}

As corollaries, we obtain that

\begin{corollary}\label{coro:automorphism}
    If a cyclic Higgs bundle $(\EE,\Phi)$ is stable, then  $\operatorname{Aut}(\mathbb{E},\Phi)=\operatorname{Aut}(\EE,\Phi)=\mathbb{C}^*\cdot\operatorname{id}_\EE$.
\end{corollary}

\begin{corollary}\label{coro:infautpoly}
    Let \((\EE,\Phi)\) be a polystable cyclic Higgs bundle and it decomposes as
\[
(\EE,\Phi)=\bigoplus_{\alpha} (\FF^\alpha,\Phi|_{\FF^\alpha})\otimes V_\alpha,
\]
where the $(\FF^\alpha,\Phi|_{\FF^\alpha})$ are pairwise non-isomorphic stable cyclic Higgs bundles of the same slope, and $V_\alpha\cong \mathbb C^{m_\alpha})$ are multiplicity spaces. Then
\[
\operatorname{aut}(\EE,\Phi)=\bigoplus_\alpha \operatorname{id}_{\FF^\alpha}\otimes\mathfrak{gl}(V_\alpha).
\]

\end{corollary}

Note that the automorphism group $\operatorname{Aut}(\mathbb{E},\Phi)$ acts on $\mathbb{H}^1(C^\bullet(\mathbb{E},\Phi))$. Using standard slice methods of Kuranishi (see \cite[Chapter 7.3]{kobayashi2014differential} for details for the moduli space of holomorphic bundles and \cite{fan2021analytic} for that of Higgs bundles, which also works for that of general Higgs pairs), a neighborhood of the isomorphism class of a $\mu$-polystable $(G_0,\mathfrak{g}_1)$-Higgs pair $(\mathbb{E},\Phi)$ in the moduli space is given by 
\[\kappa^{-1}(0)\sslash\operatorname{Aut}(\mathbb{E},\Phi),\]
where 
$\kappa\colon \mathbb{H}^1(C^\bullet(\mathbb{E},\Phi))\to\mathbb{H}^2(C^\bullet(\mathbb{E},\Phi))$
is the so called Kuranishi map. When $\mathbb{H}^2(C^\bullet(\mathbb{E},\Phi))=0$, a neighborhood of the isomorphism class of $(\mathbb{E},\Phi)$ in the moduli space is isomorphic to
\[\mathbb{H}^1(C^\bullet(\mathbb{E},\Phi))\sslash\operatorname{Aut}(\mathbb{E},\Phi).\]

For the usual $G$-Higgs bundle for a complex reductive Lie group $G$, we have the duality between the second hypercohomology group and the zeroth one of its deformation complex, see e.g. \cite[Proposition 3.13]{garcia2009hitchin}. There is a slight difference for $(G_0,\mathfrak{g}_{\overline{1}})$-Higgs pairs. We give the following definition:

\begin{definition}
    The infinitesimal automorphism space of $(\mathbb{E},\Phi)$ of weight $\overline{j}$ is defined as
    \[\operatorname{aut}_{\overline{j}}(\EE,\Phi):=\{s\in\mathbb{E}[\mathfrak{g}_{\overline{j}}]\mid[s\wedge\Phi]=0\}.\]
\end{definition}

\begin{proposition}\label{prop:dual}
    There is an isomorphism between $\mathbb{H}^2(C^\bullet(\mathbb{E},\Phi))^\vee$ and $\operatorname{aut}_{-\overline{1}}(\EE,\Phi)$.
\end{proposition}

\begin{proof}
    Since there is a invariant non-degenerate bilinear form on $\mathfrak{g}$ perfectly pairing $\mathfrak{g}_{\overline{1}}$ and $\mathfrak{g}_{-\overline{1}}$, the dual complex of $C^\bullet(\mathbb{E},\Phi)$ is:
\[C^\bullet(\mathbb{E},\Phi)^\vee\colon\mathbb{E}[\mathfrak{g}_{-\overline{1}}]\otimes\KK_X^{-1}\stackrel{\operatorname{ad}_\Phi}{\longrightarrow}\mathbb{E}[\mathfrak{g}_{0}].\]
By Serre duality, 
\[\mathbb{H}^2(C^\bullet(\mathbb{E},\Phi))^\vee\cong\mathbb{H}^0(C^\bullet(\mathbb{E},\Phi)^\vee\otimes\KK_X)=\operatorname{aut}_{-\overline{1}}(\EE,\Phi).\]
\end{proof}

\subsection{Finiteness of the map between moduli spaces}

Following from \cite[Proposition 5.7]{garcia2019involution}, there is a natural map from $\mathcal{M}(G_0,\mathfrak{g}_{\overline{1}})$ to the moduli space $\mathcal{M}(G)$ of polystable $G$-Higgs bundles. It is proven in \cite[Proposition 5.9]{garcia2019involution} that this map has finite fibre when $G_0$ is a real form of $G$. Below we give a proof for general cyclic Higgs bundles when $G=\mathrm{GL}_r\mathbb{C}$.

\begin{proposition}\label{prop:finite}
    The natural map $F\colon\mathcal{M}(G_0,\mathfrak{g}_{\overline{1}})\to\mathcal{M}(G)$ has finite fibre.
\end{proposition}

\begin{proof}
    Fix a primitive $c$-th root of unity $\zeta$. Then to obtain an $\mathbb{Z}/c\mathbb{Z}$-grading on a polystable $G$-Higgs bundle $(\EE,\Phi)$ is equivalent to choose a holomorphic bundle automorphism $T\in\operatorname{Aut}(\EE)$ such that $T^c=\operatorname{\id}_{\EE}$ and $T\Phi=\zeta\Phi T$. Moreover, two such gradings are isomorphic if and only if their grading operators $T,T'$ satisfy that $T'=\operatorname{Ad}(g)T$ for some $g\in\operatorname{Aut}(\EE,\Phi)$ and the automorphism of the cyclic Higgs bundle determined by $T$ is the stablizer of $T$ in $\operatorname{Aut}(\EE,\Phi)$. 
Thus it suffices to prove that such grading operators only have finitely many conjugacy classes under $\operatorname{Aut}(\EE,\Phi)$. Note that $(\EE,\Phi)$ decomposes as
\[\bigoplus_{\alpha\in A}
(\FF_\alpha,\Phi|_{\FF_\alpha})\otimes V_\alpha
\]
where the $(F_\alpha,\Phi|_{\FF_\alpha})$ are pairwise non-isomorphic stable Higgs bundles of slope $\mu(\EE)$. \prettyref{coro:infautpoly} gives
\[
\operatorname{Aut}(\EE,\Phi)
\cong
\prod_{\alpha\in A}\mathrm{GL}(V_\alpha).
\]
It follows that there is a uniquely determined permutation $\sigma\colon:A\longrightarrow A$ such that $(\FF_{\sigma(\alpha)},\Phi|_{\FF_{\sigma(\alpha)}})
\cong(\FF_\alpha,\zeta^{-1}\Phi|_{\FF_\alpha})$
and every $T$ sends
$\FF_\alpha\otimes V_\alpha$ onto $\FF_{\sigma(\alpha)}\otimes V_{\sigma(\alpha)}$.
In particular, the permutation does not depend on the choice of $T$.
Since $\zeta^c=1$, we have $\sigma^c=\operatorname{id}_A$. Thus each orbit has length \(\ell\) dividing \(c\). Since \(T\) is an isomorphism, the multiplicities are constant along each orbit.
Fix one orbit, and relabel its stable factors as
\[
(\FF_0,\Phi|_{\FF_0}),\ldots,(\FF_{\ell-1},\Phi|_{\FF_{\ell-1}}),
\]
with indices understood modulo $\ell$. Write $m$ for their common multiplicity and identify their multiplicity spaces with $\mathbb{C}^m$.
For each \(j\), we have a holomorphic bundle isomorphism $u_j\colon\FF_j\to \FF_{j+1}$ such that $u_j\Phi|_{\FF_j}=\zeta\Phi|_{\FF_{j+1}}u_j$ given by $T$. Set $u:=u_{\ell-1}\circ\cdots \circ u_0\colon\FF_0\to \FF_0$. Then $u\Phi|_{\FF_0}=\zeta^\ell\Phi|_{\FF_0}u$. Therefore,
\[
u^{c/\ell}\Phi|_{\FF_0}=\Phi|_{\FF_0}u^{c/\ell}.
\]
By Schur’s lemma \prettyref{prop:Schur},
\[
u^{c/\ell}=\lambda\operatorname{id}_{\FF_0}
\]
for some $\lambda\in\mathbb C^*$. Up to suitable scaling, we may assume that $u^{c/\ell}=\operatorname{id}_{\FF_0}$.
Again by Schur’s lemma \prettyref{prop:Schur}, every possible $T$ on this orbit has blocks
\[
T_j=u_j\otimes A_j,
\qquad A_j\in \mathrm{GL}_m\mathbb C.
\]

Now an automorphism $g\in\operatorname{Aut}(\EE,\Phi)$ with blocks $\operatorname{id}_{\FF_j}\otimes C_j$, transforms these matrices by
\[
A_j\longmapsto C_{j+1}A_jC_j^{-1}.
\]
Choose
\[
C_0=I_m,\qquad
C_{j+1}=C_jA_j^{-1}
\]
for $0\leqslant j<\ell-1$. After this conjugation, the matrix data have the form
\[
(I_m,\dots,I_m,B),
\qquad
B=A_{\ell-1}\cdots A_0. 
\]
Now two grading operators $T,T'$ in this form are conjugate if and only if their matrices $B,B'$ are conjugate in $\mathrm{GL}_m\mathbb{C}$. Indeed, preservation of the first $\ell-1$ identity blocks forces
\[
C_0=C_1=\cdots=C_{\ell-1},
\]
and the last block then transforms by $B'=C_0BC_0^{-1}$.
On the other hand, we also have the condition $B^{c/\ell}=I_m$. 

Hence the conjugacy classes of such grading operators on this orbit are in bijection with
\[
\{B\in \mathrm{GL}_m\mathbb C:B^{c/\ell}=I_m\}
\big/\operatorname{Ad}(\mathrm{GL}_m\mathbb{C}).
\]
This set is finite. Indeed, the polynomial $x^{c/\ell}-1$ has distinct roots, so every such matrix $B$ is diagonalizable. If $\omega$ is a primitive $(c/\ell)$-th root of unity, every such matrix is conjugate to exactly one matrix
\[
\operatorname{diag}
(I_{a_0},\omega I_{a_1},\ldots,
      \omega^{c/\ell-1}I_{a_{c/\ell-1}}),
\]
where $a_0,\ldots,a_{c/\ell-1}\geqslant0$, and $a_0+\cdots+a_{c/\ell-1}=m$.
There are
\[
\binom{m+c/\ell-1}{c/\ell-1}
\]
possible multiplicity tuples.
Finally, the different $\sigma$-orbits are independent, thus the classifications for the individual orbits combine by taking their product. Therefore, the natural map $\mathcal{M}(G_0,\mathfrak{g}_{\overline{1}})\to\mathcal{M}(G)$ has finite fibre.
\end{proof}

\begin{remark}
    Following the analytic proof of the properness of Hitchin fibration, e.g. \cite[Theorem 4.2.15]{wentworth2016higgs}, one could show that the natural map $\mathcal{M}(G_0,\mathfrak{g}_{\overline{1}})\to\mathcal{M}(G)$ is proper. Hence it is a proper quasi-finite morphism on every connected component, which is of finite type. Therefore, the natural map $\mathcal{M}(G_0,\mathfrak{g}_{\overline{1}})\to\mathcal{M}(G)$ is actually a finite morphism.
\end{remark}

\section{Calculation of the set \texorpdfstring{$\mathscr D_\tau^{\mathrm{s}}$}{Dtaus}}
Recall the notations in \prettyref{sec:stability}. Let $N_{\mathrm{s},\tau}$ (resp. $N_\tau$) denote the number of the set of degree vectors in $\mathscr D_\tau^{\mathrm{s}}$ (resp. $\mathscr D_\tau$). Let $N_{\mathrm{s},\tau}(\mu)$ (resp. $N_\tau(\mu)$) denote the number of the subset of degree vectors in $\mathscr D_\tau^{\mathrm{s}}$ (resp. $\mathscr D_\tau$) satisfying $d_0+d_1+d_2+d_3=(q+3)\mu$.  By \prettyref{lem:DualDegree}, $N_{\mathrm{s},\upRoman{1}}(\mu)=N_{\mathrm{s},\upRoman{2}}(-\mu)$, $N_{\upRoman{1}}(\mu)=N_{\upRoman{2}}(-\mu)$, and $N_{\mathrm{s},\upRoman{1}}=N_{\mathrm{s},\upRoman{2}}$, $N_{\upRoman{1}}=N_{\upRoman{2}}$.

\begin{lemma}\label{lem:calculationDs}
For $q>1$ and for each $\mu=k/(q+3)$ for $k=0,1,\cdots, q+2$, we have
\[
N_{\mathrm{s},\upRoman{1}}(\mu)=8(g-1)^2+
\scalebox{0.85}{$\begin{cases}
-2(g-1),
&
\mu=0
\ \text{or}\
\mu=\dfrac12,
\\[2mm]
0,
&
0<\mu\leqslant\dfrac14
\ \text{or}\
\dfrac12<\mu\leqslant\dfrac34,
\\[3mm]
2(g-1),
&
\dfrac14<\mu<\dfrac12
\ \text{or}\
\dfrac34<\mu<1.
\end{cases}$}
\]
\[
N_{\upRoman{1}}(\mu)=8(g-1)^2+\scalebox{0.85}{$
\begin{cases}
4(g-1)+1,
&
\mu=0,
\\[2mm]
0,
&
0<\mu<\dfrac14
\ \text{or}\
\dfrac12<\mu<\dfrac34,
\\[2mm]
2(g-1),
&
\dfrac14\leqslant\mu<\dfrac12
\ \text{or}\
\dfrac34\leqslant\mu\leqslant1,
\\[3mm]
4(g-1),
&
\mu=\dfrac12.
\end{cases}$}
\]
The total number is
\[
N_{\mathrm{s},\tau}
=
8(q+3)(g-1)^2
+
(g-1)\scalebox{0.85}{$
\begin{cases}
q,
&
q\ \text{even},
\\
q-5,
&
q\equiv 1\pmod 4,
\\
q-3,
&
q\equiv 3\pmod 4.
\end{cases}$}
\]
\[
N_{\tau}
=
8(q+3)(g-1)^2+1
+
(g-1)\scalebox{0.85}{$
\begin{cases}
q+6,
&
q\ \text{even},
\\
q+11,
&
q\equiv 1\pmod 4,
\\
q+9,
&
q\equiv 3\pmod 4.
\end{cases}$}
\]
\end{lemma}
\begin{proof}
Since $d_1=d_0-2(g-1)$, we have
$d_3=d_0-2(g-1)-u_\upRoman{1}.$
The sum condition then gives
\[
d_2
=
(q+3)\mu-d_0-d_1-d_3
=
(q+3)\mu-3d_0+4(g-1)+u_\upRoman{1}.
\]
Consequently,
\begin{align*}
v_\upRoman{1}=d_0-d_2+(q-1)\mu
&=
4d_0-4(g-1)-u_\upRoman{1}-4\mu,
\end{align*}
Define
\[
h=d_0-(g-1),
\qquad
w=u_\upRoman{1}+4(g-1).
\]
The inequalities in defining $(d_i)\in \mathscr D_\upRoman{1}^{\mathrm{s}}$ or $\mathscr D_\upRoman{1}$ become either 
\[
-2(g-1)+\frac{w}{2}+\mu
<
h
<
\frac{w}{4}+\mu,\quad w\geqslant0
\] or
\[
-2(g-1)+\frac{w}{2}+\mu
\leqslant
h
\leqslant
\frac{w}{4}+\mu,\quad w\geqslant0
\]
respectively. Since \(h\in\mathbb Z\), such an integer can occur only when
\[
0\leqslant w<8(g-1)
\qquad\text{or}\qquad
0\leqslant w\leqslant 8(g-1),
\]
respectively.
respectively.
Writing
\[
w=4r+s,
\qquad
s\in\{0,1,2,3\},
\qquad
0\leqslant r\leqslant 2(g-1),
\]
and summing the number of possible integers \(h\). Note that with fixed $\mu$, the pair $(h,w)$ is one-to-one corresponding to the degree vector $(d_j)$, one obtains the values of $N_{\mathrm{s},\upRoman{1}}(\mu)$, $N_{\upRoman{1}}(\mu)$. 
\end{proof}


\bibliographystyle{alpha}
\bibliography{bib}
\end{document}